\documentclass[11pt,a4paper]{article}

\usepackage{enumitem}

\usepackage[margin=2.5cm]{geometry}
\usepackage[flushmargin]{footmisc}
\usepackage{amsfonts}
\usepackage{amsthm}
\usepackage{amsmath}
\usepackage{amssymb}
\usepackage{centernot}
\usepackage[T1]{fontenc}
\usepackage{cite}
\usepackage{mathrsfs}
\usepackage{amscd}
\usepackage[utf8]{inputenc}
\usepackage{t1enc}
\usepackage{lmodern}
\usepackage{dsfont}
\usepackage{xcolor}
\usepackage{float}
\usepackage{hyperref}
\hypersetup{colorlinks=true, urlcolor= black, linkcolor=black, citecolor=blue}
\usepackage[mathscr]{eucal}
\usepackage{indentfirst}
\renewcommand{\epsilon}{\varepsilon}
\usepackage{bbm}
\usepackage{graphicx}
\usepackage{mathtools}

\numberwithin{equation}{section}

\theoremstyle{plain}
\newtheorem{Thm}{Theorem}[section]
\newtheorem{Lem}[Thm]{Lemma}
\newtheorem{Claim}[Thm]{Claim}
\newtheorem{Coro}[Thm]{Corollary}
\newtheorem{Prop}[Thm]{Proposition}

\newtheorem{Conj}{Conjecture}

\theoremstyle{definition}
\newtheorem{Def}[Thm]{Definition}

\newtheorem{Rem}[Thm]{Remark}

\newtheorem{Ex}[Thm]{Example}

\definecolor{darkgreen}{rgb}{0,0.6,0.05}

\newcommand{\connect}{\xleftrightarrow}
\newcommand\fg{\mathfrak{g}}

\title{Supercritical sharpness for the random cluster representation of real-valued spin models}
\date{}

\begin{document}

\author{Trishen S. Gunaratnam\footnote{TIFR Mumbai and ICTS Bengaluru, \url{trishen@math.tifr.res.in, trishen.gunaratnam@icts.res.in}}\:,  Dmitry Krachun\footnote{Princeton University, \url{dk9781@princeton.edu}}\:, Christoforos Panagiotis\footnote{University of Bath, \url{cp2324@bath.ac.uk}}\:, \\ Romain Panis\footnote{Université Lyon 1, 
\url{panis@math.univ-lyon1.fr}}\:, Franco Severo\footnote{CNRS and Sorbonne Université, \url{severo@lpsm.paris}}}

\maketitle

\begin{abstract}
We study the supercritical regime $\beta>\beta_c$ of a large family of real-valued spin models on $\mathbb Z^d$, including the Blume--Capel and $P(\varphi)$ models. We consider their random cluster representations and prove that they are well-behaved, in the sense that local uniqueness of macroscopic clusters occurs with high probability, uniformly in the boundary conditions. This implies, among other things, a surface-order exponential bound for the (lower) large deviations of the empirical magnetisation. These results were previously known only in the cases of the Ising and $\varphi^4$ models.
\end{abstract}

\setcounter{tocdepth}{2}
    
\section{Introduction}\label{sec:intro}

This article concerns real-valued spin systems with $\mathbb Z/2\mathbb Z$ symmetry. The study of these models originates in the 1920s with the seminal works of Lenz \cite{Lenz1920beitrag} and Ising \cite{Ising1924}, who introduced and analysed what is now known as the Ising model of ferromagnetism. Over the following decades, this discrete model was extended in several directions. Notable examples include the Blume--Capel model \cite{BL,Capel}, which incorporates an additional vacancy state; and continuous-spin models such as the $\varphi^4$ model and, more generally, the $P(\varphi)$ models, developed in the context of Euclidean field theory and the renormalization group \cite{GlimmJaffeQuantumBOOK}. Despite their shared symmetry, these systems exhibit a rich variety of critical phenomena: depending on the specific choice of single-site potential and interaction, one observes different universality classes, as well as the possibility of first-order phase transitions. Their behaviour away from criticality, on the other hand, is expected to display a striking degree of universality. 
In particular, truncated correlations of the extremal measures are believed to decay exponentially in the distance at all temperatures except the (unique) critical one. This dichotomy between model-dependent critical behaviour and robust off-critical properties motivates a unified analysis of high- and low-temperature regimes across this family of spin systems. 
We stress, however, that exponential decay of truncated correlations does not necessarily hold for spin models beyond those considered here. In particular, it can fail for models with continuous symmetries; see, e.g., \cite{DarioWu24Villain} and references therein.

There exist many powerful tools in statistical physics that allow for a detailed description of models in their perturbative regimes, namely at very high or very low temperatures. Examples include high- and low-temperature expansions, cluster expansions, and Pirogov--Sinai theory. While these methods typically cease to be effective in the vicinity of a critical point, it is generally expected that the qualitative behaviours they reveal persist throughout the corresponding phase, arbitrarily close to criticality. This phenomenon is often referred to as a \emph{sharp phase transition}. Establishing such sharpness results is, however, a much more challenging task, and the range of available techniques is considerably more limited.

Exponential decay of correlations above the critical temperature was established for the Ising and $\varphi^4$ models by Aizenman, Barsky, and Fernández in the 1980s \cite{AizenmanBarskyFernandezSharpnessIsing1987} (see also \cite{DuminilTassionNewProofSharpness2016} for an alternative proof in the case of the Ising model). In recent years, major advances in the study of subcritical percolation models—most notably the work of Duminil-Copin, Raoufi, and Tassion \cite{DCRT19}—have led to analogous high-temperature results for a broad class of spin systems, including all the real-valued models considered here, as proved very recently in \cite{PanaVeisubcrit}. In contrast, the low-temperature behaviour of spin systems
(and, correspondingly, the supercritical regime of percolation models) remains much less understood.
Within the family of models under consideration, exponential decay of \emph{truncated correlations} below the critical temperature has so far been established (in all dimensions $d\geq 2$) only for the Ising model, in a relatively recent paper by Duminil-Copin, Goswami, and Raoufi \cite{DCGR20} (see also \cite{GPPS26} for a simpler proof). A key ingredient in their proof is a \emph{supercritical sharpness} result for the associated \emph{random cluster} representation (or \emph{Fortuin--Kasteleyn} (FK) representation), which was established by Bodineau for the Ising model \cite{Bod05} (see also \cite{SeveroSlabIsing2023} for a streamlined argument), and also more recently for the $\varphi^4$ model \cite{gunaratnam2025supercritical}.
However, these papers rely on either the Lee--Yang theorem or the so-called random (tangled) current representation, which are specific to the Ising and $\varphi^4$ models and not available at our level of generality.  

In the present article, we establish supercritical sharpness for the random cluster representation of general real-valued spin systems. Our approach requires little beyond one common feature of these models: the validity of the so-called \emph{Lebowitz inequality}\footnote{This inequality should not be confused with the other Lebowitz inequality on Ursell's four-point function, see \cite{Lebowitz1974Inequ}.} \cite{Lebowitz1977CoexistencePhasesIsing,lebowitz1978number}. The main innovation of the paper is an almost everywhere 
uniqueness result for half-space measures with a non-zero magnetic field on the 
boundary. The proof relies on a soft percolation argument inspired by the 
work of Burton and Keane~\cite{BurtonKeane1989density}, bypassing the need for the aforementioned Ising/$\varphi^4$-specific tools. As in the case of 
the $\varphi^4$ model~\cite{gunaratnam2025supercritical}, several 
technical difficulties arise in the handling of unbounded spins, some of 
which are treated here in a new and simpler way. To isolate the key 
conceptual ideas, a more streamlined argument is presented in a forthcoming 
paper \cite{GKPPS26} in the specific setting of the Blume--Capel model, where spins take 
values in $\{-1,0,1\}$.

\subsection{Definitions and main result}

We start by defining precisely the class of spin models considered here. They will be indexed by a family of measures on $\mathbb R$, which motivates the following definitions.

\begin{Def}[Admissible single-site measures]\label{def:single site} We say that a Borel measure $\rho$ on $\mathbb R$ is \emph{admissible} if it satisfies the following properties:
\begin{enumerate}
    \item[$(i)$] $\rho$ is \emph{even}, which means that for every Borel measurable set $A\subset \mathbb R$, $\rho(A)=\rho(-A)$. 
    \item[$(ii)$] $\rho$ is \emph{super-Gaussian}, in the sense that, for every $a>0$, one has
    \begin{equation}
        \int_{\mathbb R}e^{au^2}\mathrm{d}\rho(u)<\infty.
    \end{equation}
    \item[$(iii)$] $\rho$ has a non-trivial support, i.e.\ $\rho(\mathbb{R}\setminus\{0\})>0$.
\end{enumerate}
\end{Def}

To define the spin models, we need some graph notation. For $x,y\in \mathbb Z^d$, we write $x\sim y$ if $\Vert x-y\Vert_1=1$, where $\Vert \cdot \Vert_1$ denotes the $\ell^1$ norm on $\mathbb R^d$. Given a subset of vertices $\Lambda\subset \mathbb{Z}^d$, consider its external boundary and closure given by $\partial^{\textup{ext}}\Lambda = \{y\in\mathbb{Z}^d\setminus \Lambda : \exists \, x\in\Lambda, x\sim y\}$ and $\overline{\Lambda} = \Lambda\cup \partial^{\textup{ext}}\Lambda$, respectively. Denote by $E(\Lambda):= \{ xy : x,y \in \Lambda,\, x\sim y\}$ the induced set of edges, and by $\overline E(\Lambda) := \{ xy : x \in \Lambda, \, y \in \mathbb Z^d,\, x\sim y \}$ the edges touching $\Lambda$. The edge boundary of $\Lambda$ is $\partial_E \Lambda:= \overline E(\Lambda) \setminus E(\Lambda)$.

\begin{Def}[The spin model]\label{def:spin models} Let $\rho$ be an admissible measure on $\mathbb{R}$. The corresponding real-valued spin model on a finite subset $\Lambda\subset \mathbb Z^d$, with coupling constants $J=\{J_{xy}\}_{xy\in \overline E(\Lambda)}\in (\mathbb{R}^+)^{\overline E(\Lambda)}$, external magnetic field $\mathsf h=\{\mathsf h_x\}_{x\in \Lambda}\in \mathbb{R}^\Lambda$, boundary condition $\eta\in \mathbb{R}^{\partial^{\textup{ext}}\Lambda}$, and inverse temperature $\beta \geq 0$, is defined as the probability measure $\nu^\eta_{\Lambda,\beta,J,\mathsf h}$ on $\mathbb{R}^{\overline{\Lambda}}$ given by
\begin{multline}
\mathrm{d}\nu^\eta_{\Lambda,\beta,J,\mathsf h}(\tau) = \frac{1}{\mathbf{Z}^\eta_{\Lambda,\beta,J,\mathsf h}} \exp\Big(\beta \sum_{xy\in \overline{E}(\Lambda)} J_{xy} \tau_x \tau_y + \beta \sum_{x\in \Lambda} \mathsf h_x\tau_x  
\Big)\prod_{x\in \Lambda}\mathrm{d}\rho(\tau_x) \prod_{x \in \partial^{\textup{ext}} \Lambda} \delta_{\eta_x}(\tau_x),
\end{multline}
where $\mathbf{Z}^\eta_{\Lambda,\beta,J,\mathsf h}$ is a normalisation constant, and where $\delta_t$ is the Dirac measure at $t\in \mathbb R$. We will denote by $\langle f \rangle^\eta_{\Lambda,\beta,J,\mathsf h}$ the expectation of a function $f:\mathbb{R}^{\overline{\Lambda}}\rightarrow \mathbb R$ with respect to $\nu^\eta_{\Lambda,\beta,J,\mathsf h}$. When $J\equiv 1$, $\eta\equiv 0$ or $\mathsf{h}\equiv 0$, we may drop them from the notation. 
\end{Def}

\begin{Rem}
The assumption $\int_{\mathbb{R}} e^{au^2}\mathrm{d}\rho(u)<\infty$ for every $a>0$ is necessary and sufficient for the partition function $\mathbf{Z}^{\eta}_{\Lambda,\beta,J,\mathsf{h}}$ to be well-defined \emph{for every} $\beta>0$. To see this, notice that 
$
\tau_x\tau_y=-\frac{(\tau_x-\tau_y)^2}{2}+\frac{\tau_x^2+\tau_y^2}{2}\leq \frac{\tau_x^2+\tau_y^2}{2}.
$
Observe that the Gaussian Free Field, obtained by choosing $\mathrm{d}\rho(t)=\exp(-at^2)\mathrm{d}t$ (and $\beta$ small enough so that the measure is well defined), is \emph{not} included in our class of models.
\end{Rem}

\begin{Ex} Our class of spin models includes the following examples:
\begin{enumerate}
    \item[$\bullet$] The Ising and $\varphi^4$ models, by choosing, respectively
    \begin{equation}
        \rho=\delta_{-1}+\delta_1, \qquad \mathrm{d}\rho(t)=\exp(-gt^4-at^2)\mathrm{d}t,
    \end{equation}
where $g>0$ and $a\in \mathbb R$. More generally, one may take any distribution in the Griffiths--Simon class \cite{GriffithsSimon}.
    \item[$\bullet$] The Blume--Capel model, by choosing
    \begin{equation}
        \rho=e^{\Delta}\delta_{-1}+\delta_0+e^{\Delta}\delta_1,
    \end{equation}
    where $\Delta\in \mathbb R$. 
    \item[$\bullet$] General lattice $P(\varphi)$ models, by choosing
    \begin{equation}\label{eq:equation p(phi)}
    \mathrm{d}\rho(t)=\exp(-P(t))\mathrm{d}t,
    \end{equation}
    where $P$ is an even polynomial of degree at least $4$ and has a positive leading coefficient.  
\end{enumerate}
\end{Ex}

\hfill

\noindent \textbf{Infinite-volume measures and phase transition.} For simplicity, from now on, we will always consider $J\equiv1$ and $\mathsf{h}\equiv 0$ (unless otherwise stated), and drop them from the notation. 

There exist different procedures to define infinite-volume spin measures. We can construct them by taking (sub)sequential limits of measures of the above form, with potentially random boundary conditions. Alternatively, they can be defined as solutions of the Dobrushin--Lanford--Ruelle (DLR) equation, see \eqref{eq:DLR}. These measures form a simplex which admits extremal measures. The \emph{maximal measure} of this simplex, in the sense of stochastic domination, is denoted by $\nu^+_\beta$ and is often called the \emph{plus measure}. We construct it explicitly in Section \ref{sec:preliminaries}. It is translation invariant, ergodic, and extremal. Although these facts are classical for the Ising model, they are more delicate for models with unbounded spins. A careful analysis of the difficulties stemming from unboundedness was first carried out in \cite{LebowitzPresutti1976}, and recently revisited and strengthened in \cite{PV26}. We refer the reader to Section~\ref{sec:basics} for more details. 

It is always possible to define a \emph{flip-symmetric} infinite-volume measure, called the \emph{free measure}, as the following weak limit: $\nu_{\beta}=\lim_{\Lambda\nearrow \mathbb{Z}^d} \nu_{\Lambda,\beta}^0$, see e.g.~\cite[Chapter~4.1]{GlimmJaffeQuantumBOOK}. The existence of a phase transition is related to the breaking of this 
symmetry, or, equivalently, to whether $\nu_\beta$ is the \emph{only} 
infinite-volume measure. As it turns out, the latter can be reduced to checking whether the \emph{spontaneous magnetisation}
\begin{equation}
    m^*(\beta)
    :=
    \langle \tau_0 \rangle^+_\beta
\end{equation}
vanishes or not, 
where $\langle \cdot \rangle^+_\beta$ denotes, as usual, the expectation with respect to $\nu_\beta^+$.
Therefore, we can define the \emph{critical point} as
\begin{equation}
    \beta_c=\beta_c(\rho,d)
    :=
    \inf\lbrace \beta\geq0:\text{ } m^*(\beta)>0\rbrace.
\end{equation}
It is known that, for all the models considered here, one has $\beta_c\in (0,\infty)$ for every $d\geq 2$, see \cite[Proposition~4.1]{PanaVeisubcrit}. 

\hfill

\noindent\textbf{Random cluster representation.} We will now give an informal description of the random cluster (or FK) representation of our spin models; we refer the reader to Section~\ref{sec:randomcluster} for precise definitions and basic properties. 
Under the measure $\nu_{\Lambda,\beta}^0$ and conditionally on the absolute value field $|\tau|$, the sign field $(\mathrm{sign}(\tau_x))_{x\in \Lambda}$ can be seen as an Ising model on $\{x\in \Lambda: \tau_x\neq 0\}$ with random coupling constants $J(|\tau|)_{xy}= |\tau_x||\tau_y|$. Therefore, we can consider the corresponding random cluster representation in a random environment. For $\Lambda\subset \mathbb{Z}^d$, let $\Psi_{\Lambda,\beta}^0$ be the probability measure on configurations $(\omega,\mathsf{a})\in \{0,1\}^{\overline{E}(\Lambda)} \times (\mathbb R^+)^{\overline{\Lambda}}$, where $\mathsf{a}$ is distributed as $|\tau|$ (under $\nu_{\Lambda,\beta}^0$), and $\omega$---conditionally on $\mathsf{a}$---is distributed as an \emph{Ising random cluster model} with coupling constants $J_{xy}=\mathsf{a}_x \mathsf{a}_y$ and inverse temperature $\beta$. We call this the random cluster model on $\Lambda$ with \emph{free boundary conditions}. It follows from the construction and the Edwards--Sokal coupling that $\tau$ can be obtained by sampling $(\omega,\mathsf{a})$ according to $\Psi_{\Lambda,\beta}^0$ and then assigning independent uniform signs to each cluster of $\omega$. As a consequence, certain correlation functions of $\tau$ can then be expressed in terms of connectivity properties for $\Psi_{\Lambda,\beta}^0$. For example, we have (see Corollary \ref{cor: ES correlations} below)
\begin{equation}
\langle \tau_x \tau_y\rangle^0_{\Lambda,\beta}= \Psi_{\Lambda,\beta}^0[ \mathsf{a}_x \mathsf{a}_y \mathbbm{1}\{x \overset{\omega}{\longleftrightarrow} y\}].    
\end{equation}
We can also construct random cluster measures $\Psi^\#_{\Lambda,\beta}$ with \emph{general boundary conditions} $\#=(\xi,\mathsf{b})$, given by a partition $\xi$ of $\partial^{\textup{ext}}\Lambda$ and $\mathsf{b}\in (\mathbb{R}^+)^{\partial^{\textup{ext}}\Lambda}$.

\hfill

In this paper, we analyse the supercritical properties of the measures 
$\Psi_{\Lambda_L, \beta}^\#$, where $L \geq 1$ and $\Lambda_L := [-L, L]^d \cap \mathbb{Z}^d$. We establish a supercritical sharpness result, which roughly states that these measures exhibit robust supercritical behaviour. This is naturally formulated through the \emph{local uniqueness} event, which we now define.

For every scale $L\geq1$, a percolation configuration $\omega\in \{0,1\}^{E(\Lambda_{8L})}$ satisfies the local uniqueness event $U(L)$  if there exists at least one cluster of $\omega$ crossing the annulus $\Lambda_{8L}\setminus \Lambda_L$, and if any two paths in $\omega$ crossing the annulus $\Lambda_{4L}\setminus \Lambda_{2L}$ are connected in $\omega\cap(\Lambda_{8L}\setminus \Lambda_L)$. This event is of interest as it often serves as an essential building block 
in renormalisation arguments. Our main result shows that, throughout the supercritical regime, local uniqueness occurs with high probability.

\begin{Thm}\label{thm:main} Fix $d\geq 2$ and an admissible single-site measure $\rho$. For every $\beta>\beta_c$,
\begin{equation}\label{eq:local_uniqueness}
\lim_{L\to\infty} \inf_{\#}\Psi^{\#}_{\Lambda_{10L},\beta}[U(L)] = 1.   
\end{equation}
\end{Thm}

\begin{Rem}
	In dimensions $d\geq 3$, a more classical formulation of supercritical sharpness is given by the notion of \emph{slab percolation}. This statement can be formulated directly in terms of the original spin model and asserts that the two-point function is uniformly bounded away from $0$ on a sufficiently thick two-dimensional slab, even in finite volume and under free boundary conditions. We refer the reader to the seminal work of Pisztora \cite{Pisztora96surface} for a precise definition in the context of Potts models. We choose to formulate our main theorem in terms of local uniqueness for two reasons. First, local uniqueness remains meaningful in dimension $d=2$. Second, it is the  most useful  formulation for applications based on renormalisation arguments (see the discussion below).
\end{Rem}

Although the statement of Theorem~\ref{thm:main} may at first appear purely technical, it has several far-reaching consequences for the supercritical behaviour of the underlying percolation model. Indeed, standard \emph{renormalisation} (or \emph{coarse-graining}) arguments imply that, for \emph{every} $\beta>\beta_c$, the qualitative behaviour of any random cluster measure on a sufficiently large box or in infinite volume is the same as that of a perturbative model (i.e.~a model with $\beta\gg\beta_c$, or simply Bernoulli percolation with parameter $p$ close to $1$). We briefly recall this classical argument below.

For concreteness, we consider the case of an 
infinite-volume measure. Let $\Psi_\beta$ denote a subsequential limit of any sequence of random cluster measures on growing boxes. 
Fix $L$ sufficiently large (to be chosen later), and consider the renormalised site percolation process $\omega^L$ on $\mathbb{Z}^d$ defined by
\begin{equation*}
	\omega^L_x := \mathbbm{1}_{\{\omega(Lx+\cdot)\in U(L)\}}, ~~~~~~ x\in  \mathbb{Z}^d,
\end{equation*}
where $\omega$ is distributed as the percolation marginal of $\Psi_\beta$. 
First, by the definition of the event $U(L)$, every cluster of $\omega^L$ deterministically induces a cluster of $\omega$.
Second, by the domain Markov property, one has 
\begin{equation*}
	\Psi_\beta[\omega^L_{x}=1 \,|\, ( \omega^L_{y}, \, y\in \mathbb{Z}^d \setminus \Lambda_{20}(x) )]\geq  \inf_{\#}\Psi^{\#}_{\Lambda_{10L},\beta}[U(L)].
\end{equation*}
It then follows from the main result of \cite{LSS97} that $\omega^L$ stochastically dominates a Bernoulli site percolation whose parameter can be made arbitrarily close to $1$, provided that $\inf_{\#}\Psi^{\#}_{\Lambda_{10L},\beta}[U(L)]$
is sufficiently close to $1$. By Theorem~\ref{thm:main}, this condition holds as long as $L$ is chosen large enough.
Consequently, many properties that are known to hold for strongly supercritical Bernoulli percolation automatically transfer to $\Psi_\beta$. Combined with the Edwards--Sokal coupling, this observation yields several important consequences for the original spin model.

\subsection{Mixing and consequences for the spin model}

We now discuss applications of our main result for the original spin model. Under an additional \emph{weak mixing} assumption (which holds at least for almost every $\beta>\beta_c$, see Theorem~\ref{thm:full-space uniqueness} below), Theorem~\ref{thm:main} can be effectively combined with the coarse-graining techniques developed in \cite{gunaratnam2025supercritical} (and inspired by \cite{Pisztora96surface}). As a consequence, for every $\beta>\beta_c$ at which the infinite-volume random cluster measure is unique, one obtains two important results for the spin model, see Theorems~\ref{thm:ldp free} and \ref{thm: weak plus measure} below. Their proofs are omitted, as they proceed along the same lines as the arguments presented in Section 7 and Appendix B of \cite{gunaratnam2025supercritical}, respectively.

We start by discussing the uniqueness of the infinite-volume random cluster measure.  As with the spin model, one can also define minimal and maximal infinite-volume random cluster measures, which are respectively denoted by $\Psi^0_\beta$ and $\Psi^1_\beta$: these are the measures which are naturally coupled with $\nu^0_\beta=\nu_\beta$ and $\nu^+_\beta$, respectively. 
We may then consider the set of inverse temperatures at which there is non-uniqueness of infinite-volume measure, namely
\begin{equation*}
	\Xi:=\{\beta\geq0 : \Psi^0_\beta \neq \Psi^1_\beta\}.
\end{equation*}
It is straightforward to verify that, if $m^*(\beta)=\Psi^1_\beta[\mathsf{a}_x \mathbbm{1}\{x \overset{\omega}{\longleftrightarrow} \infty\}]=0$, then $\beta\notin \Xi$---see e.g. \cite[Proposition C.2]{GunaratnamPanagiotisPanisSeveroPhi42022}. Consequently, $\Xi\subset [\beta_c,\infty)$. We propose the following conjecture regarding the critical and supercritical regimes.
\begin{Conj}\label{conj:uniqueness_Gibbs}
	For every admissible single-site measure $\rho$, one has $\Xi\subset \{\beta_c\}$. Moreover, $\Xi=\emptyset$ if and only if the phase transition is continuous, i.e.~$m^*(\beta_c)=0$.
\end{Conj}

The dichotomy proposed in this conjecture is known to hold for the Blume--Capel model on $\mathbb{Z}^2$ \cite{GKP24}. It is also proved in \cite{GKP24} that, in all dimensions, both continuous and discontinuous phase transitions may occur, depending on the vacancy parameter $\Delta$.  In the particular cases of the Ising and $\varphi^4$ models, it was proved in \cite{RaoufiGibbsMeasures} and \cite{GunaratnamPanagiotisPanisSeveroPhi42022}, respectively, that $\Xi=\emptyset$ in all dimensions, and even on more general transitive graphs.  However, continuity of the phase transition remains open for the Ising and $\varphi^4$ models for these graphs beyond the case of $\mathbb{Z}^d$.

For many spin models, one can often use soft arguments based on the monotonicity of the two-point function (or convexity of the \emph{free energy}) to prove that the set $\Xi$ is at most countable. This is indeed the case for our class of models, as shown by the following result.

\begin{Thm}\label{thm:full-space uniqueness}
	The set $\Xi$ is at most countable. 
\end{Thm}

We are now ready to state the aforementioned applications of Theorem~\ref{thm:main}. The  first result is a surface-order large deviation bound for the empirical magnetisation, analogous to \cite[Theorem~1.1]{gunaratnam2025supercritical}.

\begin{Thm}\label{thm:ldp free}
	For every $d\geq2$, $\beta \in (\beta_c,\infty)\setminus\Xi$ and $\delta\in(0,m^*(\beta) )$, there exist constants $c,C\in(0,\infty)$ such that for every $n$ large enough,
	\begin{equation}\label{eq:ldp surface}
		e^{-Cn^{d-1}}\leq \nu_{\Lambda_n,\beta}\left[ \Big|\frac{1}{|\Lambda_n |} \sum_{x\in\Lambda_n} \tau_x \Big| \leq m^*(\beta)-\delta\right]\leq e^{-cn^{d-1}}.
	\end{equation}
\end{Thm}

\begin{Rem}\label{rem: vol order}
	We stress that the upper large deviation is easier to prove and always of volume order: for every $\beta\geq0$ and $\delta>0$ small enough, there exist $c,C\in(0,\infty)$ such that for every $n$ large enough
	\begin{equation}\label{eq:ldp volume}
		e^{-Cn^d}\leq \nu_{\Lambda_n,\beta}\left[ \Big|\frac{1}{|\Lambda_n |} \sum_{x\in\Lambda_n} \tau_x \Big|  \geq m^*(\beta)+\delta \right]\leq e^{-cn^d}.
	\end{equation}
	See \cite[Section 7]{gunaratnam2025supercritical} for a proof. 
\end{Rem}

\begin{Rem} Theorem \ref{thm:ldp free} can be used to analyse the Glauber dynamics of the class of spin models we treat. Indeed, \cite{gunaratnam2025supercritical} establishes a surface-order exponential upper bound on the 
spectral gap of the Glauber dynamics of the $\varphi^4$ model in the entire regime $\beta>\beta_c$. The proof can be adapted to the present more 
general setting, establishing the analogous result for every 
$\beta\in(\beta_c,\infty)\setminus\Xi$. We refer to \cite[Section~8]{gunaratnam2025supercritical} for more information.
\end{Rem}
The second result is a characterisation of the plus state through moderately positive boundary conditions. It is analogous to \cite[Proposition~B.2]{gunaratnam2025supercritical}. 

\begin{Thm}\label{thm: weak plus measure}
	Let $d\geq2$ and $(\eta_n)_{n\in\mathbb N}$ be a sequence of boundary conditions satisfying $\eta_n(x)\ll e^{n}$ and $\eta_n(x) \gg n^{-(d-1)}$ for every $x\in \partial^{\textup{ext}}\Lambda_n$. Then, for every $\beta\in (\beta_c,\infty)\setminus\Xi$,
	\begin{equation}\label{eq:weak plus equal plus}
		\lim_{n\to\infty} \nu^{\eta_n}_{\Lambda_n,\beta} = \nu^+_{\beta}.    
	\end{equation}
\end{Thm}

\begin{Rem}\label{rem: weak plus measure}
	It is not hard to prove that if $ \eta_n(x) \ll n^{-(d-1)}$ for every $x\in \partial
    ^{\textup{ext}}\Lambda_n$ and every $n$ large enough, then $\lim_{n\to\infty} \nu^{\eta_n}_{\Lambda_n,\beta}= \nu_\beta^0 \neq \nu^+_{\beta}$ for every $\beta>\beta_c$. In particular, the assumption $\eta_n \gg n^{-(d-1)}$ of Theorem~\ref{thm: weak plus measure} is sharp. The assumption $\eta_n\ll e^n$ is only used to guarantee that the measures $(\nu^{\eta_n}_{\Lambda_n,\beta})_{n\in\mathbb N}$ are \emph{uniformly regular} around the origin and thus tight (see \cite{PV26} and Section~\ref{sec:basics} below for more details).
\end{Rem}

The uniqueness of the infinite-volume Gibbs measure may be viewed as a weak form of mixing.
A natural question is whether this can be upgraded to quantitative mixing estimates. In particular, we believe that the following exponential mixing property holds.
\begin{Conj}\label{conj:exp_mixing}
For every $\beta>\beta_c$, there exists $c>0$ such that, for every $n\geq 1$,
\begin{equation}\label{eq:exp_mixing}
   \sup_{\substack{A\in \sigma(\Lambda_n) \\ B\in \sigma(\Lambda_{2n}^c)}} \big|\Psi_\beta[A\cap B] - \Psi_\beta[A]\Psi_\beta[B]\big| \leq e^{-cn},
\end{equation}
where $\sigma(\Lambda)$ denotes the $\sigma$-algebra generated by $(\omega_e)_{e\in \overline{E}(\Lambda)}$ and $(\mathsf{a}_x)_{x\in\Lambda}$, and where $\Psi_\beta$ denotes the (conjecturally) unique infinite volume measure at parameter $\beta$. 
\end{Conj}
At present, such a (non-perturbative) result is only known for the Ising model, where it follows from the random-current representation and its associated switching principle \cite{DCGR20,GPPS26}. Establishing analogous results for more general spin systems remains a major open problem. 
As observed in \cite{DCGR20} for the case of the Ising model, the exponential mixing of Conjecture~\ref{conj:exp_mixing} combined with our main result Theorem~\ref{thm:main} would easily imply exponential decay of truncated correlations, as we quickly explain now. On the one hand, a classical consequence of Theorem~\ref{thm:main} (and the renormalisation procedure described below it) is that for every $\beta>\beta_c$ there exists $c'>0$ such that
\begin{equation}\label{eq:exp_truncated_connection}
    \Psi_\beta[0\longleftrightarrow \partial \Lambda_n, 0\centernot\longleftrightarrow \infty]\leq e^{-c'n}.
\end{equation}
On the other hand, Conjecture~\ref{conj:exp_mixing} allows us to compare truncated correlations with the probability in \eqref{eq:exp_truncated_connection}: if $|x|_\infty>3n$, then 
\begin{align*}
    \langle \tau_0 ; \tau_x \rangle_\beta^+ 
    &:= \langle \tau_0 \tau_x \rangle_\beta^+ - \langle \tau_0 \rangle_\beta^+\langle \tau_x \rangle_\beta^+ 
    = \Psi_\beta[ \mathsf{a}_0 \mathsf{a}_x \mathbbm{1}\{0 \connect{\omega\:} x\}] - \Psi_\beta[\mathsf{a}_0 \mathbbm{1}\{0 \connect{\omega\:} \infty\}]^2 \\
    &\leq \Psi_\beta[\mathsf{a}_0 \mathbbm{1}\{0 \connect{\omega\:} \partial \Lambda_n\} \mathsf{a}_x \mathbbm{1}\{x \connect{\omega\:} \partial \Lambda_n(x)\}] - \Psi_\beta[\mathsf{a}_0 \mathbbm{1}\{0 \connect{\omega\:} \infty\}]^2 \\
    &\leq e^{-cn} + \Psi_\beta[\mathsf{a}_0 \mathbbm{1}\{0 \connect{\omega\:} \partial \Lambda_n\}]^2 - \Psi_\beta[\mathsf{a}_0 \mathbbm{1}\{0 \connect{\omega\:} \infty\}]^2 \\
    &\leq e^{-cn} + C\Psi_\beta[ \mathsf{a}_0 \mathbbm{1}\{0 \connect{} \partial \Lambda_n, 0\centernot\longleftrightarrow \infty \} ] \\
    &\leq e^{-cn} + C' \sqrt{\Psi_\beta[ 0 \connect{} \partial \Lambda_n, 0\centernot\longleftrightarrow \infty ]},
\end{align*}
where the third line is due to \eqref{eq:exp_mixing}\,\footnote{Even though \eqref{eq:exp_mixing} is stated for events, it is straightforward to extend it to well-behaved random variables, and all the variables here are well-behaved due to regularity---see Proposition~\ref{prop:reg_plus}.} and in the last line we used the Cauchy–Schwarz inequality.

\subsection{Strategy of proof}

We finish this introduction by describing the main steps and ideas in our proof of Theorem~\ref{thm:main}. By the same arguments as in \cite[Section 6]{gunaratnam2025supercritical}, Theorem~\ref{thm:main} follows from the following result.

\begin{Thm}\label{thm:sharpness} Let $d\geq2$ and let $\rho$ be an admissible single-site measure. Let $\beta>\beta_c$. There exists $c>0$ such that for every $L\geq 1$ and $1\leq \ell\leq cL$, we have
\begin{equation}\label{eq:disconnection_surface}
        \Psi^0_{\Lambda_{L},\beta}[\Lambda_\ell \centernot\longleftrightarrow \partial \Lambda_{cL}]\leq e^{-c\ell^{d-1}}.
    \end{equation}
\end{Thm}

We quickly describe the argument to deduce the local uniqueness property \eqref{eq:local_uniqueness} from the disconnection bound \eqref{eq:disconnection_surface}---see Sections 6.1 and 6.2 of \cite{gunaratnam2025supercritical} for a full proof, which applies mutatis mutandis to our context. 
On the one hand, for $d=2$, \eqref{eq:disconnection_surface} says that short rectangles are crossed with high probability (w.h.p.), which then implies that long rectangles are also crossed w.h.p.~by the general Russo--Seymour--Welsh result of \cite{KohlerTassionGeneralRSW}. As a consequence, we can construct circuits in an annulus w.h.p., which then deterministically implies local uniqueness of a crossing cluster by planarity. The case $d\geq3$, on the other hand, follows closely the proof of \cite{SeveroSlabIsing2023} for the Ising model. First, fix any $\beta>\beta'>\beta_c$. Let $\omega'\sim \Psi_{\beta'}$. Note that a union bound over \eqref{eq:disconnection_surface} implies that every vertex of $\Lambda_{c'L}$ lies within distance $\ell$ of a cluster of $\omega'$ touching $\partial\Lambda_{c'L}$ w.h.p., where $\ell:=C(\log L)^{\frac{1}{d-1}} = o(\log L)$. One can then apply \cite[Proposition 6.2]{gunaratnam2025supercritical} to deduce that local uniqueness holds for $\omega'\cup \gamma_\epsilon$, and $ \gamma_\varepsilon$ is an independent Bernoulli percolation of parameter $\varepsilon$. By standard renormalization arguments (as explained above), this implies that $\omega'\cup\gamma_\varepsilon$ percolates on sufficiently thick 2D slabs. Since $\omega\sim\Psi_{\beta}$ stochastically dominates $\omega'\cup \gamma_\varepsilon$ for $\varepsilon$ small enough (see Proposition~\ref{prop:sprinkling} below), this yields slab percolation for the original model $\Psi_{\beta}$, from which local uniqueness follows through a standard ``onion-peeling'' argument.

Our main contribution thus lies in the proof of Theorem~\ref{thm:sharpness}, which we now explain. Similarly to \cite{Bod05} and \cite{gunaratnam2025supercritical}, our proof consists of two main steps: positivity of surface tension and comparison between boundary conditions.  However, our more general framework requires new ideas when implementing each of these steps, which we discuss below.

In the first step, we adapt the proof of the corresponding result for the Ising model from \cite{lebowitz1981surface}. On the one hand, the main ingredient in this proof is the so-called \emph{Lebowitz inequality} from \cite{Lebowitz1977CoexistencePhasesIsing}, which was later generalised to our setting in \cite{lebowitz1978number} (see also \cite[Theorem~2.12.5]{simon2026phase}). We include a full proof in Section \ref{sec:surface_tension} for the sake of completeness. 
On the other hand, an important technical aspect of this step is that, since the spins are unbounded, there is no obvious notion of ``maximal'' measure in finite volume, and thus no unambiguous notion of surface tension. Indeed, for the proof to work, we need a notion of finite-volume  measures $\nu^+_{\Lambda_L,\beta}$ with ``plus boundary conditions'' which satisfy the following two properties simultaneously: $(i)$ they converge to the infinite-volume plus measure $\nu_\beta^+$, and $(ii)$ they are \emph{regular} (in the sense of Proposition~\ref{prop:reg_plus}) up to the boundary. The first property guarantees that the proof of \cite{lebowitz1981surface} can be adapted, whereas the second is needed for the second step, which is explained below.
The issue of constructing a robust notion of finite-volume plus measure is resolved in \cite{gunaratnam2025supercritical} for the special case of the $\varphi^4$ model by considering a ``thick layer'' of $+1$ magnetic field near the boundary and then proving that property $(i)$ holds via the random tangled current representation of the model. Since such a representation is not available for our models, we need to follow another route: we construct an appropriate \emph{random} positive boundary condition, obtained by tilting the single-site measure on the boundary. The resulting measures satisfy properties $(i)$ and $(ii)$ by the results of \cite{PV26}, see Section~\ref{sec:plus_measure}.  With this notion of plus boundary condition, we can define the corresponding notion of surface tension and then prove that it is strictly positive for every $\beta>\beta_c$. Via the Edwards--Sokal coupling, this directly implies that the inequality \eqref{eq:disconnection_surface} holds for the measure $\Psi^1_{\Lambda_L,\beta}$ which is naturally coupled with $\nu^+_{\Lambda_L,\beta}$. 

In the second step, we transfer the bound \eqref{eq:disconnection_surface} from $\Psi^1_{\Lambda_L,\beta}$ to $\Psi^0_{\Lambda_L,\beta}$. Such a comparison between boundary conditions can be proved by taking the derivative, in boundary field strength, of the log probability of disconnection, and showing that it is negligible when compared to $L^{d-1}$, as done in both \cite{Bod05} and \cite{gunaratnam2025supercritical}.  The difficulty is that this argument relies on the uniqueness of infinite-volume random cluster measure on the half-space with constant positive magnetic field on the boundary. In \cite{Bod05}, this statement is shown to hold for the Ising model by relying on results from \cite{frohlich1987semi} and a Lee--Yang theorem. For the $\varphi^4$ model, on the other hand, this is proved in \cite{gunaratnam2025supercritical} by relying on its random tangled current representation. 
Since neither of these tools is available in our setting, we have to find a different approach. In fact, we first observe that uniqueness of half-space infinite-volume measure is only required to hold for almost every pair of inverse temperature and (constant) boundary magnetic  field.   It is then natural to try to prove a result analogous to Theorem~\ref{thm:full-space uniqueness} in the half-space. This is the content of our Theorem~\ref{thm:half-space}, which is the main innovation in our paper. However, the proof of Theorem~\ref{thm:full-space uniqueness} does not adapt to the half-space due to lack of translation invariance. Instead, we use  Theorem~\ref{thm:full-space uniqueness} to say that the minimal and maximal measures on the half-space are close sufficiently far from the boundary of the half-space, and then argue that this implies that the boundary conditions induced by the complement of a large half-space box are also approximately equal for both measures. This ultimately yields equality of the two half-space measures via an argument analogous to that of Theorem~\ref{thm:full-space uniqueness}. In order to prove that induced partitions are comparable, we use geometric ideas inspired by the classical Burton--Keane argument \cite{BurtonKeane1989density}. 

We note that our proof of Theorem~\ref{thm:half-space} admits a straightforward adaptation to the standard random cluster model with arbitrary cluster weight $q\geq1$. In fact, the proof is much simpler for the latter since there is no need to deal with the random absolute-value field $\mathsf{a}$---see Remark~\ref{rem:FK_adaptation}. 

\hfill

\noindent \textbf{Organisation of the paper.} In Section~\ref{sec:preliminaries}, we define the main models in this paper, namely the spin model and its random cluster representation, and state some of their fundamental properties. In Section~\ref{sec:Gibbs_FK}, we prove the uniqueness of half-space random cluster measures for almost every parameter, as well as its (simpler) full-space analogue stated in Theorem~\ref{thm:full-space uniqueness}. In Section~\ref{sec:surface_tension}, we prove the Lebowitz inequality and use it to deduce the strict positivity of the surface tension. Finally, in Section~\ref{sec:comparison_bc}, we implement the comparison between boundary conditions described above, thus concluding the proof of Theorem~\ref{thm:sharpness}.

\hfill

\noindent \textbf{Acknowledgements.} TSG was supported by the Department of Atomic Energy, Government of India, under project no.\ 12-R\&D-TFR-5.01-0500. CP was supported by an EPSRC New Investigator Award (UKRI1019). RP acknowledges support from the Swiss National Science Foundation through a Postdoc.Mobility grant. FS acknowledges support
from the ERC grant Vortex (No.~101043450).

\section{Preliminaries}\label{sec:preliminaries}

In this section, we fix an admissible measure $\rho$, which we suppress from the notation.

\subsection{Basic properties}\label{sec:basics}

\subsubsection{Correlation inequalities and monotonicity properties}

We now present several classical correlation inequalities used throughout 
this paper. These results were established in 
\cite{GunaratnamPanagiotisPanisSeveroPhi42022} in the context of the 
$\varphi^4$ model, and their proofs carry over to our more general setting; 
we refer to that paper and the references therein for further details. We 
begin with correlations of increasing functions.

\begin{Prop}[FKG inequalities, {\cite[Propositions~B.1--B.2]{GunaratnamPanagiotisPanisSeveroPhi42022} and \cite[Corollary~6.4]{LammersOtt2021}}] \label{prop: FKG phi4}

Let $\Lambda \subset \mathbb Z^d$ be finite, $\beta\geq0$, $J\in (\mathbb R^+)^{E(\Lambda)}$, and $\mathsf{h}\in \mathbb R^\Lambda$. For any increasing and bounded functions $F,G:\mathbb{R}^\Lambda\to \mathbb{R}$,
\begin{equation}
    \langle F(\tau)G(\tau)\rangle_{\Lambda,\beta,J, \mathsf{h}}\geq \langle F(\tau)\rangle_{\Lambda,\beta,J, \mathsf{h}}\langle G(\tau)\rangle_{\Lambda,\beta,J,\mathsf{h}}.
\end{equation}
The same holds for the absolute value field if $\mathsf h_x \geq 0$ for every $x \in \Lambda$, namely
\begin{equation}
    \langle F(|\tau|)G(|\tau|) \rangle_{\Lambda,\beta,J,\mathsf{h}}\geq \langle F(|\tau|) \rangle_{\Lambda,\beta,J,\mathsf{h}} \langle G(|\tau|) \rangle_{\Lambda,\beta,J,\mathsf{h}}.
\end{equation}
\end{Prop}

As a consequence of the FKG inequalities above, we obtain the following monotonicity properties of correlations.

\begin{Prop}\label{prop:stoc_monotonicity}
Let $\Lambda \subset \mathbb Z^d$ be finite, and $\beta\geq0$. Let also $J_1,J_2\in (\mathbb R^+)^{E(\Lambda)}$ be such that $J_1\leq J_2$, and $\mathsf{h}_1,\mathsf{h}_2\in \mathbb R^\Lambda$ be such that $\mathsf{h}_1\leq \mathsf{h}_2$. For any increasing and bounded function $F:\mathbb{R}^{\Lambda}\to\mathbb{R}$ we have 
\begin{equation}
\langle F(\tau) \rangle_{\Lambda,\beta,J_1,\mathsf{h}_1}\leq \langle F(\tau) \rangle_{\Lambda,\beta,J_1,\mathsf{h}_2}.  
\end{equation}
Moreover, if $|\mathsf{h}_1| \leq \mathsf{h}_2$, for any increasing and bounded function $F: (\mathbb{R}^+)^{\Lambda}\to\mathbb{R}$ we have 
\begin{equation}
\langle F(|\tau|) \rangle_{\Lambda,\beta,J_1,\mathsf{h}_1}\leq \langle F(|\tau|) \rangle_{\Lambda,\beta,J_2,\mathsf{h}_2}.  
\end{equation}
\end{Prop}
\begin{proof} The first inequality is standard. The second inequality when $J_1=J_2$ follows from a straightforward adaptation of the proof of \cite[Lemma 2.13]{GunaratnamPanagiotisPanisSeveroPhi42022}. The more general monotonicity statement for the absolute value field follows from the properties of the associated random cluster measure, as stated in Proposition \ref{prop: random cluster properties} below.
\end{proof}

We now turn to spin correlations. In what follows, given $\Lambda\subset\mathbb{Z}^d$ finite and $A:\Lambda\rightarrow \mathbb{N}$, we write $\tau_A\coloneqq \prod_{x\in \Lambda} \tau_x^{A_x}$.

\begin{Prop}[Monotonicity of correlation functions, {\cite[Proposition 3.20 and Remark 3.21]{GunaratnamPanagiotisPanisSeveroPhi42022}}] \label{prop:monotonicity}
Let $\Lambda \subset \mathbb Z^d$ be finite, $\beta\geq0$, $J\in (\mathbb R^+)^{E(\Lambda)}$, $\mathsf{h},\mathsf{h}'\in (\mathbb R^+)^\Lambda$, and $A: \Lambda\to \mathbb{N}$. For every $e\in E(\Lambda)$, the function
$J_e\mapsto \langle \tau_A\rangle_{\Lambda,\beta,\mathsf{h},J}$ is increasing. Moreover, if $\mathsf{h}'\leq \mathsf{h}$, one has $\langle \tau_A\rangle_{\Lambda,\beta,\mathsf{h}',J}\leq \langle \tau_A\rangle_{\Lambda,\beta,\mathsf{h},J}$.
\end{Prop}

\subsubsection{Infinite volume measures}

We now consider infinite volume Gibbs measures for the spin model at parameters $\beta\geq 0$ and $\mathsf{h}\in(\mathbb R^+)^{\mathbb Z^d}$.

\hfill

\noindent\textbf{DLR property.} We say that a Borel probability measure $\nu$ on $\mathbb R^{\mathbb Z^d}$ satisfies the DLR property at $(\beta,\mathsf{h})$ if, for every bounded measurable function $F: \mathbb R^{\mathbb Z^d} \rightarrow \mathbb R$ depending only on spins on some finite $\Lambda \subset \mathbb Z^d$, one has
\begin{equation}\label{eq:DLR}
\int_{\mathbb R^{\mathbb Z^d}} F(\tau) \mathrm{d}\nu(\tau) = \int_{\mathbb R^{\mathbb Z^d \setminus \Lambda}} \langle F \rangle_{\Lambda, \beta,\mathsf{h}}^{\tau}\mathrm{d}\nu(\tau).
\end{equation}

\hfill

\noindent\textbf{Regularity.} For $a>0$ and $b\in \mathbb R$, we say that $\nu$ is $(a,b)$-\emph{regular} if for every $\Lambda \subset \mathbb Z^d$ finite, the restriction measure $\nu^{(\Lambda)}$ has a density with respect to $\prod_{x \in \Lambda} \mathrm{d}\rho(\tau_x)$ satisfying
\begin{equation} \label{eqdef: regularity}
\frac{\mathrm{d}\nu^{(\Lambda)}(\tau)}{\prod_{x \in \Lambda} \mathrm{d}\rho(\tau_x)} \leq \prod_{x \in \Lambda} e^{a\tau_x^2+b} .
\end{equation}
We say that $\nu$ is a regular Gibbs measure at $(\beta,\mathsf{h})$ if it satisfies the DLR property and is regular for some constants $(a,b)$. We denote the set of regular Gibbs measures by $\mathcal G(\beta,\mathsf{h})$.

The set of regular Gibbs measures has a simplex structure and admits extremal measures, see \cite{LebowitzPresutti1976, GunaratnamPanagiotisPanisSeveroPhi42022,PV26}. We say that a measure $\nu$ is maximal in $\mathcal G(\beta,\mathsf{h})$ if, for every $\mu \in \mathcal G(\beta,\mathsf{h})$, $\nu$ stochastically dominates $\mu$, which we denote $\mu \preccurlyeq \nu$. The following proposition concerns the existence of a maximal extremal measure.

\begin{Prop}\label{prop: existence of plus state extremal}
Let $\beta_0,h_0\geq 0$. There exist $a_0,b_0>0$ such that the following properties hold for $\beta\leq \beta_0$ and $\mathsf{h}\in (\mathbb R^+)^{\mathbb Z^d}$ such that $\mathsf{h}\leq h_0$. 
\begin{enumerate}
    \item[$(i)$] The measures $\nu_{\Lambda, \beta,\mathsf{h}}$ are $(a_0,b_0)$-regular for every finite $\Lambda\subset \mathbb Z^d$. In addition, they converge to a measure $\nu_{\beta,\mathsf{h}}\in \mathcal{G}(\beta,\mathsf{h})$ as $\Lambda\nearrow \mathbb Z^d$. We call $\nu_{\beta,\mathsf h}$ the infinite volume free measure at parameters $(\beta, \mathsf h)$.
    \item[$(ii)$] There exists a measure $\nu^+_{\beta,\mathsf{h}} \in \mathcal{G}(\beta,\mathsf{h})$ which is maximal, extremal, and $(a_0,b_0)$-regular. We call $\nu^+_{\beta,\mathsf{h}}$ the infinite volume plus measure at parameters $(\beta,\mathsf{h})$.
\end{enumerate}
\end{Prop}

\begin{proof}
The convergence of the free measure follows directly from the regularity estimates of \cite{LebowitzPresutti1976} (see also \cite{GunaratnamPanagiotisPanisSeveroPhi42022,PV26}), together with monotonicity in the volume of Proposition \ref{prop:monotonicity}. The existence of the plus measure was also derived in \cite{LebowitzPresutti1976}, where the finite volume plus measure was defined with respect to slowly growing boundary conditions. 
\end{proof}

\subsubsection{Finite volume plus measures}\label{sec:plus_measure}

We now turn to the definition of a finite volume plus measure. When $\rho$ has unbounded support, there is no canonical way of defining an analogue of $\nu^+_{\beta}$ on a finite volume $\Lambda \subset \mathbb Z^d$ that preserves the maximality in terms of stochastic domination. As mentioned above, \cite{LebowitzPresutti1976} considered weakly growing deterministic boundary conditions. In this paper, following \cite{PV26}, we define it using \emph{random boundary conditions}. The upshots are that the corresponding measures preserve regularity estimates \emph{up to the boundary} and satisfy a form of monotonicity in the volume, see Proposition~\ref{prop:monotonicity in plus} below. We begin with a definition. If $\zeta$ is a measure on $\mathbb R$, we say that it has all Gaussian moments if, for every $A>0$, one has
\begin{equation}
    \int_{\mathbb R}e^{At^2}\mathrm{d}\zeta(t)<\infty.
\end{equation}

\begin{Def}\label{def:+measure spin finite volume}
    Let $\Lambda \subset \mathbb Z^d$ be finite, $\beta\geq 0$, and $\mathsf{h}\in \mathbb R^\Lambda$. Let also $\zeta$ be a measure supported on $\mathbb R^+$ that has all Gaussian moments. We abuse notation, and denote by $\nu^\zeta_{\Lambda,\beta, \mathsf{h}}$ the probability measure on $\mathbb R^{\overline{\Lambda}}$ defined by the Radon--Nikodym density
\begin{equation}
\mathrm{d}\nu^\zeta_{\Lambda,\beta,\mathsf{h}}(\tau) = \frac{1}{\mathbf{Z}^\zeta_{\Lambda,\beta,\mathsf{h}}} \exp\Big(\beta \sum_{xy\in \overline E(\Lambda)} \tau_x \tau_y
+\beta \sum_{x\in \Lambda}\mathsf{h}_x\tau_x\Big)\prod_{x\in \Lambda}\mathrm{d}\rho(\tau_x) \prod_{x \in \partial^{\textup{ext}} \Lambda} \mathrm{d}\zeta (\tau_x), \qquad \forall \tau \in \mathbb R^{\overline \Lambda},
\end{equation}
where $\mathbf{Z}^\zeta_{\Lambda,\beta,\mathsf{h}}$ is a normalisation constant. We write $\langle \cdot \rangle^\zeta_{\Lambda,\beta,\mathsf{h}}$ for the expectation with respect to $\nu^\zeta_{\Lambda,\beta,\mathsf{h}}$. 
\end{Def}

\begin{Prop}[\hspace{1pt}{\cite[Theorem 4.1, Proposition 6.8]{PV26}}]\label{prop:reg_plus} 
Let $\beta_0,h_0 \geq 0$.  There exists a measure $\zeta=\zeta(\rho,\beta_0,h_0)$ supported on $\mathbb R^+$ and which has all Gaussian moments,
such that the following holds for $\beta\leq \beta_0$ and $\mathsf{h}\in (\mathbb R^+)^{\mathbb Z^d}$ such that $\mathsf{h}\leq h_0$.
\begin{enumerate}
    \item[$(i)$] The measures $\nu^{\zeta}_{\Lambda,\beta,\mathsf{h}}$ converge weakly to the measure $\nu_{\beta,\mathsf{h}}^+$ defined in Proposition \textup{\ref{prop: existence of plus state extremal}} as $\Lambda\nearrow \mathbb Z^d$.
\item[$(ii)$] There exist $a=a(\beta_0,h_0),b=b(\beta_0,h_0)>0$  such that, for every finite $\Lambda\subset \mathbb Z^d$, the measure $\nu^\zeta_{\Lambda,\beta,\mathsf{h}}$ is $(a,b)$-regular, in the following sense. For every $\Delta \subset \overline{\Lambda}$, the restriction measure $(\nu^\zeta_{\Lambda,\beta,\mathsf{h}})^{(\Delta)}$ has a density with respect to $\prod_{x\in \Delta \cap \Lambda}\mathrm{d}\rho(\tau_x)\prod_{x\in \partial^{\textup{ext}}\Lambda\cap \Delta}\mathrm{d}\zeta(\tau_x)$ satisfying
\begin{equation}
 \frac{\mathrm{d}(\nu^\zeta_{\Lambda,\beta,\mathsf{h}})^{(\Delta)}(\tau)}{\prod_{x\in \Delta\cap \Lambda}\mathrm{d}\rho(\tau_x)\prod_{x\in \Delta\cap \partial^{\textup{ext}}\Lambda}\mathrm{d}\zeta(\tau_x)} \leq \prod_{x\in \Delta}e^{a\tau_x^2+b}.
\end{equation}
\end{enumerate}
\end{Prop}

\noindent\textbf{Convention.} In this article, we will always be working with values of $\beta,h$ belonging to a compact interval (i.e.~we never make these parameters diverge to infinity). This means that in all the plus measures considered below, we can choose the same measure $\zeta(\beta_0,h_0)$ in their definitions and drop $\beta_0,h_0$ from the notation.

\begin{Rem}\label{rem:plus}
We make a few comments on Proposition \ref{prop:reg_plus}.
\begin{enumerate}
    \item[$(i)$] As a consequence of \cite[Remark 6.12]{PV26}, we can always choose $\zeta$ to be supported on $[C,\infty)$, for $C>0$ arbitrarily large. We will use this fact in Section \ref{sec:comparison_bc} below. 
    \item[$(ii)$] In the case $\rho(\mathrm{d}t)\propto \exp(-gt^{2n})\mathrm{d}t$, $\zeta$ can be chosen as the shift of a probability measure with density proportional to $\exp(-(g/2)t^{2n})\mathrm{d}t$, see \cite[Remark~6.13]{PV26} for more details.
    \item[$(iii)$] We may always view the plus measure as a free measure of a model with \emph{inhomogeneous} single-site measures that are identically equal to $\rho$ in $\Lambda$, and equal to $\zeta$ on $\partial^{\textup{ext}}\Lambda$. Therefore, as we shall see in Proposition \ref{prop: monot plus}, many correlation inequalities and properties valid for free measures with homogeneous single-site measures also hold true for this more general class. A key exception is monotonicity in volume, which will be treated in Proposition \ref{prop:monotonicity in plus}.
    \item[$(iv)$] We could also define, as before, a measure $\nu^+_{\Lambda,\beta,J,\mathsf{h}}$ for coupling constants $J=\{J_{xy}\}_{xy\in E(\Lambda)}$. 
\end{enumerate}
\end{Rem}

Many of the correlation inequalities valid for the measure $\nu_{\Lambda,\beta,\mathsf{h}}$ also hold for $\nu_{\Lambda,\beta,\mathsf{h}}^+$, as formalised in the proposition below.

\begin{Prop}\label{prop: monot plus} The following properties hold.
\begin{enumerate}
    \item[$\bullet$] \textup{(FKG inequalities)} Let $\Lambda \subset \mathbb Z^d$ be finite, $\beta \geq 0$, $\mathsf h \in \mathbb R^\Lambda$. For all increasing and bounded functions $F,G:\mathbb R^\Lambda \rightarrow \mathbb R$,
    \begin{equation}
    \langle F(\tau) G(\tau) \rangle_{\Lambda,\beta, \mathsf h}^+ \geq \langle F(\tau) \rangle_{\Lambda,\beta,\mathsf h}^+  \langle G(\tau) \rangle_{\Lambda,\beta,\mathsf h}^+. 
    \end{equation}
    The same holds for the absolute value field if $\mathsf h_x \geq 0$ for every $x \in \Lambda$, namely
\begin{equation}
    \langle F(|\tau|)G(|\tau|) \rangle_{\Lambda,\beta,\mathsf{h}}^+\geq \langle F(|\tau|) \rangle_{\Lambda,\beta,\mathsf{h}}^+ \langle G(|\tau|) \rangle_{\Lambda,\beta,\mathsf{h}}^+.
\end{equation}
\item[$\bullet$] \textup{(Monotonicity in $\mathsf{h}$)} Let $\Lambda \subset \mathbb Z^d$ be finite, and $\beta\geq0$. Let also  $\mathsf{h}_1,\mathsf{h}_2\in \mathbb R^\Lambda$ be such that $\mathsf{h}_1\leq \mathsf{h}_2$. For every increasing and bounded function $F:\mathbb{R}^{\Lambda}\to\mathbb{R}$ we have 
\begin{equation}
\langle F(\tau) \rangle_{\Lambda,\beta,\mathsf{h}_1}^+\leq \langle F(\tau) \rangle_{\Lambda,\beta,\mathsf{h}_2}^+.  
\end{equation}
The same holds for the absolute value field if $|\mathsf{h}_1| \leq \mathsf{h}_2$. Namely, for every increasing and bounded function $F: (\mathbb{R}^+)^{\Lambda}\to\mathbb{R}$ we have 
\begin{equation}
\langle F(|\tau|) \rangle_{\Lambda,\beta,\mathsf{h}_1}^+\leq \langle F(|\tau|) \rangle_{\Lambda,\beta,\mathsf{h}_2}^+.  
\end{equation}
\item[$\bullet$] \textup{(Monotonicity of correlation functions)} Let $\Lambda \subset \mathbb Z^d$ be finite, $\beta\geq0$, $\mathsf{h},\mathsf{h}'\in (\mathbb R^+)^\Lambda$, and $A: \Lambda\to \mathbb{N}$. The function
$\beta \mapsto \langle \tau_A\rangle_{\Lambda,\beta,\mathsf{h}}^+$ is increasing. Moreover, if $\mathsf{h}'\leq \mathsf{h}$, one has $\langle \tau_A\rangle_{\Lambda,\beta,\mathsf{h}'}^+\leq \langle \tau_A\rangle_{\Lambda,\beta,\mathsf{h}}^+$.
\item[$\bullet$] \textup{(Comparison to the free measure)} Let $\Lambda \subset \mathbb Z^d$ be finite, $\beta\geq0$, $\mathsf{h}\in (\mathbb R^+)^\Lambda$, and $A: \Lambda\to \mathbb{N}$. One has
\begin{equation}
    \langle \tau_A\rangle_{\Lambda,\beta,\mathsf{h}}\leq \langle \tau_A\rangle^+_{\Lambda,\beta,\mathsf{h}}.
\end{equation}
\end{enumerate}
\end{Prop}
\begin{proof} The first three points follow similarly to Proposition \ref{prop:monotonicity} using Remark \ref{rem:plus}. The last one follows from a combination of the domain Markov property for the plus measure and Proposition \ref{prop:monotonicity}.
\end{proof}

We now turn to the monotonicity in the volume that we will need in the sequel. For sets $X,Y\subset \mathbb R^d$, we let $\mathrm{d}(X,Y):=\inf_{x\in X, y\in Y}|x-y|_\infty$, where $|\cdot|_\infty$ denotes the $\ell_\infty$ norm on $\mathbb R^d$. For a measure $\mu$, we write $\mu(|\cdot|)$ to denote the push-forward of $\mu$ by the map $\tau\mapsto |\tau|$.

\begin{Prop}[\hspace{1pt}{\cite[Proposition 6.11]{PV26}}]\label{prop:monotonicity in plus}
Let $\beta,h \geq 0$. There exists $r=r(\beta,h)\geq 0$ such that, for every $\Lambda'\subset \Lambda \subset \mathbb Z^d$ finite and such that $\mathrm{d}(\partial \Lambda,\Lambda')\geq r$, and every $\mathsf h \in (\mathbb R^+)^{\Lambda}$ and $\mathsf h' \in (\mathbb R^+)^{\Lambda'}$ such that $\mathsf h' \geq \mathsf h_{|_{\Lambda'}}$ and $\mathsf h,\mathsf h'\leq h$,
\begin{equation}\label{eq:spin-absolute value monotonicity}
\nu_{\Lambda,\beta, \mathsf h}^+ \preccurlyeq \nu_{\Lambda',\beta, \mathsf h'}^+ \quad \text{and} \quad \nu_{\Lambda,\beta, \mathsf h}^+(|\cdot|) \preccurlyeq \nu_{\Lambda',\beta, \mathsf h'}^+(|\cdot|).
\end{equation}
In particular, for every $\Lambda \subset \mathbb Z^d$ finite and every $\mathsf h \in (\mathbb R^+)^{\mathbb Z^d}$ with $\mathsf h \leq h$,
\begin{equation}\label{eq:plus monotonicity}
\nu_{\beta,\mathsf h}^+ \preccurlyeq \nu_{\Lambda,\beta, \mathsf h}^+ \quad \text{and} \quad \nu_{\beta, \mathsf h}^+(|\cdot|) \preccurlyeq \nu_{\Lambda,\beta, \mathsf h}^+(|\cdot|).
\end{equation}
\end{Prop}

\begin{proof}
The proof of \eqref{eq:spin-absolute value monotonicity} follows from \cite[Proposition 6.11]{PV26}, and the proof of
\eqref{eq:plus monotonicity} follows from \eqref{eq:spin-absolute value monotonicity} and the convergence of 
$\nu^+_{\Lambda,\beta, \mathsf{h}}$ to $\nu^+_{\beta, \mathsf{h}}$ of Proposition~\ref{prop:reg_plus}.
\end{proof}

\begin{Rem} Observe that this proposition matches known results in the setting of the Ising model, see for instance \cite[Chapter~3.6.3]{FriedliVelenikIntroStatMech2017}. In this easier setting, we can even choose $r=0$ above.
\end{Rem}

\subsection{Random cluster representation}\label{sec:randomcluster}
\subsubsection{Notations}

The random cluster representation of a one-component spin model samples an absolute value field on the vertices and a percolation configuration on the edges. This model has a natural domain Markov property; in order to state it, we define the model with both external magnetic fields and boundary conditions. Below, we record some standard notation, introduce boundary conditions, and define the state space of the measures. Throughout this subsection, $\Lambda$ is a finite subset of $\mathbb Z^d$.

\hfill

\noindent \textbf{Vertex sets.} Recall that $\partial \Lambda = \{ x\in \Lambda : \exists y \in \mathbb Z^d \setminus \Lambda, \, x \sim y \}$ and $\partial^{\textup {ext}} \Lambda=\{y \in \mathbb Z^d \setminus \Lambda: \exists x\in \Lambda, \, x\sim y\}$ denote the inner and outer vertex boundaries of $\Lambda$, respectively, and that $\overline \Lambda= \Lambda \sqcup \partial^{\textup{ext}} \Lambda$. In order to account for a positive external field, we add a ghost vertex $\mathfrak g$, and we set $\Lambda^\mathfrak g:= \Lambda \sqcup \{\mathfrak g\}$ and $\overline \Lambda^\mathfrak g:= \Lambda^{\mathfrak g}  \sqcup \partial^{\textup{ext}} \Lambda$.  Below, we also set $\tau_\fg=\mathsf{a}_\fg=1$.

\hfill

\noindent \textbf{Edge sets.} Recall that $E(\Lambda)= \{ xy : x,y \in \Lambda,\, x\sim y\}$ is the induced set of edges on $\Lambda$, and $\overline E(\Lambda) = \{ xy : x \in \Lambda, \, y \in \mathbb Z^d,\, x\sim y \}$ is the set of edges touching $\Lambda$. The edge boundary of $\Lambda$ is $\partial_E \Lambda= \overline E(\Lambda) \setminus E(\Lambda)$. For the graph with the ghost vertex, we augment the edge set as follows. Given $\mathsf h \in (\mathbb R^+)^\Lambda$, write $E(\Lambda[\mathsf h]):=E(\Lambda) \sqcup \{ x\mathfrak g: x \in \Lambda\}$ and let $\overline E(\Lambda[\mathsf h]) = E(\Lambda[\mathsf h]) \cup \overline E(\Lambda)$. We sometimes abuse notation and consider $(\Lambda^\fg,E(\Lambda[\mathsf{h}]))$ as a weighted graph (i.e.,~remembering the values of $\mathsf h$). 

\hfill

\noindent\textbf{Boundary conditions and state space.} 
A boundary condition on $\Lambda$ is a pair $(\xi,\mathsf{b})$, where  $\xi$ is a partition of $\partial^{\textup{ext}}\Lambda\cup \{\fg\}$ and $\mathsf b \in (\mathbb R^+)^{\partial^{\textup{ext}} \Lambda}$. The $\xi$ term will only play a role in the weight of a configuration but not the state space. We emphasise that all vertices $x\in \partial^{\textup{ext}}\Lambda\cup \{\fg\}$ are included in some partition class of $\xi$, even those with $\mathsf{b}_x=0$. A random cluster configuration is a pair $(\omega,\mathsf a) \in \{0,1\}^{\overline E(\Lambda[\mathsf h])}\times (\mathbb R^+)^{\overline{\Lambda}} $ such that $\mathsf a_x = 0$ implies $\omega_e=0$ for every edge $e$ incident to $x$. We say that $(\omega,\mathsf a)$ is compatible with $(\xi,\mathsf b)$ if $\mathsf a|_{\partial^{\textup{ext}} \Lambda} = \mathsf b$. Given a boundary condition, the set of all such compatible configurations is the state space for our model.

\begin{Rem}\label{rem:double ghost}
We can also treat external magnetic fields $\mathsf h \in \mathbb R^\Lambda$ which are signed by adding two ghost vertices $\mathfrak g^+, \mathfrak g^-$, where vertices with $\mathsf h_x<0$ have an edge to $\mathfrak g^-$ and those with $\mathsf h_x > 0$ have an edge to $\mathfrak g^+$. Many of the definitions and properties carry over to this setting with straightforward albeit tedious modifications.
\end{Rem}

\subsubsection{Definition of the random cluster measure on a finite graph}

The random cluster representation is obtained in two steps. First, we condition on the absolute value field on the vertices and sample an Ising random cluster model (or FK-Ising model) with inhomogeneous weights depending on the absolute value field: this is the \emph{quenched} model. Second, we integrate out the absolute value field: this is the \emph{annealed} model. Let $\Lambda\subset \mathbb Z^d$ be finite. 

\hfill

\noindent \textbf{Quenched model.} Let $\mathsf a \in (\mathbb R^+)^{\overline \Lambda}$, $
\mathsf h \in (\mathbb R^+)^\Lambda$, and $\xi$ be a partition of $\partial^{\textup{ext}}\Lambda \cup \{ \mathfrak g\}$. Recall that $\mathsf{a}_\fg=1$. Let $k^\xi(\omega,\mathsf{a})$ denote the number of connected components in the graph with vertex set $\{x\in \overline\Lambda^\fg: \mathsf{a}_x\neq 0\}$ and edge set $\{xy: \omega_{xy}=1\}$ after identifying vertices in $\partial^{\textup{ext}}\Lambda\cup \{\fg\}$ in the same partition class of $\xi$.
For every $\beta \geq 0$, every edge $xy\in \overline E(\Lambda)$, and every $x\in \Lambda$, define 
\begin{equation}\label{eq:def p(beta,a)}
p(\beta, \mathsf a)_{xy}:=1-e^{-2\beta \mathsf a_x \mathsf a_y}, \qquad p(\beta, \mathsf h, \mathsf a)_{x\mathfrak g}:=1-e^{-2\beta \mathsf h_x \mathsf a_x}.
\end{equation}
Denote by $\phi^\xi_{\Lambda,\beta,\mathsf h, \mathsf a}$ the probability measure on $\omega\in \{0,1\}^{\overline E(\Lambda[\mathsf h])}$ defined by
\begin{equation} \label{def: random weight rc}
\phi^\xi_{\Lambda,\beta,\mathsf h, \mathsf a}[\omega]
=
\frac{1}{{Z}^{\xi}_{\Lambda,\beta,\mathsf h, \mathsf a}}\prod_{xy \in \overline E(\Lambda)} \left(\frac{p(\beta, \mathsf a)_{xy}}{1-p(\beta, \mathsf a)_{xy}}\right)^{\omega_{xy}} \prod_{x \in \Lambda} \left(\frac{p(\beta, \mathsf{h},\mathsf a)_{x\fg}}{1-p(\beta,\mathsf{h}, \mathsf a)_{x\fg}}\right)^{\omega_{x\fg}} \, 2^{k^\xi(\omega,\mathsf{a})},
\end{equation} 
where $Z^{\xi}_{\Lambda,\beta,\mathsf h, \mathsf a}$ is such that $\phi^\xi_{\Lambda,\beta,\mathsf{h},\mathsf{a}}$ is a probability measure. 

\hfill

\noindent\textbf{Random environment.} We now define a measure on the absolute-value field $\mathsf{a}$. Let $(\xi,\mathsf{b})$ be a boundary condition on $\Lambda$.
Define $\mu^{(\xi,\mathsf{b})}_{\Lambda, \beta, \mathsf h}$ to be the probability measure on $(\mathbb R^+)^{\overline \Lambda}$ defined by the density
\begin{equation} \label{def: annealed rc weight}
\mathrm{d} \mu^{(\xi,\mathsf{b})}_{\Lambda,\beta,\mathsf h}[\mathsf a] = 
\prod_{xy \in \overline E(\Lambda)} \sqrt{1-p(\beta,\mathsf a)_{xy}} \prod_{x \in \Lambda} \sqrt{1-p(\beta,\mathsf h, \mathsf a)_{x\mathfrak g}}\frac{Z^{\xi}_{\Lambda, \beta, \mathsf h, \mathsf a}}{Z^{(\xi,\mathsf{b})}_{\Lambda, \beta, \mathsf h}} \prod_{x \in \Lambda} \mathrm{d}\rho(\mathsf a_x) \prod_{x \in \partial^{\textup{ext}} \Lambda} \delta_{\mathsf{b}_x}(\mathsf{a}_x),
\end{equation}
where $Z^{(\xi,\mathsf{b})}_{\Lambda, \beta, \mathsf h}$ is such that $\mu_{\Lambda,\beta,\mathsf{h}}^{(\xi,\mathsf{b})}$ is a probability measure.

\begin{Rem}
The factors in \eqref{def: annealed rc weight} involving square roots arise when relating the partition function of the inhomogeneous FK-Ising model to the partition function of the inhomogeneous Ising spin model (see \cite{gunaratnam2025supercritical}). 
\end{Rem}

\hfill

\noindent\textbf{Annealed model.} We are now in a position to define our random cluster representation.

\begin{Def}[Random cluster representation]
The generalised random cluster representation on a finite subset $\Lambda \subset \mathbb Z^d$ with parameters $\beta\geq  0$ and $\mathsf h\in (\mathbb R^+)^\Lambda$, and boundary conditions $(\xi,\mathsf{b})$ is the probability measure $\Psi^{(\xi,\mathsf{b})}_{\Lambda,\beta,\mathsf h}$ on $ \{0,1\}^{\overline E(\Lambda[\mathsf h])}\times (\mathbb R^+)^{\overline \Lambda} $ defined by
\begin{equation}
\mathrm{d}\Psi^{(\xi,\mathsf{b})}_{\Lambda,\beta,\mathsf h}[(\omega,\mathsf a)] = \phi^\xi_{\Lambda, \beta, \mathsf h, \mathsf a}[\omega] \mathrm{d}\mu^{(\xi,\mathsf{b})}_{\Lambda, \beta, \mathsf h}[\mathsf a].
\end{equation}
We write $\Phi^{(\xi,\mathsf{b})}_{\Lambda,\beta,\mathsf{h}}$ for the distribution of the $\omega$-marginal of $\Psi^{(\xi,\mathsf{b})}_{\Lambda,\beta,\mathsf h}$.
\end{Def}

Unpacking the definitions, it is easy to see that
\begin{equation}\label{eq: random cluster explicit density}
\begin{aligned}
\mathrm{d}\Psi_{\Lambda,\beta,\mathsf{h}}^{(\xi,\mathsf{b})}[(\omega,\mathsf{a})]=&\frac{1}{Z^{(\xi,\mathsf{b})}_{\Lambda,\beta,\mathsf{h}}}\prod_{xy \in \overline{E}(\Lambda)} \sqrt{1-p(\beta,\mathsf a)_{xy}}\left(\frac{p(\beta,\mathsf a)_{xy}}{1-p(\beta,\mathsf a)_{xy}}\right)^{\omega_{xy}} \, \\
& \times \prod_{x\in \Lambda} \sqrt{1-p(\beta, \mathsf h, \mathsf a)_{x\fg}}\left(\frac{p(\beta, \mathsf{h}, \mathsf a)_{x\fg}}{1-p(\beta,\mathsf{h}, \mathsf a)_{x\fg}}\right)^{\omega_{x\fg}} 2^{k^\xi(\omega,\mathsf{a})} 
\prod_{x\in \Lambda}\mathrm{d} \rho(\mathsf{a}_x)\prod_{x \in \partial^{\textup{ext}} \Lambda} \delta_{\mathsf{b}_x}(\mathsf{a}_x).
\end{aligned}
\end{equation}

\hfill

\noindent \textbf{Free and wired boundary conditions.} The following two random cluster measures are of particular importance and will play the role of \emph{minimal} and \emph{maximal} boundary conditions for this representation, respectively.

 The \emph{free random cluster measure} $\Psi_{\Lambda,\beta, \mathsf h}^{0}$ is defined by taking $\xi=\{\{x\}: x\in \partial^{\mathrm{ext}} \Lambda\cup \{\fg\}\}$, i.e.,\ each partition class consists of a single vertex,
and $\mathsf{b}\equiv 0$. Similarly, we denote the associated quenched and absolute value measures by  $\phi^{0}_{\Lambda,\beta,\mathsf{h}, \mathsf a}$ and $\mu_{\Lambda,\beta,\mathsf{h}}^{0}$.  In the case $\mathsf{h}\equiv0$, we simply write $\Psi^0_{\Lambda,\beta}$. The measure $\Psi^0_{\Lambda,\beta}$ naturally arises when considering $\nu_{\Lambda,\beta}$, see the Edwards--Sokal coupling of Corollary \ref{cor: ES correlations}.

The \emph{wired random cluster measure} $\Psi^1_{\Lambda,\beta, \mathsf h}$ is defined in a different way by the density
\begin{equation}\label{eq: random cluster explicit density wired}
\begin{aligned}
\mathrm{d}\Psi^1_{\Lambda,\beta,\mathsf{h}}[(\omega, \mathsf a)]=&\frac{1}{Z^{1}_{\Lambda,\beta,\mathsf{h}}}\prod_{xy \in \overline{E}(\Lambda)} \sqrt{1-p(\beta,\mathsf a)_{xy}}\left(\frac{p(\beta,\mathsf a)_{xy}}{1-p(\beta,\mathsf a)_{xy}}\right)^{\omega_{xy}}
\\&\times \prod_{x\in \Lambda} \sqrt{1-p(\beta, \mathsf h, \mathsf a)_{x\fg}}\left(\frac{p(\beta, \mathsf{h}, \mathsf a)_{x\fg}}{1-p(\beta,\mathsf{h}, \mathsf a)_{x\fg}}\right)^{\omega_{x\fg}} 2^{k^w(\omega,\mathsf{a})} 
\prod_{x\in \Lambda}\mathrm{d} \rho(\mathsf{a}_x) \prod_{x \in \partial^{\textup{ext}} \Lambda} \mathrm{d}\zeta(\mathsf a_x),
\end{aligned}
\end{equation}
where $w := \{ \partial^{\textup{ext}} \Lambda\cup \{\fg\}\}$, i.e.\ the partition where all vertices in $\partial^{\textup{ext}} \Lambda\cup \{\fg\}$ are in the same class and $\zeta$ is given by Proposition \ref{prop:reg_plus}. We denote the associated quenched and absolute value measures by $\phi^{1}_{\Lambda,\beta, \mathsf h, \mathsf a}$ and $\mu_{\Lambda,\beta, \mathsf h}^{1}$ respectively. Note also that
\begin{equation} \label{eq: psi-phi-mu plus}
    \mathrm d \Psi^1_{\Lambda, \beta, \mathsf h}[(\omega,\mathsf a)] = \phi^w_{\Lambda,\beta,\mathsf h, \mathsf a}[\omega] \mathrm d\mu^1_{\Lambda,\beta,\mathsf h}[\mathsf{a}].
\end{equation} 
In the case $\mathsf h\equiv 0$, we simply write $\Psi^1_{\Lambda,\beta}$. The measure $\Psi^1_{\Lambda,\beta}$ naturally arises when considering $\nu^+_{\Lambda,\beta}$, see the Edwards--Sokal coupling of Corollary \ref{cor: ES correlations for plus}.

\subsubsection{Properties of the random cluster model}\label{subsubsection: random cluster}

We now collect some fundamental properties of the random cluster representation that we will use throughout the paper. 

\hfill

\noindent\textbf{Law of the absolute value fields.}
Our random cluster measures are defined in such a way that they naturally relate to the spin measures of interest. 
Indeed, by a straightforward adaptation of the computations in \cite[Proposition~4.6]{gunaratnam2025supercritical}, we obtain that the law of $\mathsf{a}$ under $\Psi^0_{\Lambda,\beta,\mathsf{h}}$ (resp. under $\Psi^1_{\Lambda,\beta,\mathsf{h}}$) is the law of $|\tau|$ under $\nu_{\Lambda,\beta,\mathsf{h}}$ (resp. $\nu_{\Lambda,\beta,\mathsf{h}}^+$).

\hfill

\noindent\textbf{Edwards--Sokal coupling.} Let $\Lambda\subset \mathbb Z^d$ be finite. For every $t\in \mathbb R$, let $\textup{sgn}(t) = \mathds 1_{t> 0} - \mathds 1_{t<0}$. Recall that conditionally on $|\tau|$, on the set of vertices $x$ where $|\tau_x|\neq 0$, ${\rm sgn}(\tau)$  is distributed as an Ising model. Similarly, $\omega$ conditionally on $\mathsf a$ is distributed as an FK-Ising model in a random environment. Furthermore, by the previous observations, these random environments (the absolute value fields) coincide. Therefore, the standard Edwards-Sokal coupling between the Ising and FK-Ising models \cite[Chapter 1]{Grimmett2006RCM} yields the following couplings. We first describe the coupling between $\nu_{\Lambda,\beta,\mathsf{h}}$ and $\Psi^0_{\Lambda,\beta,\mathsf{h}}$. 

\medskip

\noindent\textit{From percolation to spin model.} 
Given a pair $(\omega,\mathsf{a})\sim \Psi^0_{\Lambda,\beta,\mathsf{h}}$, we can sample a field $\tau\sim \langle \cdot \rangle_{\Lambda,\beta,\mathsf{h}}$ as follows. Below, recall we set $\mathsf{a}_{\mathfrak{g}}=1$.
\begin{enumerate}
\item[$(i)$] For every $x\in \Lambda$ such that $\mathsf{a}_x=0$, set $\sigma_x=0$.
\item[$(ii)$] Consider a sequence of independent random variables $\sigma_{\mathcal{C}}\in \{\pm 1\}$, indexed by the connected components $\mathcal{C}$ of the graph with vertex set $\{x\in \Lambda^\fg: \mathsf{a}_x\neq 0\}$ and edge set $\{xy: \omega_{xy}=1\}$, such that $\mathbb{P}[\sigma_{\mathcal{C}}=1]=1$ if $\mathcal{C}$ is the cluster of $\fg$, and $\mathbb{P}[\sigma_{\mathcal{C}}=1]=\mathbb P[\sigma_{\mathcal{C}}=-1]=1/2$ otherwise. Given $x$, we write $\sigma_x = \sigma_{\mathcal C}\mathbbm{1}_{x \in \mathcal C}$.
\item[$(iii)$] For every $x\in \Lambda$, set $\tau_x = \mathsf{a}_x \sigma_{x}$.
\end{enumerate}

\medskip 
\noindent \textit{From spin model to percolation.} Conversely, given a field $\tau\sim \langle \cdot \rangle_{\Lambda,\beta,\mathsf{h}}$, we can sample a pair $(\omega, \mathsf{a})\sim \Psi^0_{\Lambda,\beta,\mathsf{h}}$ as follows. 
\begin{enumerate}
\item[$(i)$] Set $\mathsf{a}_x:=|\tau_x|$ for $x\in\Lambda$, and $\mathsf{a}_x:=0$ 
for $x\in\partial^{\textup{ext}}\Lambda$.
\item[$(ii)$] For each edge $xy\in \overline{E}(\Lambda[\mathsf{h}])$ such that $\textup{sgn}(\tau_x) \neq \textup{sgn}(\tau_y)$ or $\textup{sgn}(\tau_x)=\textup{sgn}(\tau_y)=0$, set $\omega_{xy}=0$. 
\item[$(iii)$] For each edge $xy\in E(\Lambda)$ such that $\textup{sgn}(\tau_x) = \textup{sgn}(\tau_y) \in \{\pm 1\}$, let $\omega_{xy}=1$ with probability equal to $p(\beta, \mathsf{a})_{xy}$, independently of  the other edges.
\item[$(iv)$] For each vertex $x\in \Lambda$ such that $\textup{sgn}(\tau_x) = 1$, let $\omega_{x\mathfrak g}=1$
with probability equal to $p(\beta,\mathsf{h}, \mathsf{a})_{x\fg}$, independently of the other edges.
\end{enumerate}

As a direct consequence of this coupling (see \cite[Theorem 1.16]{Grimmett2006RCM}), we obtain the following result for $\Psi^0_{\Lambda,\beta,\mathsf{h}}$.
Below, if $x,y\in \Lambda^\fg$, we write $\{x\connect{} y\}$ to denote the event that $x$ and $y$ lie in the same connected component of $\omega$.

\begin{Coro} \label{cor: ES correlations}
Let $\Lambda\subset \mathbb Z^d$ be finite, $\beta\geq 0$ and $\mathsf h \in (\mathbb R^+)^\Lambda$. For every $x,y\in \Lambda$, 
\begin{align}
\label{eq: es cor: 1}
\langle \tau_x \rangle_{\Lambda,\beta,\mathsf{h}}
&=
\Psi^0_{\Lambda,\beta,\mathsf{h}}[\mathsf{a}_x \mathds{1}\{x\longleftrightarrow \fg\}], 
\\
\label{eq: es cor: 2}
\langle \textup{sgn}(\tau_x) \rangle_{\Lambda,\beta,\mathsf{h}}
&=
\Psi^0_{\Lambda,\beta,\mathsf{h}}[x\longleftrightarrow \fg], 
\\ \label{eq: es cor: 3}
\langle \tau_x \tau_y \rangle_{\Lambda,\beta,\mathsf{h}}
&=
\Psi^0_{\Lambda,\beta,\mathsf{h}}[\mathsf{a}_x \mathsf{a}_y \mathds{1}\{x\longleftrightarrow y\}], 
\\ \label{eq: es cor: 4}
\langle \textup{sgn}(\tau_x) \textup{sgn}(\tau_y) \rangle_{\Lambda,\beta,\mathsf{h}}
&=
\Psi^0_{\Lambda,\beta,\mathsf{h}}[x\longleftrightarrow y]. 
\end{align}   
\end{Coro}

The above procedure extends mutatis mutandis to couple the measures $\nu^+_{\Lambda,\beta,\mathsf{h}}$ and $\Psi^1_{\Lambda,\beta,\mathsf{h}}$, the only difference is that we take the convention that $\sigma_\mathcal{C}=1$ where $\mathcal{C}$ is the cluster of $\partial^{\textup{ext}}\Lambda \cup \{\mathfrak g\}$.

\begin{Coro} \label{cor: ES correlations for plus}
Let $\Lambda\subset \mathbb Z^d$ be finite, $\beta\geq 0$ and $\mathsf h \in (\mathbb R^+)^\Lambda$. For every $x,y\in \Lambda$, 
\begin{align}
\label{eq: es cor: 1 plus}
\langle \tau_x \rangle_{\Lambda,\beta,\mathsf{h}}^+
&=
\Psi^1_{\Lambda,\beta,\mathsf{h}}[\mathsf{a}_x \mathds{1}\{x\connect{\:} \mathfrak g \}], 
\\
\label{eq: es cor: 2 plus}
\langle \textup{sgn}(\tau_x) \rangle_{\Lambda,\beta,\mathsf{h}}^+
&=
\Psi^1_{\Lambda,\beta,\mathsf{h}}[x\connect{\:}\mathfrak g], 
\\ \label{eq: es cor: 3 plus}
\langle \tau_x \tau_y \rangle_{\Lambda,\beta,\mathsf{h}}^+
&=
\Psi^1_{\Lambda,\beta,\mathsf{h}}[\mathsf{a}_x \mathsf{a}_y \mathds{1}\{x\connect{\:} y\}], 
\\ \label{eq: es cor: 4 plus}
\langle \textup{sgn}(\tau_x) \textup{sgn}(\tau_y) \rangle_{\Lambda,\beta,\mathsf{h}}^+
&=
\Psi^1_{\Lambda,\beta,\mathsf{h}}[x\connect{\:} y]. 
\end{align} 
\end{Coro}

\hfill

\noindent\textbf{Domain Markov property and correlation inequalities.} We now state further useful properties of the random cluster measures, including positive association and domain Markov properties.

First, we treat properties with deterministic boundary conditions. We fix $\Lambda \subset \mathbb Z^d$ finite, $\beta \geq 0$, $\mathsf h \in (\mathbb R^+)^\Lambda$, and boundary conditions consisting of $\xi$ a partition of $\partial^{\textup{ext}} \Lambda\cup \{\fg\}$ and $\mathsf b \in (\mathbb R^+)^{\partial^{\textup{ext}} \Lambda}$. We will need two conventions. For boundary conditions $(\xi_1,\mathsf{b}_1), (\xi_2,\mathsf{b}_2)$ on $\Lambda$, we write $(\xi_1,\mathsf{b}_1) \leq (\xi_2,\mathsf{b}_2)$ if every partition class of $\xi_1$ is contained in a partition class of $\xi_2$, and $\mathsf{b}_1\leq \mathsf{b}_2$. For two measures $\mu_1$ and $\mu_2$, we write $\mu_1 \preccurlyeq \mu_2$ if $\mu_1(F)\leq \mu_2(F)$ for every increasing function $F$.
\begin{Prop}\label{prop: random cluster properties}
The following properties hold.
\begin{enumerate}
\item[$\bullet$]\textup{(Domain Markov Property)} For every $\Lambda' \subset \Lambda$, every $\theta\in \{0,1\}^{\overline E(\Lambda[\mathsf{h}]) \setminus \overline E(\Lambda'[\mathsf{h}])}$, and every $\mathsf{s}\in (\mathbb R^+)^{\overline{\Lambda}\setminus\Lambda'}$ satisfying $\mathsf{s}_{|\partial^{\textup{ext}} \Lambda}=\mathsf{b}$,
\begin{equation}
\Psi^{(\xi,\mathsf{b})}_{\Lambda,\beta,\mathsf{h}}[(\omega,\mathsf{a}) \mid (\omega|_{\overline E(\Lambda[\mathsf{h}])\setminus \overline E(\Lambda'[\mathsf{h}])},\mathsf{a}|_{\overline{\Lambda}\setminus \Lambda'})=(\theta,\mathsf{s})]=\Psi^{(\xi^{\theta},\mathsf{s})}_{\Lambda',\beta,\mathsf{h}_{|\Lambda'}}[(\omega,\mathsf{a})],
\end{equation}
where $\xi^{\theta}$ is the partition of $(\partial^{\mathrm{ext}} \Lambda')^{\fg}$ where $x,y$ are in the same partition class if and only if they are in the same connected component in the graph obtained from  $((\overline{\Lambda}\setminus \Lambda')^{\mathfrak g}, \{ e \in \overline E(\Lambda[\mathsf h])\setminus \overline E(\Lambda'[\mathsf h]):\theta_e=1\})$ by identifying the vertices in the elements of the partition $\xi$.
\item[$\bullet$]\textup{(Absolute value FKG inequality)} For every  $A \subset \Lambda$, $\eta \in (\mathbb R^+)^A$, and all increasing and square integrable functions $F,G: (\mathbb{R}^+)^{\overline{\Lambda} \setminus A}\to\mathbb{R}$, 
\begin{equation} \label{eq: conditional abs fkg}
\mu_{\Lambda,\beta,\mathsf{h}}^{(\xi,\mathsf{b})}[F\cdot G \mid \mathsf a|_A=\eta]\geq \mu_{\Lambda,\beta,\mathsf{h}}^{(\xi,\mathsf{b})}[F \mid \mathsf a|_A=\eta]\mu_{\Lambda,\beta,\mathsf{h}}^{(\xi,\mathsf{b})}[G \mid \mathsf a|_A=\eta].
\end{equation}
\item[$\bullet$]\textup{(FKG inequality)} For every $A \subset \Lambda$, $\eta \in (\mathbb R^+)^A$, and every pair of increasing and square-integrable functions $F,G: \{0,1\}^{\overline E(\Lambda[\mathsf{h}])}\times (\mathbb R^+)^{\overline{\Lambda}}\to\mathbb{R}$,
we have
\begin{equation}
\Psi_{\Lambda,\beta,\mathsf{h}}^{(\xi,\mathsf{b})}[F\cdot G \mid \mathsf a|_A=\eta]
\geq
\Psi_{\Lambda,\beta,\mathsf{h}}^{(\xi,\mathsf{b})}[F \mid \mathsf a|_A=\eta] \Psi_{\Lambda,\beta,\mathsf{h}}^{(\xi,\mathsf{b})}[G \mid \mathsf a|_A=\eta].
\end{equation}
\item[$\bullet$]\textup{(Monotonicity in parameters)} Let $\Lambda\subset \mathbb{Z}^d$ be finite and $A\subset \Lambda$. For every boundary conditions $(\xi_1,\mathsf{b}_1) \leq (\xi_2,\mathsf{b}_2)$ on $\Lambda$, every $\eta_1,\eta_2\in (\mathbb{R}^+)^{A}$ such that $\eta_1\leq \eta_2$, every external magnetic fields $\mathsf{h}_1,\mathsf{h}_2\in (\mathbb R^+)^\Lambda$ with $\mathsf h_1\leq \mathsf h_2$, and every $0\leq\beta_1\leq \beta_2$ we have
\begin{equation} \label{eq: conditional stochastic domiation}
\Psi_{\Lambda,\beta_1,\mathsf h_1}^{(\xi_1,\mathsf{b}_1)}[\:\cdot \mid \mathsf{a}|_A=\eta_1] \preccurlyeq \Psi_{\Lambda,\beta_2,\mathsf h_2}^{(\xi_2,\mathsf{b}_2)}[\:\cdot \mid \mathsf{a}|_A=\eta_2].
\end{equation}
In particular, 
\begin{equation}
\Psi^0_{\Lambda,\beta,\mathsf h} \preccurlyeq \Psi_{\Lambda,\beta,\mathsf h}^{(\xi,\mathsf{b})},
\end{equation}
for every boundary condition $(\xi,\mathsf{b})$ on $\Lambda$ and every $\mathsf{h}\in(\mathbb R^+)^\Lambda$. 

\item[$\bullet$]\textup{(Monotonicity in domain)} 
Let $\Lambda_1\subset \Lambda_2\subset \mathbb{Z}^d$ be finite and $\mathsf h_1\in (\mathbb{R}^+)^{\Lambda_1},\mathsf h_2\in (\mathbb{R}^+)^{\Lambda_2}$ such that $\mathsf h_1(x)\leq \mathsf h_2(x)$ for all $x\in \Lambda_1$. For every $\beta\geq 0$, we have
\begin{equation}
\Psi_{\Lambda_1,\beta,\mathsf{h}_1}^{0} \preccurlyeq \Psi_{\Lambda_2,\beta,\mathsf{h}_2}^{0}.
\end{equation}    
\end{enumerate}
\end{Prop}

\begin{proof}
    The proof follows as in \cite[Sections~4.2 and 4.3]{gunaratnam2025supercritical} (see Remark 4.3 therein). 
\end{proof}

We now consider the corresponding properties for the wired measure $\Psi^1_{\Lambda,\beta,\mathsf{h}}$. Similarly to $\nu_{\Lambda,\beta}^+$, $\Psi_{\Lambda,\beta,\mathsf{h}}^1$ inherits these properties by viewing it as a free measure with inhomogeneous single-site distributions for the absolute value field, see Remark \ref{rem:plus}.

\begin{Prop}\label{prop: wired random cluster properties} 
The following properties hold. Fix $\Lambda\subset \mathbb Z^d$ and $\beta\geq 0$.
\begin{enumerate}
\item[$\bullet$]\textup{(Domain Markov Property)} For every $\Lambda' \subset \Lambda$, every $\theta\in \{0,1\}^{\overline E(\Lambda[\mathsf{h}]) \setminus \overline E(\Lambda'[\mathsf{h}])}$, and every $\mathsf{s}\in (\mathbb R^+)^{\overline{\Lambda}\setminus\Lambda'}$ satisfying $\mathsf{s}_{|\partial^{\textup{ext}} \Lambda}=\mathsf{b}$,
\begin{equation}
\Psi^{1}_{\Lambda,\beta,\mathsf{h}}[(\omega,\mathsf{a}) \mid (\omega|_{\overline E(\Lambda[\mathsf{h}])\setminus \overline E(\Lambda'[\mathsf{h}])},\mathsf{a}|_{\overline{\Lambda}\setminus \Lambda'})=(\theta,\mathsf{s})]=\Psi^{(w^\theta,\mathsf{s})}_{\Lambda',\beta,\mathsf{h}_{|\Lambda'}}[(\omega,\mathsf{a})],
\end{equation}
where $w^{\theta}$ is the partition of $(\partial^{\mathrm{ext}} \Lambda')^{\fg}$ where $x,y$ are in the same partition class if and only if they are in the same connected component in the graph obtained from  $((\overline{\Lambda}\setminus \Lambda')^{\mathfrak g}, \{ e \in \overline E(\Lambda[\mathsf h])\setminus \overline E(\Lambda'[\mathsf h]):\theta_e=1\})$ by identifying the vertices in the elements of the wired partition $w$. 

\item[$\bullet$] \textup{(Absolute value FKG inequality)} For every  $A \subset \overline \Lambda$, $\eta \in (\mathbb R^+)^A$, and all increasing and square integrable functions $F,G: (\mathbb{R}^+)^{\overline{\Lambda} \setminus A}\to\mathbb{R}$, 
\begin{equation} \label{eq: conditional abs fkg plus}
\mu_{\Lambda,\beta,\mathsf{h}}^{1}[F\cdot G \mid \mathsf a|_A=\eta]\geq \mu_{\Lambda,\beta,\mathsf{h}}^{1}[F \mid \mathsf a|_A=\eta]\mu_{\Lambda,\beta,\mathsf{h}}^{1}[G \mid \mathsf a|_A=\eta].
\end{equation}
    \item[$\bullet$] \textup{(FKG inequality)} For every $A \subset \overline\Lambda$, $\eta \in (\mathbb R^+)^A$, and all increasing and square-integrable functions $F,G: \{0,1\}^{\overline E(\Lambda[\mathsf h])}\times (\mathbb R^+)^{\overline{\Lambda}}\to\mathbb{R}$,
we have
\begin{equation} \label{eq: fkg plus}
\Psi_{\Lambda,\beta,\mathsf{h}}^{1}[F\cdot G \mid \mathsf a|_A=\eta]
\geq
\Psi_{\Lambda,\beta,\mathsf{h}}^{1}[F \mid \mathsf a|_A=\eta] \Psi_{\Lambda,\beta,\mathsf{h}}^{1}[G \mid \mathsf a|_A=\eta].
\end{equation}
\item[$\bullet$] \textup{(Comparison to the free measure)} One has
    \begin{equation}\label{eq:phi1 dominates psi0 finite volume}
    \Psi^0_{\Lambda,\beta,\mathsf{h}}\preccurlyeq \Psi^{1}_{\Lambda,\beta,\mathsf{h}}.
\end{equation}
\item[$\bullet$]\textup{(Monotonicity in parameters)} For every $A\subset \overline\Lambda$, every $\eta_1,\eta_2\in (\mathbb{R}^+)^{A}$ such that $\eta_1\leq \eta_2$, every external magnetic fields $\mathsf{h}_1,\mathsf{h}_2\in (\mathbb R^+)^\Lambda$ with $\mathsf h_1\leq \mathsf h_2$, and every $0\leq\beta_1\leq \beta_2$ we have
\begin{equation} \label{eq: conditional stochastic domiation wired}
\Psi_{\Lambda,\beta_1,\mathsf h_1}^1[\:\cdot \mid \mathsf{a}|_A=\eta_1] \preccurlyeq \Psi_{\Lambda,\beta_2,\mathsf h_2}^1[\:\cdot \mid \mathsf{a}|_A=\eta_2].
\end{equation}
    \item[$\bullet$] \textup{(Monotonicity in domain with gap)} Let $\beta,h \geq 0$. There exists $r=r(\beta,h)>0$ such that, for every $\Lambda_1 \subset \Lambda_2 \subset \mathbb Z^d$ finite and such that $\mathrm{d}(\Lambda_1, \partial \Lambda_2) \geq r$, and every $\mathsf h_1 \in (\mathbb R^+)^{\Lambda_1}$ and $\mathsf h_2 \in (\mathbb R^+)^{\Lambda_2}$ such that $\mathsf h_1 \geq (\mathsf h_2)_{|_{\Lambda_1}}$, $\mathsf h_1,\mathsf{h}_2\leq h$, we have that
    \begin{equation}\label{eq: monot in domain wired FK}
        \Psi^1_{\Lambda_1, \beta, \mathsf h_1} \succcurlyeq \Psi^1_{\Lambda_2, \beta, \mathsf h_2}. 
    \end{equation}
\end{enumerate}
\end{Prop}
\begin{proof}
The proof of the domain Markov property does not depend on the choice of single-site measures satisfying assumption $(ii)$ of Definition~\ref{def:single site}. Similarly, the proof of the absolute value FKG inequality for the random cluster measure with deterministic boundary conditions \eqref{eq: conditional abs fkg} also applies for any collection of single-site measures satisfying the assumption $(ii)$ of Definition~\ref{def:single site}. Hence, the same argument applies to establish \eqref{eq: conditional abs fkg plus}.

We now turn to \eqref{eq: fkg plus}. By \eqref{eq: psi-phi-mu plus}, we may condition on absolute value fields $\mathsf a$ satisfying $\mathsf a|_A=\eta$, and use the conditional FKG inequality for deterministic boundary conditions to obtain
\begin{equation}
    \Psi^1_{\Lambda,\beta,\mathsf h}[F \cdot G \mid \mathsf a] = \phi^w_{\Lambda, \beta, \mathsf h, \mathsf a}[ F(\cdot, \mathsf a) G(\cdot, \mathsf a) ] \geq \phi^w_{\Lambda, \beta, \mathsf h, \mathsf a}[F(\cdot,\mathsf a)] \phi^w_{\Lambda, \beta, \mathsf h, \mathsf a}[G(\cdot,\mathsf a)]. 
\end{equation}
The function $\mathsf a \mapsto \phi^w_{\Lambda, \beta, \mathsf h, \mathsf a}[F(\cdot,\mathsf a)]$ is increasing thanks to monotonicity in the coupling for deterministic random cluster measures and the fact that $F$ is jointly increasing in $(\omega, \mathsf a)$. We then obtain \eqref{eq: fkg plus} by averaging with respect to $\mu^1_{\Lambda,\beta, \mathsf h}[\:\cdot \mid \mathsf a|_A=\eta]$ and using the conditional absolute value FKG \eqref{eq: conditional abs fkg plus}. 

The stochastic domination of the free by the wired measure \eqref{eq:phi1 dominates psi0 finite volume} follows from \eqref{eq: conditional stochastic domiation} thanks to the domain Markov property, where we condition on the partition and absolute value in $\overline{\Lambda}\setminus \Lambda$.  

For \eqref{eq: conditional stochastic domiation wired}, note that \eqref{eq: conditional stochastic domiation} applies to measures with inhomogeneous single-site measures, in particular to the wired measure.

Finally, for \eqref{eq: monot in domain wired FK} we argue as follows. For any increasing $\overline \Lambda_1$-measurable function $F$, and any $\eta_1\in (\mathbb R^+)^{\overline \Lambda_1}$, $\eta_2\in (\mathbb R^+)^{\overline \Lambda_2}$ with $\eta_1\geq (\eta_2)_{|\overline\Lambda_1}$,
\begin{equation}
    \Psi^1_{\Lambda_1,\beta,\mathsf{h}_1}[F\: | \: \mathsf{a}=\eta_1]\geq \Psi^1_{\Lambda_1,\beta,(\mathsf{h}_2)_{|\Lambda_1}}[F\: | \: \mathsf{a}=\eta_1]\geq \Psi^1_{\Lambda_2,\beta,\mathsf{h}_2}[F\: | \:\mathsf{a}=\eta_2],
\end{equation}
where, in the second inequality, we used the monotonicity in the volume of the standard random cluster measure with wired boundary conditions. Now, integrating and using the stochastic monotonicity of the absolute value field of Proposition \ref{prop:monotonicity in plus} (recall that the law of $\mathsf{a}$ under $\Psi^1$ is exactly $\mu^1$ as seen in Section \ref{subsubsection: random cluster}) we get
\begin{align}
\begin{aligned}
     \Psi^1_{\Lambda_1,\beta,\mathsf{h}_1}[F]=\Psi^1_{\Lambda_1,\beta,\mathsf{h}_1}\Big[\Psi^1_{\Lambda_1,\beta,\mathsf{h}_1}[F\: | \:\mathsf{a}=\eta]\Big]&\geq \Psi^1_{\Lambda_1,\beta,\mathsf{h}_1}\Big[\Psi^1_{\Lambda_2,\beta,\mathsf{h}_2}[F\: | \:\mathsf{a}_{|\Lambda_1}=\eta]\Big]\\&\geq \Psi^1_{\Lambda_2,\beta,\mathsf{h}_2}\Big[\Psi^1_{\Lambda_2,\beta,\mathsf{h}_2}[F\: | \:\mathsf{a}_{|\Lambda_1}=\eta]\Big]  = \Psi^1_{\Lambda_2,\beta,\mathsf{h}_2}[F],
\end{aligned}
\end{align}
as desired.
\end{proof}

\subsubsection{Infinite volume measures}

We briefly turn to infinite volume random cluster measures. For compactness, let us assume that $\mathsf{h}\equiv 0$. There exist probability measures $\Psi^0_{\beta}$ and $\Psi^1_\beta$, respectively called the infinite volume free and wired random cluster measures, such that the following convergence holds: for every local event $A \subset \{0,1\}^{E(\mathbb Z^d)}\times (\mathbb R^+)^{\mathbb Z^d}$,
\begin{equation}
\lim_{\Lambda \nearrow \mathbb Z^d} \Psi^0_{\Lambda,\beta}[A] = \Psi^0_\beta[A], \quad \lim_{\Lambda\nearrow \mathbb Z^d} \Psi^1_{\Lambda,\beta}[A] = \Psi^1_{\beta}[A]. 
\end{equation}
This convergence is a consequence of the convergence of the spin measures $\nu^0_{\Lambda,\beta} \rightarrow \nu^0_\beta$ and $\nu^+_{\Lambda,\beta} \rightarrow \nu^+_\beta$, and the Edwards--Sokal coupling. As a result, we also have analogous couplings between $\Psi^0_\beta$ and $\nu^0_\beta$, and $\Psi^1_\beta$ and $\nu^+_\beta$, respectively. Moreover, thanks to \eqref{eq:phi1 dominates psi0 finite volume}, we obtain
\begin{equation}\label{eq:psi1 dominates psi0}
    \Psi^0_{\beta}\preccurlyeq \Psi^{1}_{\beta}.
\end{equation}

\subsubsection{Finite energy and sprinkling property}

The random cluster measures we consider have a finite-energy property. Recall that the edge-marginal of $\Psi^0_{\Lambda,\beta}$ is denoted $\Phi^0_{\Lambda,\beta}$.
\begin{Prop}[Finite-energy property] \label{prop:finite-energy}
Let $\beta > 0$. There exists $\varepsilon > 0$ such that the following holds. Let $\Lambda \subset \mathbb Z^d$ and $e=xy \in E(\Lambda)$. For every $\omega' \in \{0,1\}^{\overline E(\Lambda)\setminus \{e\}}$,
\begin{equation} \label{eq: finite energy phi}
\varepsilon \leq \Phi^0_{\Lambda,\beta}[\omega_{xy}=1 \mid \omega|_{\overline E(\Lambda)\setminus \{xy\}} = \omega']\leq 1-\varepsilon.
\end{equation}
\end{Prop}

\begin{proof}
    We prove insertion-tolerance, corresponding to the first inequality  in \eqref{eq: finite energy phi}. The proof of deletion-tolerance, the second inequality in \eqref{eq: finite energy phi}, follows by similar arguments. Let us fix $\Lambda \subset \mathbb Z^d$ and fix $xy \in E(\Lambda)$. For convenience, let us write $\mathcal N(x,y):= \{ u \in \Lambda \setminus \{x,y\} : u \sim x \text{ or } u \sim y \}$. 
    
    Let $\omega' \in \{0,1\}^{\overline E(\Lambda)\setminus \{xy\}}$ and let $\mathsf a' \in (\mathbb R^+)^{\Lambda\setminus\{x,y\}}$ be compatible with $\omega'$ in the sense that $\omega'_{uv}=1$ implies that $\mathsf a'_u\mathsf a'_v > 0$. By the domain Markov property, 
        \begin{equation}
        \Psi^0_{\Lambda,\beta}[\omega_{xy}=1 \mid \mathsf a|_{\Lambda \setminus \{x,y\}} = \mathsf a', \omega|_{\overline E(\Lambda)\setminus \{xy\}}=\omega']
        =\Psi_{\{x,y\},\beta}^{(\xi(\omega'),\mathsf{a}'_{|\partial^{\textup{ext}}\{x,y\}})}[\omega_{xy}=1 \mid \mathcal{E}(\omega')]=:c_1(\omega',\mathsf a'),
    \end{equation}
    where $\xi(\omega')$ denotes the partition of $\partial^{\textup{ext}}\{x,y\}$ induced by $\omega'$, and where
    \begin{equation}
        \mathcal{E}(\omega')=\{\omega_e=\omega_e', \: \forall e\in \partial_E \{x,y\}\}.
    \end{equation}
    We note that $c_1(\omega',\mathsf a')>0$ since $\rho(\mathbb R\setminus \{0\})>0$, and that $c_1(\omega',\cdot)$ is defined on 
    \[\bigtimes_{\mathcal{V}(\omega';x,y)} (0,\infty)\bigtimes_{\mathcal{N}(x,y)\setminus \mathcal{V}(\omega';x,y)}[0,\infty),\] 
    where $\mathcal V(\omega';x,y) := \{ u \in \mathcal N(x,y) : \exists u'\sim u : \omega_{uu'}'=1 \}$ is the set of vertices $u$ neighbouring $x$ and $y$ that must have $\mathsf a'_u>0$. Let us consider any possible realisation of this set, $\mathcal V$. For $\delta > 0$, consider the event $\mathcal A(\mathcal V, \delta)$ such that $\mathsf a_u' \in [\delta,1/\delta]$ for every $u \in \mathcal V$ and $\mathsf a'_u \in [0,1/\delta]$ for every $u \in \mathcal N(x,y)\setminus \mathcal V$. By conditional regularity (see Lemma \ref{lem:conditional regularity} and Remark \ref{remark: conditional reg phi} below), there exists $c_2(\mathcal V)>0$ such that, for every $\delta=\delta(\mathcal V) > 0$ sufficiently small and for every $\omega'$ such that $\mathcal V(\omega';x,y)=\mathcal V$,
    \begin{equation}
        \Psi^0_{\Lambda,\beta}\Big[ \mathcal A(\mathcal V, \delta) \Big| \:\omega|_{\overline E(\Lambda)\setminus \{xy\}} = \omega'\Big] > c_2(\mathcal V).
    \end{equation}
    Note that we can take $\delta$ and $c_2$ uniform over $\mathcal V$ since there are only finitely many possibilities. Furthermore, since the function $c_1(\omega',\cdot)$ is continuous on the compact set $\mathcal A(\mathcal V(\omega';x,y),\delta)$ and there are finitely many possibilities for $\xi(\omega')$, we have that it is bounded from below by some $c_1(\mathcal V)>0$. This gives the first inequality in \eqref{eq: finite energy phi} and concludes the proof.
    \end{proof}

\begin{Rem}
Recall from Remark \ref{rem:plus} that the wired measure is a free measure with inhomogeneous single-site measure at the boundary. Therefore, the finite-energy property holds for the edge-marginal of $\Psi^1_{\Lambda,\beta}$. It also applies to the half-space measures considered in Section \ref{sec:Gibbs_FK}.
\end{Rem}

The proof of Proposition \ref{prop:finite-energy} relies on a conditional regularity statement for the measures $\Phi^0_{\Lambda,\beta}$. We state a slightly stronger conditional regularity statement that holds for a measure with sprinkling. It is natural to view  sprinkled measures, obtained by considering a union of two percolations, on a multigraph where each edge is duplicated. More precisely, consider the multigraph $\Lambda[\beta,\beta']$ with vertex set $\Lambda$ and edge set $\{ xy(\beta-\beta') : xy \in E(\Lambda)\} \sqcup \{xy(\beta') : xy \in E(\Lambda) \}$.

We consider\footnote{Note that the analogous measure considered in \cite{gunaratnam2025supercritical} was denoted $\Psi^0_{\Lambda,\beta,\beta'}$.} a probability measure $\Upsilon^0_{\Lambda,\beta,\beta'}$ on $\{0,1\}^{E(\Lambda[\beta,\beta'])}\times (\mathbb R^+)^\Lambda $ defined by
\begin{equation}
\mathrm{d}\Upsilon^0_{\Lambda,\beta,\beta'}[(\omega,\mathsf a)] =  \gamma_{\Lambda,\beta,\beta',\mathsf a}^{0}[\omega] \mathrm{d}\mu^0_{\Lambda,\beta'}[\mathsf a].
\end{equation}
Above, $\gamma_{\Lambda,\beta,\beta',\mathsf a}^0$ is a percolation measure on $\Lambda[\beta,\beta']$, where the $xy(\beta')$ edges are sampled according to $\phi^0_{\Lambda,\beta',\mathsf{a}}$ and the $xy(\beta-\beta')$ edges are sampled according to an independent inhomogeneous Bernoulli percolation $\textup{Ber}(\varepsilon(\beta,\beta', \mathsf a))$ with edge-parameters
\begin{equation}
\varepsilon(\beta,\beta',\mathsf a)_{xy} = \tanh ((\beta -\beta')\mathsf a_x\mathsf a_y).
\end{equation}

\begin{Lem}[{\hspace{1pt}\cite[Theorem~4.3, Remark 4.4]{PV26}}]\label{lem:conditional regularity}
Let $0\leq\beta'\leq\beta$. There exist $A,B>0$ such that, for every $\Lambda' \subset \Lambda$ and for every $\omega$, the restriction of $\Upsilon^0_{\Lambda,\beta,\beta'}$ to $\Lambda'[\beta,\beta']$, which we denote $\Upsilon_{\Lambda,\beta,\beta'}^{0, (\Lambda')}$, has a density with respect to $\prod_{x \in \Lambda'} \mathrm{d}\rho(\mathsf a_x)$ satisfying
 \begin{equation}
\frac{\mathrm{d}\Upsilon_{\Lambda,\beta,\beta'}^{0, (\Lambda')}[ \mathsf a \mid \omega]}{\prod_{x\in \Lambda'}\mathrm{d}\rho(\mathsf a_x)} \leq \prod_{x \in \Lambda'}e^{A \mathsf a_x^2 + B} .
\end{equation}
\end{Lem}

\begin{Rem} \label{remark: conditional reg phi}
The edge marginal  $\Phi^0_{\Lambda,\beta}$ can be realised as the projection of a random cluster measure on $\Lambda[\beta,\beta']$ with parameters $p(\beta-\beta',\mathsf a)_{xy}$ on $xy(\beta-\beta')$ and $p(\beta', \mathsf a)_{xy}$ on $xy(\beta')$. By projection, we stress that we mean union of the percolation configurations on the duplicated edges. This coincides with the edge marginal of $\Upsilon^0_{\Lambda,\beta,\beta'}$ when $\beta=\beta'$. Hence, we obtain a conditional regularity statement for the full random cluster measure $\Psi^0_{\Lambda,\beta}$ from the conditional regularity statement for $\Upsilon^0_{\Lambda,\beta,\beta}$. 
\end{Rem}

We now turn to a comparison between sprinkled random cluster measures and the original random cluster measure with a slightly modified parameter. We recall from the Introduction that this is a crucial step in deducing Theorem~\ref{thm:main} from Theorem \ref{thm:sharpness}. Let us write ${\rm Ber_\Lambda}(\varepsilon)$ to denote Bernoulli bond percolation of parameter $\varepsilon$.  
\begin{Prop}\label{prop:sprinkling}
Let $0<\beta'<\beta$. There exists $\varepsilon=\varepsilon(\beta,\beta')\in (0,1)$ such that for every $\Lambda\subset \mathbb{Z}^d$, 
\begin{equation}\label{eq:sprinkling}
\Phi_{\Lambda,\beta'}^0 \sqcup \textup{ Ber}_{\Lambda}(\varepsilon) \preccurlyeq \Phi_{\Lambda,\beta}^0.
\end{equation}
\end{Prop}

\begin{proof}
Recall the definition of $\Upsilon^0_{\Lambda,\beta,\beta'}$ above. We write $\Gamma^0_{\Lambda,\beta,\beta'}$ to denote the corresponding edge marginal on the edges of $\Lambda[\beta,\beta']$. By the stochastic domination $\mu^0_{\Lambda,\beta} \succcurlyeq \mu^0_{\Lambda,\beta'}$ and the fact that the usual random cluster model and Bernoulli bond percolation are monotone in their edge parameters, we obtain (see \cite[Lemma 6.7]{gunaratnam2025supercritical}) that
\begin{equation}\label{eq:firststepsprinkling}
    \Phi^0_{\Lambda,\beta} \succcurlyeq \Gamma^0_{\Lambda,\beta,\beta'},
\end{equation}
where we view $\Phi^0_{\Lambda,\beta}$ as a percolation measure on $\{0,1\}^{E(\Lambda[\beta,\beta'])}$ as described above.

The goal is to replace the inhomogeneous random sprinkling parameters $\{ \varepsilon(\beta,\beta',\mathsf a)_{xy} \}$ by a uniform parameter $\varepsilon > 0$. In order to do this, we consider the state of a single $\beta-\beta'$ edge $xy(\beta-\beta')$ conditional on the state of every other edge. By arguing as in Proposition \ref{prop:finite-energy} and leveraging Lemma \ref{lem:conditional regularity}, it follows that there exists $\varepsilon > 0$ such that, uniformly in the state of the other edges, the probability that $xy(\beta-\beta')$ is open is uniformly bounded from below by $\varepsilon$. In particular, one can show that $\Gamma^0_{\Lambda,\beta,\beta'}$ stochastically dominates $\Phi^0_{\Lambda,\beta'} \sqcup \textup{Ber}_\Lambda(\varepsilon)$. Combining this with \eqref{eq:firststepsprinkling}, we obtain \eqref{eq:sprinkling}. 
\end{proof}

\section{Uniqueness of half-space measures}\label{sec:Gibbs_FK}

In this section, we prove the almost everywhere uniqueness of random cluster measures on the half-space, Theorem \ref{thm:half-space} below. In order to state this, we begin by extending the definitions of the preceding section to the half-space setting. 

\medskip

\noindent\textbf{Spin models on $\mathbb H$.} Let $\mathbb H:= \mathbb Z^{d-1}\times \mathbb N$ be the half-space and $\partial \mathbb H:= \mathbb Z^{d-1}\times \{0\}$ be its boundary. Fix $\Lambda\subset \mathbb H$ finite. For every $h \geq 0$, let $\mathsf h_{\Lambda}^h \in (\mathbb R^+)^{\Lambda}$ be given by
\begin{equation}\label{eq:def mag field on halfspace}
\mathsf h_{\Lambda}^h(x)= \begin{cases} h, & x \in \Lambda\cap \partial \mathbb H, \\ 0, & \text{otherwise.} \end{cases}
\end{equation}
We define $\nu^{0,h}_{\Lambda,\beta}:= \nu_{\Lambda, \beta, \mathsf h^h_\Lambda}$. The free measure on $\mathbb H$ with inverse temperature $\beta \geq 0$ and external field $h\geq 0$ on $\partial \mathbb H$ is the probability measure $\nu^{0,h}_{\mathbb H,\beta}$ with expectation  $\langle \cdot \rangle^{0,h}_{\mathbb H,\beta}$ defined as the infinite volume limit as $\Lambda\nearrow \mathbb H$ of the probability measures $\nu^{0,h}_{\Lambda, \beta}$. This limit exists by standard monotonicity and regularity arguments. 

Defining the half-space plus measure requires more care. Let $\partial^{\textup{ext},+}\Lambda:=\{y\in \mathbb H\setminus \Lambda: \exists x\in \Lambda, \: x\sim y\}$. We define the finite volume half-space plus measure $\nu^{+,h}_{\Lambda, \beta}$ on $\Lambda \subset \mathbb H$ by the density
\begin{equation}
\mathrm{d}\nu^{+,h}_{\Lambda, \beta}(\tau) = \frac{1}{\mathbf{Z}^{+,h}_{\Lambda,\beta}}\exp\Big(\beta \sum_{xy\in \overline{E}^+(\Lambda)} \tau_x \tau_y +  \beta h \sum_{x \in \Lambda \cap \partial \mathbb H} \tau_x
\Big) \prod_{x \in \Lambda}\mathrm d \rho(\tau_x) \prod_{x \in \partial^{\textup{ext},+}\Lambda}  \mathrm d\zeta(\tau_x),
\end{equation}
where $\overline{E}^+(\Lambda) =  \{xy: x\in \Lambda , \: y\in \mathbb H, \: x\sim y\}$, and where $\zeta$ was defined in Proposition \ref{prop:reg_plus}. Additionally, we let $\partial_E^+\Lambda=\overline{E}^+(\Lambda)\setminus E(\Lambda)$.

These half-space plus measures satisfy the following monotonicity properties.

\begin{Prop}[\hspace{1pt}{\cite[Proposition 6.11, Theorem 4.1]{PV26}}]\label{prop:reg_plus half-space} 
Let $\beta,h \geq 0$. There exists $r\geq 0$ such that, for every $\Lambda'\subset \Lambda \subset \mathbb H$ finite satisfying $\mathrm{d}(\partial^{\textup{ext},+} \Lambda,\Lambda')\geq r$,
\begin{equation}\label{eq:spin-absolute value monotonicity half-space}
\nu_{\Lambda,\beta}^{+,h} \preccurlyeq \nu_{\Lambda',\beta}^{+,h} \quad \text{and} \quad \nu_{\Lambda,\beta}^{+,h}(|\cdot|) \preccurlyeq \nu_{\Lambda',\beta}^{+,h}(|\cdot|).
\end{equation} Moreover, there exist $a=a(\beta,h),b=b(\beta,h)>0$ such that, for every $\Lambda\subset \mathbb H$, the measure $\nu^{+,h}_{\Lambda,\beta}$ is $(a,b)$-regular. In particular, $\nu^{+,h}_{\Lambda,\beta}$ converges weakly to an infinite-volume measure $\nu^{+,h}_{\mathbb{H},\beta}=\langle \cdot\rangle^{+,h}_{\mathbb H,\beta}$ as $\Lambda\nearrow \mathbb H$. 
\end{Prop}

The plus measure on $\mathbb H$ with inverse temperature $\beta \geq 0$ and external field $h\geq 0$ on $\partial \mathbb H$ is the probability measure $\nu^{+,h}_{\mathbb H, \beta}$ defined by the previous proposition.

\medskip

\noindent\textbf{Random cluster measures on $\mathbb H$.} Let $\Lambda \subset \mathbb H$ finite. For every $h \geq 0$, recall $\mathsf h^h_\Lambda$ from \eqref{eq:def mag field on halfspace}. The free random cluster measure on $\Lambda$ with inverse temperature $\beta \geq 0$ and magnetic field $h \geq 0$ on $\partial \mathbb H$ is the probability measure
$\Psi^{0,h}_{\Lambda, \beta}:=\Psi^0_{\Lambda, \beta, \mathsf h^h_\Lambda}$. The half-space free random cluster measure
$\Psi^{0,h}_{\mathbb H,\beta}$ is the probability measure on $\{0,1\}^{E(\mathbb H)} \times (\mathbb R^+)^\mathbb H$ obtained as the weak limit of $\Psi^{0,h}_{\Lambda, \beta}$ as $\Lambda\nearrow \mathbb H$. It exists by standard monotonicity and regularity arguments.

For the wired random cluster measure on $\mathbb H$, we have to adapt the definition slightly to account for wired measures with external magnetic fields. Consider the finite-volume wired measures $\Psi^{1,h}_{\Lambda, \beta}$ defined by the density
\begin{multline}\label{eq: random cluster explicit density half-space}
\mathrm{d}\Psi^{1,h}_{\Lambda,\beta}[(\omega, \mathsf a)]=\frac{1}{Z^{1,h}_{\Lambda,\beta}}\prod_{xy \in \overline E^+(\Lambda)} \sqrt{1-p(\beta,\mathsf a)_{xy}}\left(\frac{p(\beta,\mathsf a)_{xy}}{1-p(\beta,\mathsf a)_{xy}}\right)^{\omega_{xy}}\,
\\
 \prod_{x \in \Lambda \cap \partial \mathbb H}  \sqrt{1-p(\beta, h, \mathsf a)_{x\fg}}\left(\frac{p(\beta, h, \mathsf a)_{x\fg}}{1-p(\beta,h, \mathsf a)_{x\fg}}\right)^{\omega_{x\fg}}  2^{k^w(\omega,\mathsf{a})} 
\prod_{x\in \Lambda}\mathrm{d} \rho(\mathsf{a}_x) \prod_{x \in \partial^{\textup{ext},+} \Lambda} \mathrm{d}\zeta(\mathsf a_x),
\end{multline}
where $p(\beta,h,\mathsf a)_{x\mathfrak g} = 1-e^{-2\beta h\mathsf a_x}$ for $x \in \Lambda \cap \partial \mathbb H$. 

Thanks to the convergence statements in Proposition \ref{prop:reg_plus half-space} and the Edwards--Sokal coupling for wired measures, there exists a probability measure $\Psi^{1,h}_{\mathbb H, \beta}$ on $\{0,1\}^{E(\mathbb H)}\times (\mathbb R^+)^{\mathbb H}$ such that, for every local event $A$,
\begin{equation}
\lim_{\Lambda \nearrow \mathbb H} \Psi^{1,h}_{\Lambda,\beta}[A] = \Psi^{1,h}_{\mathbb H,\beta}[A].
\end{equation}
We call $\Psi^{1,h}_{\mathbb H,\beta}$ the wired half-space random cluster measure with inverse temperature $\beta$ and external field $h\geq 0$ on $\partial \mathbb H$.

\begin{Rem} \label{rem: half space general rc}
Observe that we could similarly define the half-space measure $\Psi^{\xi,\mathsf{b},h}_{\Lambda,\beta}$ on $\Lambda\subset \mathbb H$, with arbitrary boundary conditions $(\xi,\mathsf{b})$. For the sake of clarity, we define these measures locally when they are needed. 
\end{Rem}

The following result is crucial to prove Theorem \ref{thm:sharpness}. It is the main innovation in our paper.
\begin{Thm}\label{thm:half-space}
For almost every $(\beta,h) \in (\mathbb R^+)^2$, we have $\Psi^{1,h}_{\mathbb{H},\beta}= \Psi^{0,h}_{\mathbb{H},\beta}$.
\end{Thm}

A stronger version of Theorem~\ref{thm:half-space}, holding for every 
$\beta, h > 0$, was obtained for the Ising model in \cite{frohlich1987semi} 
and more recently for the $\varphi^4$ model in \cite{gunaratnam2025supercritical}. 
In these works, the key tools are either the Lee--Yang theorem or a switching 
principle for the underlying random currents, neither of which is available 
at our level of generality. Our proof instead relies on a purely probabilistic 
percolation argument, using only the generalised random cluster measures.

\subsection{Warmup: uniqueness of the full-space measures}

In this subsection, we give a proof of Theorem~\ref{thm:full-space uniqueness}, thus establishing almost everywhere uniqueness in the full space.
This result will play a key role in the proof of Theorem \ref{thm:half-space}. While the two proofs follow similar high-level strategies, they differ in essential conceptual ways.

In order to establish Theorem~\ref{thm:full-space uniqueness}, we will use the following criterion, which is reminiscent of a criterion available for the standard random cluster model, see \cite[Theorem~4.63]{Grimmett2006RCM}.
\begin{Prop}\label{prop: two-pt classification}
Let $\beta > 0$. Assume that for every pair of neighbours $x,y \in \mathbb Z^d$,
\begin{equation} \label{eq: 2pt fcn equality}
\langle \tau_x \tau_y\rangle_\beta^+ = \langle \tau_x \tau_y \rangle_\beta. 
\end{equation}
Then,
\begin{equation}
\Psi^1_\beta = \Psi^0_\beta.
\end{equation}
\end{Prop}

\begin{proof} 
Observe that by \eqref{eq:psi1 dominates psi0}, one has that $\Psi^0_\beta \preccurlyeq \Psi^1_\beta$. By Strassen's theorem \cite{Strassen1965}, there exists a monotone coupling between these measures $(\mathbb Q, (\mathsf{a}^1, \omega^1), (\mathsf a^0, \omega^0))$ with $\mathsf a^1 \geq \mathsf a^0$ and $\omega^1\geq \omega^0$ almost surely. 
Note that $\mathbb Q$-almost surely we have $\Psi^1_{\beta}[ \;\cdot \mid \mathsf a = \mathsf{a}^1 ] \succcurlyeq \Psi^0_{\beta}[ \; \cdot \mid \mathsf a = \mathsf{a}^0 ]$ by the monotonicity of the usual FK-model in finite volume and taking limits. Therefore, $\mathbb Q$-almost surely
$\mathsf{a}^1_x\mathsf{a}^1_y \mathbb{Q}[ x\connect{\:\omega^1} y \mid \mathsf{a}^1] \geq \mathsf{a}^0_x\mathsf{a}^0_y \mathbb{Q}[ x\connect{\:\omega^0} y \mid \mathsf{a}^0]$  for every pair of neighbours $x,y\in\mathbb{Z}^d$. Moreover, the Edwards--Sokal coupling and \eqref{eq: 2pt fcn equality} yield $\mathbb{Q}[\mathsf{a}^1_x\mathsf{a}^1_y \mathbb{Q}[ x\connect{\:\omega^1} y \mid \mathsf{a}^1]] = \langle \tau_x \tau_y\rangle_\beta^+ = \langle \tau_x \tau_y \rangle_\beta = \mathbb{Q}[\mathsf{a}^0_x\mathsf{a}^0_y \mathbb{Q}[ x\connect{\:\omega^0} y \mid \mathsf{a}^0]]$. As a result, $\mathbb Q$-almost surely we have
\begin{equation}\label{eq: 2pt fcn equality FK}
    \mathsf{a}^1_x\mathsf{a}^1_y \mathbb{Q}[ x\connect{\:\omega^1} y \mid \mathsf{a}^1] = \mathsf{a}^0_x\mathsf{a}^0_y \mathbb{Q}[ x\connect{\:\omega^0} y \mid \mathsf{a}^0],
\end{equation}
for every pair of neighbours $x,y\in \mathbb Z^d$. 
Observe that, for each $i \in \{0,1\}$ and $y\sim x$, $\mathbb{Q}[ x\connect{\:\omega^i} y \mid \mathsf{a}^i]=\Psi^i_{\beta}[ x\connect{\:} y \mid \mathsf a = \mathsf{a}^i] = 0$ if and only if $\mathsf a^i_x\mathsf a^i_y = 0$ (to see this, recall that $p(\beta, \mathsf a)_{uv}=1-e^{-2\beta \mathsf a_u \mathsf a_v}$). Since $\mathsf a^1_x \mathsf a^1_y \geq \mathsf a^0_x \mathsf a^0_y$ and $\mathbb{Q}[ x\connect{\:\omega^1} y \mid \mathsf{a}^1]\geq \mathbb{Q}[ x\connect{\:\omega^0} y \mid \mathsf{a}^0]$, we deduce from \eqref{eq: 2pt fcn equality FK} that $\mathbb Q$-almost surely, 
\begin{equation}\label{eq:2 point neighbours}
\mathsf a^1_x \mathsf a^1_y = \mathsf a^0_x \mathsf a^0_y.
\end{equation}
If $\mathsf a^i$ had no atom at $0$---which by regularity happens if and only if the single-site measure $\rho$ has no atom at $0$, and it is independent of whether $i=1$ or $0$---we would immediately obtain that $\mathsf a^1_x=\mathsf a^0_x$ from \eqref{eq:2 point neighbours} by monotonicity. The case where zeros are allowed requires more care and is treated below.

Fix $x\in \mathbb{Z}^d$. Let $\mathcal Z_x :=
\left\{
\mathsf a :
\mathsf a_y=0\text{ for every }y\sim x
\right\}$ and, under the coupling $\mathbb Q$, write $Z_0$ and $Z_1$ for the events $\{\mathsf a^0\in\mathcal Z_x\}$ and $\{\mathsf a^1\in\mathcal Z_x\}$ respectively. Since for any neighbour $y\sim x$ we $\mathbb{Q}$-almost surely have
\[
0 = \mathsf a^1_x \mathsf a^1_y - \mathsf a^0_x \mathsf a^0_y
=
(\mathsf a_x^1-\mathsf a_x^0)\mathsf a^1_y
+
\mathsf a_x^0(\mathsf a^1_y-\mathsf a^0_y)
\]
by \eqref{eq:2 point neighbours}
and both summands are non-negative, we conclude that $\mathsf a_x^1=\mathsf a_x^0$ outside $Z_1$ (i.e.\ when $\mathsf a_y^1>0$ for some $y\sim x$), and that both vanish on
$Z_0\setminus Z_1$. Hence, $\mathsf a_x^1-\mathsf a_x^0 =
(\mathsf a_x^1-\mathsf a_x^0)\mathbf 1_{Z_1}$ and $\mathsf a_x^0\mathbf 1_{Z_1} = \mathsf a_x^0\mathbf 1_{Z_0}$, where for the latter we used that $Z_1\subseteq Z_0$ by monotonicity of the coupling. Combining these two facts we get
$\mathsf a_x^1-\mathsf a_x^0 =
\mathsf a_x^1\mathbf 1_{Z_1}
-\mathsf a_x^0\mathbf 1_{Z_0}$. Therefore,  
\begin{equation}
0
\leq \mathbb E_{\mathbb Q}[\mathsf a_x^1-\mathsf a_x^0]
=
\Psi_\beta^1[\mathsf a_x \mathbbm{1}_{\mathcal Z_x}]
-
\Psi_\beta^0[\mathsf a_x \mathbbm{1}_{\mathcal Z_x}]
=
\rho[\mathsf a_x]
\bigl(
\Psi_\beta^1[\mathcal Z_x]
-
\Psi_\beta^0[\mathcal Z_x]
\bigr)
\leq0,
\end{equation}
where we used the domain Markov property and the fact that $\Psi_\beta^0\preccurlyeq\Psi_\beta^1$. Since $\mathsf a_x^1-\mathsf a_x^0\geq0$, we conclude that $\mathsf a_x^1=\mathsf a_x^0$ holds $\mathbb{Q}$-almost surely. Since $x\in\mathbb{Z}^d$ is arbitrary, $\mathsf a^1=\mathsf a^0$ almost surely.

We now prove that $\omega^1=\omega^0$ holds $\mathbb{Q}$-almost surely. 
For simplicity, we henceforth write $\mathsf{a}':=\mathsf{a}^1=\mathsf{a}^0$. 
Using again the fact that $\mathbb{Q}[ x\connect{\:\omega^i} y \mid \mathsf{a}']=0$ if and only if $\mathsf{a}'_x\mathsf{a}'_y=0$, \eqref{eq: 2pt fcn equality FK} yields
\begin{equation} \label{eq: 2pt fcn equality fk2}
    \mathbb{Q}[ x\connect{\:\omega^1} y \mid \mathsf{a}'] = \mathbb{Q}[ x\connect{\:\omega^0} y \mid \mathsf{a}']
\end{equation}
$\mathbb{Q}$-almost surely for every $xy\in E(\mathbb Z^d)$.
By the Edwards--Sokal coupling, for $i=0,1$, we have
\begin{align}
\mathbb{Q}[ x\connect{\:\omega^i} y \mid \mathsf{a}'] 
&=\langle \textup{sgn}(\tau_x) \textup{sgn}(\tau_y) \mid |\tau|=\mathsf a' \rangle^i_{\beta} ~~~~~~\text{  and  }  \\ 
\mathbb{Q}[ \omega^i_{xy}=1 \mid \mathsf{a}'] 
&= p_{xy}(\beta,\mathsf{a}') \langle \mathds{1}_{\textup{sgn}(\tau_x) \textup{sgn}(\tau_y)=1}\mid |\tau|=\mathsf a' \rangle^i_{\beta},
\end{align}
where we wrote for convenience $\langle \cdot \rangle^1_\beta = \langle \cdot \rangle^+_\beta$ and $\langle \cdot \rangle^0_\beta = \langle \cdot \rangle_\beta$, and we recall that $\textup{sgn}(t) = \mathds 1_{t> 0} - \mathds 1_{t<0}$. 
Now, since $\mathds{1}_{\textup{sgn}(\tau_x) \textup{sgn}(\tau_y)=1}=\tfrac{1}{2}(\textup{sgn}(\tau_x) \textup{sgn}(\tau_y)+1)$ when $\mathsf{a}'_x\mathsf{a}'_y\neq 0$, one has
\begin{equation} \label{eq: 2pt fcn equality fk3}
    \mathbb{Q}[ \omega^i_{xy}=1 \mid \mathsf{a}'] = \frac{p_{xy}(\beta,\mathsf{a}') }{2}\Big( \mathbb{Q}[ x\connect{\:\omega^i} y \mid \mathsf{a}']+ 1 \Big).
\end{equation}
Combining \eqref{eq: 2pt fcn equality fk2} and \eqref{eq: 2pt fcn equality fk3} yields that $\mathbb{Q}[ \omega^1_{xy}=1 \mid \mathsf{a}'] = \mathbb{Q}[ \omega^0_{xy}=1 \mid \mathsf{a}']$ almost surely. In particular, averaging over $\mathsf a'$ yields $\mathbb{Q}[ \omega^1_{xy}=1] = \mathbb{Q}[\omega^0_{xy}=1]$ for every $xy\in E(\mathbb Z^d)$. Since $\omega^1\geq\omega^0$, this implies that $\omega^1=\omega^0$ almost surely. This concludes the proof that $\Psi^1_\beta = \Psi^0_\beta$. 
\end{proof}

We follow a classical argument originally due to Lebowitz and Martin-Löf in the case of the Ising model \cite{lebowitz1972uniqueness} (see also \cite[Theorem~1.12]{DuminilLecturesOnIsingandPottsModels2019}), and reduce the proof of Theorem \ref{thm:full-space uniqueness} to the following proposition.
\begin{Prop}\label{prop: key prop uniqueness full space} For every $0<\beta'<\beta$,
\begin{equation}\label{eq:2 point ineq full space}
\langle \tau_0 \tau_{\mathbf{e}_1} \rangle^+_{\beta'}\leq \langle \tau_0 \tau_{\mathbf{e}_1} \rangle_{\beta},
\end{equation}
where $\mathbf{e}_1=(1,0,\ldots,0)$.
\end{Prop}
With this result and Proposition \ref{prop: two-pt classification}, it is easy to complete the proof of Theorem \ref{thm:full-space uniqueness}.
\begin{proof}[Proof of Theorem~\textup{\ref{thm:full-space uniqueness}}] By Proposition \ref{prop: monot plus}, the map $\beta\mapsto \langle \tau_0 \tau_{\mathbf{e}_1} \rangle^+_{\beta}$ is increasing. As a consequence, its set of points of discontinuity is at most countable. If $\beta$ is a point of continuity of this map, \eqref{eq:2 point ineq full space} gives $\langle \tau_0 \tau_{\mathbf{e}_1} \rangle^+_{\beta}\leq \langle \tau_0 \tau_{\mathbf{e}_1} \rangle_{\beta}$, and since the reverse inequality holds (by Proposition \ref{prop: monot plus}), we get $\langle \tau_0 \tau_{\mathbf{e}_1} \rangle^+_{\beta}=\langle \tau_0 \tau_{\mathbf{e}_1} \rangle_{\beta}$. By symmetry, this equality extends to every neighbouring $x,y\in \mathbb Z^d$. Proposition \ref{prop: two-pt classification} then allows us to conclude. 
\end{proof}

In order to prove Proposition \ref{prop: key prop uniqueness full space}, we study concentration events for the energy $\sum_{xy\in E(\Lambda_n)}\tau_x\tau_y$ under the measures $\nu^+_{\Lambda_n,\beta'}$ and $\nu^0_{\Lambda_n,\beta}$. We exploit the symmetries of $\mathbb Z^d$ to argue that it concentrates around $|E(\Lambda_n)|\langle \tau_0\tau_{\mathbf{e}_1}\rangle_{\beta'}^+$ and $|E(\Lambda_n)|\langle\tau_0\tau_{\mathbf{e}_1}\rangle_{\beta}$, respectively. We compare the probabilities of these concentration events under the two measures and make the following observations: (1) going from plus to free boundary conditions costs at most $\exp(C|\partial \Lambda_n|)$ for some $C<\infty$; (2) if \eqref{eq:2 point ineq full space} were not true, going from $\beta'$ to $\beta$ would induce an energy cost of order at least $\exp(c|E(\Lambda_n)|)$ for some $c>0$.  Since all probabilities involved are bounded away from $0$, the two competing effects cannot hold true simultaneously, which yields a contradiction.

\begin{proof}[Proof of Proposition \textup{\ref{prop: key prop uniqueness full space}}]

As discussed above, we aim to prove \eqref{eq:2 point ineq full space}. Let us write $a=\langle \tau_0 \tau_{\mathbf{e}_1} \rangle^+_{\beta'}$ and $b=\langle \tau_0 \tau_{\mathbf{e}_1}\rangle_{\beta}$, so our goal is to prove that $a\leq b$.

Let $\varepsilon,C>0$ to be fixed a posteriori and define the events \begin{align}
A_n&= \Big\{ \sum_{xy\in E(\Lambda_n)} \tau_x\tau_y \geq |E(\Lambda_n)|(a-\varepsilon) \Big\}, \\
B_n&= \Big\{ \sum_{xy\in E(\Lambda_n)} \tau_x\tau_y \leq |E(\Lambda_n)|(b+\varepsilon) \Big\}, \\
C_n&=\Big\{\sum_{xy\in \partial_E \Lambda_n} |\tau_x \tau_y|\leq Cn^{d-1}\Big\},
\end{align}
where we recall that $\partial_E\Lambda_n=\{xy\in E(\mathbb Z^d) : x\in \Lambda_n, \: y\notin \Lambda_n\}$. 

\begin{Claim} \label{claim: technical full space}
Let $0 < \beta' < \beta$. For every $\varepsilon>0$, and every $C>0$ sufficiently large, there exists $\delta > 0$ such that, for every $n$ large enough,
\begin{equation}\label{eq:technical lemma full space}
\nu^+_{\Lambda_n,\beta'}[A_n] \geq \delta, \quad \nu^0_{\Lambda_n,\beta}[B_n] \geq \delta, \quad \nu^+_{\Lambda_n,\beta'}[C_n] \geq 1-\frac \delta 2. 
\end{equation}
\end{Claim}
The first two bounds in \eqref{eq:technical lemma full space} are concentration estimates, while the last one follows from regularity. 
Before proving Claim \ref{claim: technical full space}, we conclude the proof of Proposition~\ref{prop: key prop uniqueness full space}. 
We collect two easy consequences of \eqref{eq:technical lemma full space}. First, letting $G(\tau)=\prod_{xy\in E(\Lambda_n)} e^{(\beta'-\beta)\tau_x \tau_y}$, one has
\begin{equation} \label{eq: full space: An}
    \nu^0_{\Lambda_n, \beta'}[A_n]=\frac{\langle G(\tau) \mathbbm{1}_{A_n} \rangle ^0_{\Lambda_n,\beta}}{\langle G(\tau) \rangle^0_{\Lambda_n,\beta}}\leq \frac{\langle G(\tau) \mathbbm{1}_{A_n} \rangle ^0_{\Lambda_n,\beta}}{\langle G(\tau) \mathbbm{1}_{B_n}\rangle^0_{\Lambda_n,\beta}}\leq \frac{1}{\delta}\cdot e^{(\beta'-\beta)|E(\Lambda_n)|(a-b-2\varepsilon)}.
\end{equation}
Second, for any $\eta\in \mathbb R^{\partial^{\textup{ext}}\Lambda_n}$,
\begin{equation}\label{eq: full space: AnCn}
   \nu^\eta_{\Lambda_n,\beta'}[A_n\cap C_n(\eta)]= \frac{ \nu^0_{\Lambda_n,\beta'}[e^{\beta' H^{\eta}(\tau)}\mathbbm{1}_{A_n \cap C_n(\eta)}]}{\nu^0_{\Lambda_n,\beta'}[e^{\beta' H^{\eta}(\tau)}]}\leq e^{C\beta' n^{d-1}}\nu^0_{\Lambda_n,\beta'}[A_n],
\end{equation}
where 
\[
H^\eta(\tau)
:=
\sum_{\substack{xy\in\partial_E\Lambda_n\\
x\in\Lambda_n,\;y\in\partial^{\mathrm{ext}}\Lambda_n}}
\tau_x\eta_y, \qquad C_n(\eta):=\Big\{\sum_{xy\in \partial_E \Lambda_n} |\tau_x \eta_y|\leq Cn^{d-1}\Big\},
\]
and we used that $\nu^0_{\Lambda_n,\beta'}
[e^{\beta' H^\eta}]=
\nu^0_{\Lambda_n,\beta'}
\left[\cosh\left(\beta'H^\eta\right)\right]\geq 1$.
Using the domain Markov property and integrating appropriately over $\eta$ (recall Definition \ref{def:+measure spin finite volume}) yields
\begin{equation}\label{eq: full space:An bis}
    \nu^+_{\Lambda_n,\beta'}[A_n\cap C_n]\leq e^{C\beta' n^{d-1}}\nu^0_{\Lambda_n,\beta'}[A_n].
\end{equation}
Moreover, by \eqref{eq:technical lemma full space}, $\frac{\delta}{2}\leq \nu^+_{\Lambda_n,\beta'}[A_n\cap C_n]$. Hence, combining \eqref{eq: full space: An} and \eqref{eq: full space:An bis}, we obtain
\begin{equation} \label{eq: full space cornd}
\frac{\delta}{2} \leq \frac{e^{C\beta' n^{d-1}}}{\delta}e^{(\beta'-\beta)|E(\Lambda_n)|(a-b-2\varepsilon)}. 
\end{equation}
Since $|E(\Lambda_n)|\geq cn^d$ and $\beta'<\beta$, sending $n$ to infinity, we get that $a-b-2\varepsilon\leq 0$. Since $\varepsilon>0$ is arbitrary, we conclude that \eqref{eq:2 point ineq full space} holds, as desired. 
\end{proof}

Finally, we prove Claim \ref{claim: technical full space}. We rely on a simple concentration result, stated in Lemma \ref{lem:random variables} below.

\begin{proof}[Proof of Claim~\textup{\ref{claim: technical full space}}] 

Fix $\varepsilon>0$. Observe that for every $xy\in E(\Lambda_n)$, Proposition \ref{prop: monot plus} gives that $\langle\tau_x\tau_y\rangle_{\Lambda_n,\beta'}^+\geq a$. Moreover, by regularity (see Proposition \ref{prop:reg_plus}), there exists $M>0$ large enough such that $\nu_{\Lambda_n,\beta'}^+[|\tau_x\tau_y|\mathds{1}_{|\tau_x\tau_y|\geq M}]\leq \tfrac{\varepsilon}{2}$. Therefore, Lemma \ref{lem:random variables}, applied to the random variables $(\tau_x\tau_y)_{xy\in E(\Lambda_n)}$, $a$, and $M$ yields the existence of $\delta'=\delta'(a,M)>0$ such that
\begin{equation}
    \nu_{\Lambda_n,\beta'}^+[A_n]\geq \delta'.
\end{equation}
The proof of the lower bound on $\nu_{\Lambda_n,\beta}^0[B_n]$ follows the same strategy. By Proposition \ref{prop:monotonicity}, one has $\langle \tau_x\tau_y\rangle_{\Lambda_n,\beta}\leq b$. Moreover, there exists $M>0$ (which can be chosen to be the same as above) which satisfies $\nu_{\Lambda_n,\beta}^0[|\tau_x\tau
_y|\mathds{1}_{|\tau_x\tau_y|\geq M}]\leq \tfrac{\varepsilon}{2}$.  Lemma \ref{lem:random variables} applied to the random variables $(-\tau
_x\tau_y)_{xy\in E(\Lambda_n)}$ yields the existence of $\delta''=\delta''(b,M)>0$ such that
\begin{equation}
    \nu_{\Lambda_n,\beta}^0[B_n]\geq \delta''.
\end{equation}
We define $\delta:=\delta'\wedge \delta''$. Finally, the bound on $\nu_{\Lambda_n,\beta'}^+[C_n]$ is a straightforward consequence of regularity (Proposition \ref{prop:reg_plus}), by choosing $C=C(\delta)$ large enough. This concludes the proof.
\end{proof}

\begin{Lem}\label{lem:random variables} Let $X_1,\ldots, X_n$ be integrable real-valued random variables on a probability space $(\Omega, \mathcal F, \mathbb P)$. Let $S_n= X_1 + \ldots +X_n$. 
Assume the following.
\begin{enumerate}
    \item[(i)] There exists $a\in\mathbb R$ such that $\mathbb E[X_i] \geq a$ for every $1 \leq i \leq n$.
    \item[(ii)] There exist $M\geq \max\{a,0\}$ and $\delta > 0$ such that $\mathbb E[X_i \mathbbm{1}_{X_i>M}]\leq \delta$ for every $1 \leq i \leq n$.
\end{enumerate}
Then,
\begin{equation}
    \mathbb P[S_n \geq (a-2\delta)n]\geq \delta':=\frac{\delta}{M-a+2\delta}.
\end{equation}

\end{Lem}
\begin{proof} Let $Y_i:=X_i\mathbbm{1}_{X_i\leq M}$ and notice that by $(i)$ and $(ii)$ we have $\mathbb E[Y_i] = \mathbb E[X_i] - \mathbb E[X_i\mathbbm{1}_{X_i>M}] \geq a-\delta$. In particular, the non-negative random variable $Z:=\sum_{i=1}^n (M-Y_i)$ satisfies $\mathbb{E} [Z]\leq (M-a+\delta)n$. By Markov's inequality we have
\begin{equation}
\mathbb P[S_n \leq (a-2\delta)n] \leq \mathbb P[Z \geq (M-a+2\delta)n ]\leq \frac{(M-a+\delta)n}{(M-a+2\delta)n}=1-\delta',
\end{equation}
which concludes the proof.
\end{proof}

\subsection{Proof of Theorem \ref{thm:half-space}}

We now turn our attention to proving uniqueness of the half-space random cluster measures, as stated in Theorem \ref{thm:half-space}. It is natural to follow the strategy of Lebowitz and Martin-L\"of used to prove Theorem \ref{thm:full-space uniqueness}. We first observe that the following counterpart of Proposition \ref{prop: two-pt classification} holds in our setting.
\begin{Prop}\label{prop: two-pt classification half-space} Let $(\beta,h)\in (\mathbb R^+)^2$. Assume that for every pair of neighbours $u,v\in \mathbb H$,
\begin{equation}
    \langle \tau_u\tau_v\rangle_{\mathbb H,\beta}^{+,h}=\langle \tau_u\tau_v\rangle_{\mathbb H,\beta}^{0,h}.
\end{equation}
Then,
\begin{equation}
    \Psi^{1,h}_{\mathbb H,\beta}=\Psi^{0,h}_{\mathbb H,\beta}.
\end{equation}
\end{Prop}
\begin{proof} The proof follows exactly the same lines as the one of Proposition \ref{prop: two-pt classification}.
\end{proof}

Establishing the right counterpart of Proposition \ref{prop: key prop uniqueness full space} is significantly more difficult. Indeed, we cannot reduce to only considering the edge $0\mathbf e_1$ as our setting is no longer invariant under all automorphisms of $\mathbb Z^d$. Therefore, in order to extend the concentration argument used in the full-space setting, we are forced to average separately along each ``layer'' of $\mathbb H$, or, more specifically, along the orbits of a given edge under the automorphism group of $\mathbb H$. As a result, we change $\beta$ to $\beta'$ only along such sets. These considerations motivate the following definitions.  

Let $\Gamma$ be the group of automorphisms of $\mathbb{H}$. 
Given $0<\beta'<\beta$ and $e=uv\in E(\mathbb{H})$, let $(J_{xy}
(\beta,\beta',e))_{xy\in E(\mathbb H)}$ be the coupling constants given by
\begin{equation}
J_{xy}(\beta,\beta',e)=\beta'\mathbbm{1}_{xy\in \Gamma uv} + \beta\mathbbm{1}_{xy\notin \Gamma uv},
\end{equation}
where $\Gamma uv$ is the orbit of $uv$ under $\Gamma$. Let $\langle \cdot 
\rangle^{+,h}_{\mathbb{H},J(\beta,\beta',e)}$ be 
the corresponding plus measure on $\mathbb{H}$ with 
coupling constants $J(\beta,\beta',e)$. If $e$ lies in $\partial\mathbb{H}$, we simply write $\langle \cdot 
\rangle^{+,h}_{\mathbb{H},\beta,\beta'}$.

Finally, to work out the concentration argument as we did in the setting of $\mathbb Z^d$ we crucially leverage that $\Psi^1_\beta=\Psi^0_\beta$. This important restriction is explained in more detail below.  We recall that the set $\Xi$, where these measures do not coincide, was defined above Theorem \ref{thm:full-space uniqueness}.

\begin{Prop}\label{prop: key prop uniqueness half space} For every $\beta \notin \Xi$, every $0<\beta'<\beta$, and every $h\geq 0$, 
\begin{equation}\label{eq:2 point ineq hs}
    \langle \tau_{u}\tau_v \rangle^{+,h}_{\mathbb{H},J(\beta,\beta',uv)}\leq \langle \tau_u \tau_v  \rangle^{0,h}_{\mathbb{H},\beta} \quad \forall u,v\in \mathbb H, \: u\sim v.
\end{equation}
\end{Prop}

The combination of Proposition \ref{prop: two-pt classification half-space} and Proposition \ref{prop: key prop uniqueness half space} readily yields Theorem \ref{thm:half-space}.

\begin{proof}[Proof of Theorem~\textup{\ref{thm:half-space}}]
Denote by $\mathfrak{U}$ the set of $(\beta,h)\in (0,\infty)^2$ such that 
\begin{equation}
\beta\notin \Xi, \qquad 
\lim_{\beta'\uparrow \beta}\langle \tau_u \tau_v\rangle^{+,h}_{\mathbb{H},\beta'}=\langle \tau_u \tau_v\rangle^{+,h}_{\mathbb{H},\beta} \quad \forall u,v\in \mathbb H, \: u\sim v.
\end{equation}
On the set $\mathfrak U$, we have, for every $u,v\in \mathbb H$ with $u\sim v$, 
\begin{equation}
\lim_{\beta'\uparrow \beta}\langle \tau_u \tau_v \rangle^{+,h}_{\mathbb{H},J(\beta,\beta',uv)}=\langle \tau_u \tau_v\rangle^{+,h}_{\mathbb{H},\beta},
\end{equation}
since (by Proposition \ref{prop: monot plus}) $\langle \tau_u \tau_v \rangle^{+,h}_{\mathbb{H},\beta'}\leq \langle \tau_u \tau_v \rangle^{+,h}_{\mathbb{H},J(\beta,\beta',uv)}\leq \langle \tau_u \tau_v \rangle^{+,h}_{\mathbb{H},\beta}$. Hence, by Proposition \ref{prop: key prop uniqueness half space}, for every $u,v\in \mathbb H$ with $u\sim v$, $\langle \tau_u \tau_v \rangle^{+,h}_{\mathbb{H},\beta}\leq \langle \tau_u \tau_v \rangle^{0,h}_{\mathbb{H},\beta}$, and the converse inequality holds by Proposition \ref{prop: monot plus}. Proposition \ref{prop: two-pt classification half-space} then allows us to conclude that $\Psi^{1,h}_{\mathbb{H},\beta}=\Psi^{0,h}_{\mathbb{H}, \beta}$.

Finally, we need to argue that $\mathfrak U$ has full measure. Let $e=uv$ be an edge with $u,v\in \mathbb{H}$. Note that for every $h>0$, the function $\beta\mapsto \langle \tau_u \tau_v\rangle^{+,h}_{\mathbb{H},\beta}$ is increasing, hence it has at most countably many discontinuities. Combined with Theorem~\ref{thm:full-space uniqueness} and Fubini--Tonelli's theorem,
\begin{equation}
\int_{(\mathbb{R}^+)^2} \mathbbm{1}_{(\beta,h)\not\in \mathfrak{U}}\mathrm{d}(\beta,h)=\int_{\mathbb{R}^+}\Big(\int_{\mathbb{R}^+} \mathbbm{1}_{(\beta,h)\not\in \mathfrak{U}}\mathrm{d}\beta\Big)\mathrm{d}h=\int_{\mathbb R^+} 0\:\mathrm{d}h=0.
\end{equation}   
Since there are countably many edges, this concludes the proof. 
\end{proof}

We now explain the main conceptual ideas involved in the proof of Proposition \ref{prop: key prop uniqueness half space}. If we were to run a similar argument as in the proof of Proposition \ref{prop: key prop uniqueness full space}, there would be a significant obstacle in the comparison of the probabilities of the finite-volume concentration events. Notably, due to the reduced symmetries, the change in $\beta$ now induces a surface-order exponential cost $\exp(Cn^{d-1})$, for some $C<\infty$, which is a priori the same exponential order as the change in boundary conditions. This invalidates the rest of the argument. Our aim is to improve the change in boundary condition cost to $\exp(\varepsilon n^{d-1})$ for $\varepsilon>0$ arbitrarily small. Rather than comparing the concentration events directly for the finite-volume free and plus spin measures, we leverage the Edwards--Sokal coupling, and work with finite-volume random cluster measures with random boundary conditions induced by $\Psi^{1,h}_{\mathbb H,\beta,\beta'}$ and $\Psi^{0,h}_{\mathbb H,\beta}$ respectively. Thanks to our assumption $\beta \not\in\Xi$, we know that the full-space random cluster measures coincide: $\Psi^0_\beta = \Psi^1_\beta$. Hence, it is natural to expect that the boundary conditions induced ``far away from the boundary of $\mathbb H$'' look the same under these two measures. Indeed, we show that, up to a small surface-order exponential contribution, the following hold: (1) the induced absolute values are approximately the same under the wired and free measures---see Lemma \ref{lem:abs value approx}; (2) the induced partitions are approximately the same under the wired and free measures---see Lemma \ref{lem:partitions approx}. Whilst (1) is intuitive due to Theorem \ref{thm:full-space uniqueness}, (2) is less obvious since the induced partition is not a local event and depends on the whole complement of the domain. In order to justify (2), we adapt ideas from Burton and Keane \cite{BurtonKeane1989density}. In order to state what (1) and (2) correspond to precisely, we first fix notation and introduce the standard couplings between the measures discussed above.

\vspace{5pt}

For technical reasons, we work on random cluster measures on thin hyperrectangles. For $n,\ell \in \mathbb N^*$ (think $n\gg \ell$), let 
$$R(n,\ell):= \{-n,\dots,n\}^{d-1}\times \{0,\dots, \ell\} \subset \mathbb H.$$ 
Given $n,\ell \in \mathbb N^*$, $h \geq 0$, and boundary conditions $(\xi, \mathsf b)$ on $R(n,\ell)$, we recall that we denote the associated random cluster measure $\Psi^{\xi,\mathsf b, h}_{R(n,\ell),\beta}$, see Remark \ref{rem: half space general rc}. For $\beta'<\beta$, we write $\Psi^{\xi,\mathsf b, h}_{R(n,\ell),\beta,\beta'}$ to denote the measure with parameter $\beta'$ for every edge $xy$ with $x,y \in \partial \mathbb H$ and $\beta$ elsewhere. 
\begin{enumerate}
    \item[$\bullet$] \textit{Monotone coupling.} Let $(\mathbb P_{\beta}, (\mathsf a^1, \omega^1), (\mathsf a^0, \omega^0))$ be a monotone coupling of $\Psi^{1,h}_{\mathbb H, \beta}$ and $\Psi^{0,h}_{\mathbb H, \beta}$. It can be defined by Strassen's theorem since $\Psi^{1,h}_{\mathbb H, \beta} \succcurlyeq \Psi^{0,h}_{\mathbb H, \beta}$.
    
    \item[$\bullet$] \textit{Edwards--Sokal coupling.}  Let $n,\ell \in \mathbb N^*$, $\beta,h \geq 0$. Given a boundary condition $(\xi,\mathsf b)$ on $R(n,\ell)$, we let $\mathbf P^{\xi,\mathsf b,h}_{R(n,\ell), \beta}$ denote the Edwards--Sokal coupling of $\Psi^{\xi,\mathsf b,h}_{R(n,\ell),\beta}$.
    More precisely, we say that a triplet $(\mathsf a, \omega, \sigma)$ with $\sigma_{\mathfrak g}=\mathsf{a}_{\mathfrak g}=1$ is compatible if $\sigma \in \{-1,0,1\}^{R(n,\ell)^\mathfrak g\sqcup \partial^{\textup{ext},+}R(n,\ell)}$ is constant non-zero on the connected components of 
the graph with vertex set $\{x\in R(n,\ell)^\fg\sqcup \partial^{\textup{ext},+}R(n,\ell): \mathsf{a}_x\neq 0\}$ and edge set $\{xy: \omega_{xy}=1\}$ 
after identifying vertices of $\partial^{\textup{ext},+}R(n,\ell)\cup \{\fg\}$ in the same partition class of $\xi$, and $\sigma_x=0$ whenever $\mathsf{a}_x=0$. The coupling measure $\mathbf P^{\xi,\mathsf b,h}_{R(n,\ell), \beta}$ is the probability measure on compatible triplets $(\mathsf{a},\omega,\sigma)$ given by
\begin{align}\label{eq:formula ES}
\begin{aligned}
\mathrm{d}\mathbf P^{\xi,\mathsf b, h}_{R(n,\ell), \beta}[(&\mathsf{a}, \omega,\sigma)]=
\frac{1}{Z^{\xi,\mathsf{b},h}_{R(n,\ell),\beta}} 
\prod_{xy\in \overline{E}^+(R(n,\ell))} \sqrt{1-p_{xy}(\mathsf{a})}\left(\frac{p_{xy}(\mathsf{a})}{1-p_{xy}(\mathsf{a})}\right)^{\omega_{xy}}  \\&\prod_{x\in R(n,\ell)\cap \partial \mathbb H} \sqrt{1-p_{x\fg}(\mathsf a)}\left(\frac{p_{x\fg}(\mathsf a)}{1-p_{x\fg}(\mathsf a)}\right)^{\omega_{x\fg}} \prod_{x\in R(n,\ell)} \mathrm{d}\rho(\mathsf{a}_x)\prod_{x \in \partial^{\textup{ext},+} R(n,\ell)} \delta_{\mathsf{b}_x}(\mathsf{a}_x).
\end{aligned}
\end{align}
For convenience, above we abbreviated $p_{xy}(\mathsf{a})=p(\beta,\mathsf{a})_{xy}$. Note that under this measure, the spin configuration $\tau\in \mathbb R^{R(n,\ell)\sqcup \partial^{\textup{ext},+}R(n,\ell)}$ defined as $\tau_x:=\mathsf a_x \sigma_x$ is indeed distributed as $\nu^{0,h}_{R(n,\ell)\sqcup \partial^{\textup{ext},+}R(n,\ell), \beta}[\: \cdot \mid \mathcal S(\xi, \mathsf b)]$, where $\mathcal S(\xi, \mathsf b)$ is the event that the spins on $\partial^{\textup{ext},+}R(n,\ell)$ have absolute value field $\mathsf b$ and $\textup{sgn}(\tau_u)=\textup{sgn}(\tau_v)$ for every $u,v \in \partial^{\textup{ext},+}R(n,\ell)$ in the same partition class of $\xi$.
 For $\beta'\leq\beta$, we may also define the Edwards--Sokal coupling $\mathbf P^{\xi,\mathsf b, h}_{R(n,\ell),\beta, \beta'}$ analogously.
\end{enumerate}

\begin{Def} Given a random cluster configuration $(\omega,\mathsf a)\in \{0,1\}^{E(\mathbb H)}\times (\mathbb R^+)^{\mathbb H}$, for every $n,\ell \in \mathbb N^*$, we may define a map $(\omega, \mathsf a) \mapsto (\xi_{n,\ell}(\omega), \mathsf b_{n,\ell}(\mathsf a))$ to denote the induced boundary condition on $R(n,\ell)$. Given $n,\ell$, we will denote by $(\xi^1_{n,\ell}, \mathsf b^1_{n,\ell})$ and $(\xi^0_{n,\ell}, \mathsf b^0_{n,\ell})$ the boundary conditions on $R(n,\ell)$ induced by $(\omega^1,\mathsf a^1)$ and $(\omega^0,\mathsf a^0)$, respectively. 
\end{Def}

Following the strategy described above, we introduce events that help us measure how close the induced boundary conditions $(\xi^1_{n,\ell}, \mathsf b^1_{n,\ell})$ and $(\xi^0_{n,\ell}, \mathsf b^0_{n,\ell})$ are under the monotone coupling. The first event captures how close $\mathsf{b}_{n,\ell}^1$ and $\mathsf{b}_{n,\ell}^0$ are. Let $\varepsilon>0$. For every $n,\ell \in \mathbb N^*$, define
\begin{equation}
\mathcal A_0(\varepsilon,n,\ell):= \Big\{ (\mathsf a^1,\omega^1,  \mathsf a^0,\omega^0) : \sum_{y \in \partial^{\mathrm{ext},+} R(n,\ell)} \Big(\mathds{1}_{(\mathsf{b}_{n,\ell}^0)_y=0}-\mathds{1}_{(\mathsf{b}_{n,\ell}^1)_y=0}\Big)\leq \varepsilon n^{d-1}\Big\}, 
\end{equation}
and, for $\delta>0$
\begin{equation}
\mathcal A_1(\varepsilon,\delta,n,\ell):= \Big\{ (\mathsf a^1,\omega^1,  \mathsf a^0,\omega^0) : \mathbf P^{\xi^1_{n,\ell}, \mathsf b^1_{n,\ell},h}_{R(n,\ell), \beta, \beta'}\Big[\sum_{xy \in \partial^{+}_E R(n,\ell)} \mathsf a_x \Big((\mathsf b^1_{n,\ell})_y - (\mathsf b^0_{n,\ell})_y\Big)\leq \varepsilon n^{d-1}\Big]\geq 1-\delta \Big\}, 
\end{equation}
where we recall that $\partial^+_E R(n,\ell)=\overline{E}^+(R(n,\ell))\setminus E(R(n,\ell))$ and where we implicitly assume $x \in R(n,\ell)$ and $y \in \partial^{\rm ext,+}R(n,\ell)$. Finally, let
\begin{equation}
    \mathcal{A}(\varepsilon,\delta,n,\ell):=\mathcal A_0(\varepsilon,n,\ell)\cap \mathcal A_1(\varepsilon,\delta,n,\ell).
\end{equation}

The second event measures how close the induced bulk partitions are on the level of the number of clusters. Let $\varepsilon > 0$. For every $n, \ell \in \mathbb N^*$, define
\begin{multline}
\mathcal P(\varepsilon, n,\ell) = \Big\{ (\mathsf a^1, \omega^1, \mathsf a^0, \omega^0) : \forall \, \omega \in \{0,1\}^{\overline E^+(R(n,\ell))},\forall \mathsf{a}\in (\mathbb{R}^+)^{R(n,\ell)}, \,  
\\
\overline{k}^{(\xi^0_{n,\ell},\mathsf{b}^0_{n,\ell})}(\omega,\mathsf{a}) \leq \overline{k}^{(\xi^1_{n,\ell},\mathsf{b}^1_{n,\ell})}(\omega,\mathsf{a}) + \varepsilon n^{d-1} \Big\}, 
\end{multline}
where we have set $\overline{k}^{(\xi,\mathsf{b})}(\omega,\mathsf{a})$ to be equal to the number of connected components in the graph with vertex set $\{x\in R(n,\ell): \mathsf{a}_x\neq 0\}\cup \partial^{\textup{ext},+}R(n,\ell)$ and edge set $\{e: \omega_e\neq 0\}$ after identifying vertices in the same partition class of $\xi$. Note that this definition differs from the one near \eqref{eq:def p(beta,a)} in that the connected components of vertices $y\in \partial^{\rm{ext},+}R(n,\ell)$ such that $\mathsf{b}_y=0$ contribute one to $\overline{k}^{(\xi,\mathsf{b})}(\omega,\mathsf{a})$. Observe that with this definition, one has $\mathbb{P}_{\beta}$-almost surely
\begin{equation}\label{eq: k bar are ordered}
    \overline{k}^{(\xi^1_{n,\ell},\mathsf{b}^1_{n,\ell})}(\omega,\mathsf{a})\leq   \overline{k}^{(\xi^0_{n,\ell},\mathsf{b}^0_{n,\ell})}(\omega,\mathsf{a}).
\end{equation}
We prove the following results.
\begin{Lem}\label{lem:abs value approx}
Let $\beta >0$ such that $\beta \not\in \Xi$. For every $\varepsilon,\delta > 0$, there exists $\ell$ large enough such that
\begin{equation}
\liminf_{n \rightarrow \infty} \mathbb P_{\beta}[\mathcal A(\varepsilon,\delta,n, \ell)] \geq 1-\delta.
\end{equation}
\end{Lem}
\begin{Lem}\label{lem:partitions approx}
Let $\beta >0$ such that $\beta \not\in \Xi$. For every $\varepsilon, \delta > 0$, there exists $\ell$ large enough such that
\begin{equation}
\liminf_{n \rightarrow \infty} \mathbb P_\beta[\mathcal P(\varepsilon,n,\ell)] \geq 1-\delta. 
\end{equation}
\end{Lem}
In words, Lemma \ref{lem:abs value approx} says that \emph{bulk absolute values under the wired and free measures are approximately equal}, and Lemma \ref{lem:partitions approx} says that \emph{induced bulk partitions under wired and free measures are approximately equal} under the monotone coupling.

\begin{proof}[Proof of Proposition \textup{\ref{prop: key prop uniqueness half space}}]

We now turn to the proof of \eqref{eq:2 point ineq hs}. We prove this for the case $u,v\in \partial \mathbb{H}$ for ease of notation, and the proof of the general case is similar. We fix an edge $uv$ included in $\partial \mathbb H$. Set 
\begin{equation}
a_{uv}=a_{uv}(\beta,\beta')=\langle \tau_u \tau_v \rangle ^{+,h}_{\mathbb{H},\beta,\beta'}
\quad \text{and} \quad
b_{uv}=b_{uv}(\beta)=\langle \tau_u \tau_v \rangle ^{0,h}_{\mathbb{H},\beta}.
\end{equation}
Our goal is to prove that $a_{uv}\leq b_{uv}$. Let $\varepsilon>0$ to be chosen small enough below. For every $n \in \mathbb N^*$, consider the events
\begin{equation}\label{eq: def a_n b_n half space}
A_n=\Big\{\sum_{xy\in E_n}\tau_x \tau_y\ge |E_n|(a_{uv}-\varepsilon)\Big\}, \qquad B_n=\Big\{\sum_{xy\in E_n}\tau_x \tau_y\le |E_n| (b_{uv}+\varepsilon)\Big\},
\end{equation}
where $E_n=\Gamma uv \cap E(R(n,\ell))$ denotes the set of edges $xy$ in $R(n,\ell)$ with $x,y\in \partial \mathbb{H}$. We would like to understand the probabilities of these events on large hyperrectangles with induced boundary conditions from the half-space free and wired random cluster measures. To this end, for $\delta>0$, let us introduce the following events:
\begin{align}
\begin{aligned}
\mathcal C(\varepsilon,\delta,n,\ell)& := \Big\{ (\mathsf a^1, \omega^1, \mathsf a^0, \omega^0) : \mathbf P^{\xi^1_{n,\ell},\mathsf b^1_{n,\ell}, h}_{R(n,\ell), \beta,\beta'}[A_n] \geq \delta \Big\},
\\ 
\mathcal{D}(\varepsilon,\delta,n,\ell)&:=\Big\{(\mathsf a^1, \omega^1, \mathsf a^0, \omega^0): \mathbf P^{\xi^0_{n,\ell},\mathsf b^0_{n,\ell}, h}_{R(n,\ell), \beta}[B_n] \geq \delta\Big\}.
\end{aligned}
\end{align}
We claim the following bounds.
\begin{Claim}\label{claim:mathcalCD} For every $\varepsilon>0$, there exists $c_0>0$ such that the following holds. For every $\ell\geq 1$, every $\delta\in (0,c_0)$, and every $n$ large enough,
\begin{equation}
    \mathbb P_\beta[\mathcal C(\varepsilon,\delta,n,\ell)] > c_0,
\end{equation}
and
\begin{equation}
\mathbb P_\beta[\mathcal D(\varepsilon,\delta,n,\ell)]\geq 1-2\delta.  
\end{equation}
\end{Claim}

Fix $\varepsilon>0$. Let $\delta\in (0,1)$ to be chosen small enough, and $\ell$ to be chosen large enough (in terms of $\varepsilon$ and $\delta$). Let us work on the event 
\begin{equation}\label{eq:intersection of events}
\mathcal A(\varepsilon,\delta/2,n,\ell)\cap \mathcal P(\varepsilon,n,\ell)\cap \mathcal C(\varepsilon,\delta,n,\ell) \cap \mathcal D(\varepsilon, \delta, n,\ell)\cap \{\mathsf{b}_{n,\ell}^1\geq \mathsf{b}_{n,\ell}^0\}.
\end{equation}
We will need the following technical claim which allows to transfer information of $\overline{k}^{\xi,\mathsf{b}}$ into information on $k^\xi$.

\begin{Claim}\label{claim: from k bar to k} Assume that the event on \eqref{eq:intersection of events} occurs. Let $\mathsf{a}\in (\mathbb R^+)^{R(n,\ell)}$ and for $i \in \{0,1\}$ define $\mathsf{a}^i$ as follows: $\mathsf{a}_x^i=\mathsf{a}_x$ for every $x\in R(n,\ell)$, and $\mathsf{a}^i_y=(\mathsf{b}_{n,\ell}^i)_y$ for every $y\in \partial^{\textup{ext},+}R(n,\ell)$. Let $\omega\in \overline{E}^+(R(n,\ell))$ such that $(\omega,\mathsf{a}^0)$ is $(\xi_{n,\ell}^0,\mathsf{b}_{n,\ell}^0)$-compatible. Finally, let $\omega'\in \overline{E}^+(R(n,\ell))$ be a percolation configuration which potentially differs from $\omega$ only on edges having one endpoint in $\{y\in \partial^{\textup{ext},+}R(n,\ell):(\mathsf{b}^0_{n,\ell})_y=0 \textup{ and }(\mathsf{b}_{n,\ell}^1)_y>0\}$. Then,
 \begin{equation}
    \Big|k^{\xi^1_{n,\ell}}(\omega',\mathsf{a}^1)-k^{\xi^0_{n,\ell}}(\omega,\mathsf{a}^0)\Big|\leq 2\varepsilon n^{d-1},
\end{equation}
\end{Claim}
\begin{proof} Observe that $(\omega',\mathsf{a}^1)$ is $(\xi^1_{n,\ell},\mathsf{b}_{n,\ell}^1)$-compatible and that $\omega'\succcurlyeq \omega$. As a consequence,
\begin{multline}
    k^{\xi^1_{n,\ell}}(\omega',\mathsf{a}^1)+|\{y: (\mathsf{b}^1_{n,\ell})_1=0\}|= \overline{k}^{(\xi^1_{n,\ell},\mathsf{b}^1_{n,\ell})}(\omega',\mathsf{a})\leq \overline{k}^{(\xi^1_{n,\ell},\mathsf{b}^1_{n,\ell})}(\omega,\mathsf{a})\leq \overline{k}^{(\xi^0_{n,\ell},\mathsf{b}^0_{n,\ell})}(\omega,\mathsf{a})\\= k^{\xi^0_{n,\ell}}(\omega,\mathsf{a}^0)+|\{y: (\mathsf{b}^0_{n,\ell})_y=0\}|,
\end{multline}
where we used \eqref{eq: k bar are ordered} in the second inequality. Under the occurrence of $\mathcal{A}_0(\varepsilon,n,\ell)$, one has $|\{y: (\mathsf{b}^0_{n,\ell})_y=0\}|-|\{y: (\mathsf{b}^1_{n,\ell})_1=0\}|\leq \varepsilon n^{d-1}$. Combined with the above display, this yields
\begin{equation}
    k^{\xi^1_{n,\ell}}(\omega',\mathsf{a}^1)\leq k^{\xi^0_{n,\ell}}(\omega,\mathsf{a}^0)+\varepsilon n^{d-1}.
\end{equation}
Now, under the occurrence of $\mathcal{P}(\varepsilon,n,\ell)$ one has $\overline{k}^{(\xi^0_{n,\ell},\mathsf{b}^0_{n,\ell})}(\omega,\mathsf{a})\leq \overline{k}^{(\xi^1_{n,\ell},\mathsf{b}^1_{n,\ell})}(\omega,\mathsf{a})+\varepsilon n^{d-1}$. Hence,
\begin{multline}
    k^{\xi^0_{n,\ell}}(\omega,\mathsf{a}^0)+|\{y: (\mathsf{b}^0_{n,\ell})_y=0\}|=\overline{k}^{(\xi^0_{n,\ell},\mathsf{b}^0_{n,\ell})}(\omega,\mathsf{a})\leq \overline{k}^{(\xi^1_{n,\ell},\mathsf{b}^1_{n,\ell})}(\omega,\mathsf{a})+\varepsilon n^{d-1}\\=k^{\xi^1_{n,\ell}}(\omega,\mathsf{a}^1)+|\{y: (\mathsf{b}^1_{n,\ell})_y=0\}|+\varepsilon n^{d-1},
\end{multline}
so that
\begin{equation}
    k^{\xi^0_{n,\ell}}(\omega,\mathsf{a}^0)\leq k^{\xi^1_{n,\ell}}(\omega,\mathsf{a}^1)+\varepsilon n^{d-1}.
\end{equation}
It remains to compare $k^{\xi^1_{n,\ell}}(\omega,\mathsf{a}^1)$ to $k^{\xi^1_{n,\ell}}(\omega',\mathsf{a}^1)$. By definition $\omega'$ and $\omega$ can only differ on edges having an endpoint in $\{y\in \partial^{\textup{ext},+}R(n,\ell):(\mathsf{b}^0_{n,\ell})_y=0 \textup{ and }(\mathsf{b}_{n,\ell}^1)_y>0\}$. Again, since $\mathcal{A}_0(\varepsilon,n,\ell)$ occurs, this set has cardinality smaller than $\varepsilon n^{d-1}$, which gives that 
\begin{equation}
    k^{\xi^1_{n,\ell}}(\omega,\mathsf{a}^1)\leq k^{\xi^1_{n,\ell}}(\omega',\mathsf{a}^1)+\varepsilon n^{d-1},
\end{equation}
and concludes the proof.
\end{proof}

By Lemmas \ref{lem:abs value approx}--\ref{lem:partitions approx} and Claim \ref{claim:mathcalCD}, we can find $\delta>0$ small enough and $\ell\geq 1$ large enough such that, for $n$ large enough, the event in \eqref{eq:intersection of events} has positive probability.
On this event, recall that $\mathbf P^{\xi^1_{n,\ell}, \mathsf b^1_{n,\ell},h}_{R(n,\ell), \beta, \beta'}[D_{n,\ell}] \geq 1-\delta/2$, where
\begin{equation}
    D_{n,\ell}=D_{n,\ell}(\mathsf{b}_{n,\ell}^1,\mathsf{b}_{n,\ell}^0):= \Big\{ \sum_{xy \in \partial^{+}_E R(n,\ell)} \mathsf a_x \Big((\mathsf b^1_{n,\ell})_y - (\mathsf b^0_{n,\ell})_y\Big)\leq \varepsilon n^{d-1}\Big\}.
\end{equation}
Our goal is to establish that
\begin{equation} \label{eq: claim A_nD_n bound}
\mathbf P^{\xi^1_{n,\ell}, \mathsf{b}^1_{n,\ell}, h}_{R(n,\ell),\beta,\beta'}[A_n \cap D_{n,\ell}] \leq 2^{5\varepsilon n^{d-1}+1} e^{3\beta \varepsilon n^{d-1}}  \mathbf P^{\xi^0_{n,\ell}, \mathsf{b}^0_{n,\ell}, h}_{R(n,\ell),\beta,\beta'}[A_n ].
\end{equation}
Intuitively, this follows by two observations. First, the individual weights in the coupling can be compared up to costs that are exponential in the sum of the absolute values along the boundary, which are in turn controlled by $D_{n,\ell}$. Second, the cardinalities of the sets we sum $\sigma$ over are determined by the number of connected clusters under each measure in \eqref{eq: claim A_nD_n bound}, 
which in turn are comparable up to $2\varepsilon n^{d-1}$ thanks to Claim \ref{claim: from k bar to k}. 

Before proving \eqref{eq: claim A_nD_n bound}, let us first see that it implies the result. Indeed, we can argue almost exactly as we did in the full space argument. By the definition of the events $\mathcal{A}(\varepsilon,\delta/2,n,\ell)$ and $\mathcal{C}(\varepsilon,\delta,n,\ell)$, if $n$ is large enough,
\begin{equation} \label{eq:D_n}
\frac{\delta}{2}\leq \mathbf P^{\xi^1_{n,\ell}, \mathsf b^1_{n,\ell},h}_{R(n,\ell), \beta,\beta'}[A_n \cap D_{n,\ell}] \leq 2^{5\varepsilon n^{d-1}+1} e^{3\beta \varepsilon n^{d-1}} \mathbf P^{\xi^0_{n,\ell}, \mathsf b^0_{n,\ell}, h}_{R(n,\ell),\beta,\beta'}[A_n].
\end{equation}
Moreover,
\begin{equation}
\mathbf P^{\xi^0_{n,\ell}, \mathsf b^0_{n,\ell}, h}_{R(n,\ell),\beta,\beta'}[A_n]=\frac{\mathbf E^{\xi^0_{n,\ell}, \mathsf b^0_{n,\ell}, h}_{R(n,\ell),\beta}[G(\tau) \mathbbm{1}_{A_n}]}{\mathbf E^{\xi^0_{n,\ell}, \mathsf b^0_{n,\ell}, h}_{R(n,\ell),\beta}[G(\tau)]}\le \frac{\mathbf E^{\xi^0_{n,\ell}, \mathsf b^0_{n,\ell}, h}_{R(n,\ell),\beta}[G(\tau) \mathbbm{1}_{A_n}]}{\mathbf E^{\xi^0_{n,\ell}, \mathsf b^0_{n,\ell}, h}_{R(n,\ell),\beta}[G(\tau)\mathbbm{1}_{B_n}]},
\end{equation}
where $G(\tau)=\prod_{xy\in E_n} e^{-(\beta-\beta')\tau_x \tau_y}$. Thus, using \eqref{eq:D_n} and the definitions of $A_n$ and $B_n$, we get
\begin{equation}
\frac{\delta}{2}\le \frac{e^{(3\beta+5\log 2) \varepsilon n^{d-1}+\log 2}}{\delta} e^{|E_n|(\beta'-\beta)(a_{uv}-b_{uv}-2\varepsilon)}.
\end{equation}
Since $\beta'<\beta$ and $|E_n|\geq n^{d-1}$, by taking $n$ to be large enough we get that 
\begin{equation}
a_{uv}-b_{uv}-2\varepsilon \leq  C\varepsilon,
\end{equation}
where $C=\tfrac{3\beta+5\log 2}{\beta-\beta'}$. Since $\varepsilon>0$ is arbitrary, we conclude that $a_{uv}(\beta',\beta)\leq b_{uv}(\beta)$, which implies \eqref{eq:2 point ineq hs}, as desired.

It remains to prove \eqref{eq: claim A_nD_n bound}. Recall \eqref{eq:formula ES} and fix boundary conditions $(\xi^1_{n,\ell}, \mathsf b^1_{n,\ell})$ and $(\xi^0_{n,\ell}, \mathsf b^0_{n,\ell})$. Consider $\mathsf{a}^1$ and $\mathsf{a}^0$ whose values on $\partial^{\textup{ext},+}R(n,\ell)$ coincide with $\mathsf{b}^1_{n,\ell}$ and $\mathsf{b}^0_{n,\ell}$, respectively, and such that $\mathsf a^0=\mathsf a^1=\mathsf{a}$ in $R(n,\ell)$. Let $(\omega,\sigma)$ be such that $(\mathsf{a}^1,\omega,\sigma)$ is $\xi^1_{n,\ell}$-compatible. To this triplet, we can associate $(\mathsf{a}^0,\tilde \omega,\tilde \sigma)$ which is also $\xi^0_{n,\ell}$-compatible, and defined as follows. For every $e\in E(R(n,\ell))$, we let $\tilde \omega_e=\omega_e$; and for every $e=xy\in \partial^{+}_E R(n,\ell)$ with $x\in R(n,\ell)$ and $y\notin R(n,\ell)$, we let $\tilde \omega_e=\omega_e$ if $\mathsf{a}^0_y\neq 0$, and $\tilde \omega_e=0$ otherwise. Moreover, for $x\in {R}(n,\ell)$ we let
$\tilde\sigma_x=\sigma_x$, and for $y\in \partial^{\textup{ext},+}R(n,\ell)$, we let $\tilde \sigma_y=\sigma_y$ if $\mathsf{a}^0_y\neq 0$, and $\tilde \sigma_y=0$ otherwise.

We now perform some calculations with the aim of comparing $\mathrm{d}\mathbf P^{\xi^1_{n,\ell},\mathsf b^1_{n,\ell}, h}_{R(n,\ell), \beta,\beta'}[(\mathsf{a}^1,\omega,\sigma)]$ with $\mathrm{d}\mathbf P^{\xi^0_{n,\ell},\mathsf b^0_{n,\ell}, h}_{R(n,\ell), \beta,\beta'}[(\mathsf{a}^0,\tilde\omega,\tilde\sigma)]$ on the event $A_n\cap D_{n,\ell}$. Recall that we work under the event in \eqref{eq:intersection of events}. 

First, we bound $\mathds{1}_{(\mathsf{a}^1,\sigma)\in A_n\cap D_{n,\ell}} Z^{\xi^1_{n,\ell},\mathsf{b}_{n,\ell}^1,h}_{R(n,\ell),\beta,\beta'}\mathrm{d}\mathbf P^{\xi^1_{n,\ell},\mathsf b^1_{n,\ell}, h}_{R(n,\ell), \beta,\beta'}[(\mathsf{a}^1,\omega,\sigma)]$ in terms of a multiple of $\mathds{1}_{(\mathsf{a}^0,\tilde \sigma)\in A_n} Z^{\xi^0_{n,\ell},\mathsf{b}_{n,\ell}^0,h}_{R(n,\ell),\beta,\beta'}\mathrm{d}\mathbf P^{\xi^0_{n,\ell},\mathsf b^0_{n,\ell}, h}_{R(n,\ell), \beta,\beta'}[(\mathsf{a}^0,\tilde \omega,\tilde \sigma)]$, where $(\tilde \omega,\tilde \sigma)$ was defined above. Recall \eqref{eq:formula ES}. Note that if an edge $xy\in \overline{E}^+(R(n,\ell))$ is such that $\sigma_x\neq \sigma_y$, then any compatible $\omega$ satisfies $\omega_{xy}=0$. Thus, we trivially have 
\begin{equation} \label{eq:pf5}
\left(\frac{p_{xy}(\mathsf{a}^1)}{1-p_{xy}(\mathsf{a}^1)
}\right)^{\omega_{xy}}\leq \left(\frac{p_{xy}(\mathsf{a}^0)}{1-p_{xy}(\mathsf{a}^0)}\right)^{\tilde\omega_{xy}}. 
\end{equation}
Otherwise, both $\omega_{xy}=1$ and $\omega_{xy}=0$ may occur, and the monotonicity is less obvious. Let us first restrict to edges on which $\omega_{xy}=\tilde \omega_{xy}$. Observe that if we sum over these possibilities, we have that
\begin{multline} \label{eq:pf6}
\left(\frac{p_{xy}(\mathsf{a}^1)}{1-p_{xy}(\mathsf{a}^1)}\right)^{\omega_{xy}}+\left(\frac{p_{xy}(\mathsf{a}^1)}{1-p_{xy}(\mathsf{a}^1)}\right)^{1-\omega_{xy}}\\\leq e^{2\beta(\mathsf{a}^1_x \mathsf{a}^1_y-\mathsf{a}^0_x \mathsf{a}^0_y)}\left(\left(\frac{p_{xy}(\mathsf{a}^0)}{1-p_{xy}(\mathsf{a}^0)}\right)^{\tilde\omega_{xy}}+\left(\frac{p_{xy}(\mathsf{a}^0)}{1-p_{xy}(\mathsf{a}^0)}\right)^{1-\tilde\omega_{xy}}\right),
\end{multline}
where the inequality comes from edges with parameter $\beta'$, i.e. using that $\beta'\leq \beta$. Again, similar bounds hold for edges of the form $x\fg$ with $x\in R(n,\ell)$. If now we work at a boundary edge $xy$ such that $\tilde \omega_{xy}=0$ and $\omega_{xy}=1$, then
\begin{equation}
    \left(\frac{p_{xy}(\mathsf{a}^1)}{1-p_{xy}(\mathsf{a}^1)}\right)^{\omega_{xy}}\leq e^{2\beta(\mathsf{a}^1_x \mathsf{a}^1_y-\mathsf{a}^0_x \mathsf{a}^0_y)}\left(\frac{p_{xy}(\mathsf{a}^0)}{1-p_{xy}(\mathsf{a}^0)}\right)^{\tilde\omega_{xy}}.
\end{equation}
Combining \eqref{eq:pf2} and \eqref{eq:pf5}--\eqref{eq:pf6}, we get that, for every triplet $(\mathsf{a}^1,\mathsf{a}^0,\sigma)$ as before,
\begin{align}\label{eq:pf7}
\begin{aligned}
    \mathds{1}_{(\mathsf{a}^1,\sigma)\in A_n\cap D_{n,\ell}}&\sum_{\omega}\prod_{xy\in \overline{E}^+(R(n,\ell))} \sqrt{1-p_{xy}(\mathsf{a}^1)}\left(\frac{p_{xy}(\mathsf{a}^1)}{1-p_{xy}(\mathsf{a}^1)}\right)^{\omega_{xy}} 
    \\&\qquad\qquad\prod_{x\in R(n,\ell)\cap \partial \mathbb H} \sqrt{1-p_{x\fg}(\mathsf a^1)}\left(\frac{p_{x\fg}(\mathsf a^1)}{1-p_{x\fg}(\mathsf a^1)}\right)^{\omega_{x\fg}}
    \\&\leq e^{2\beta \varepsilon n^{d-1}}\mathds{1}_{(\mathsf{a}^0,\tilde\sigma)\in A_n}\sum_{\omega} \prod_{xy\in \overline{E}^+(R(n,\ell))}\sqrt{1-p_{xy}(\mathsf{a}^0)}\left(\frac{p_{xy}(\mathsf{a}^0)}{1-p_{xy}(\mathsf{a}^0)}\right)^{\tilde \omega_{xy}} 
    \\& \qquad\qquad\prod_{x\in R(n,\ell)\cap \partial \mathbb H} \sqrt{1-p_{x\fg}(\mathsf a^0)}\left(\frac{p_{x\fg}(\mathsf a^0)}{1-p_{x\fg}(\mathsf a^0)}\right)^{\tilde \omega_{x\fg}},
    \end{aligned}
\end{align}
where above we are summing over  $\omega$ such that $(\omega, \mathsf{a}^1,\sigma)$ is compatible. Now, observe that under the occurrence of $\mathcal{A}_0(\varepsilon,n,\ell)$, any element in the image of $\omega\mapsto \tilde
 \omega$ has at most $2^{\varepsilon n^{d-1}}$ preimages.
A similar observation holds for the map $\sigma\mapsto \tilde \sigma$.
Integrating over $(\mathsf{a}^1,\mathsf{a}^0,\sigma)$ above therefore yields
\begin{equation}\label{eq:pf8}
     Z^{\xi^1_{n,\ell},\mathsf{b}_{n,\ell}^1,h}_{R(n,\ell),\beta,\beta'}\mathbf P^{\xi^1_{n,\ell},\mathsf b^1_{n,\ell}, h}_{R(n,\ell), \beta,\beta'}[A_n\cap D_n]\leq  e^{2\beta \varepsilon n^{d-1}} 2^{2\varepsilon n^{d-1}}\cdot Z^{\xi^0_{n,\ell},\mathsf{b}_{n,\ell}^0,h}_{R(n,\ell),\beta,\beta'} \mathbf P^{\xi^0_{n,\ell},\mathsf b^0_{n,\ell}, h}_{R(n,\ell), \beta,\beta'}[A_n].
\end{equation}

Next, we prove that, if $n$ is large enough
\begin{equation}\label{eq:pf4.5}
    Z^{\xi^0_{n,\ell},\mathsf{b}_{n,\ell}^0,h}_{R(n,\ell),\beta,\beta'}\leq 2\cdot 2^{3\varepsilon n^{d-1}}e^{\beta \varepsilon n^{d-1}}Z^{\xi^1_{n,\ell},\mathsf{b}_{n,\ell}^1,h}_{R(n,\ell),\beta,\beta'}.
\end{equation}
It suffices to prove that
\begin{equation}\label{eq:pf1}
Z^{\xi^0_{n,\ell},\mathsf{b}_{n,\ell}^0,h}_{R(n,\ell),\beta,\beta'}[D_{n,\ell}]\leq 2^{3\varepsilon n^{d-1}}e^{\beta \varepsilon n^{d-1}}Z^{\xi^1_{n,\ell},\mathsf{b}_{n,\ell}^1,h}_{R(n,\ell),\beta,\beta'}.
\end{equation}
Indeed, since we work under $\mathcal{A}(\varepsilon,\delta/2,n,\ell)$ with $\delta\in (0,1)$, we have that $\mathbf P^{\xi^1_{n,\ell},\mathsf{b}_{n,\ell}^1,h}_{R(n,\ell),\beta,\beta'}[D_{n,\ell}]\geq 1-\delta/2\geq \tfrac{1}{2}$ if $n$ is large enough. Since $D_{n,\ell}$ is decreasing in $\mathsf{a}^i$, Proposition \ref{prop: random cluster properties} gives that $\mathbf P^{\xi^0_{n,\ell},\mathsf{b}_{n,\ell}^0,h}_{R(n,\ell),\beta,\beta'}[D_{n,\ell}]\geq \mathbf P^{\xi^1_{n,\ell},\mathsf{b}_{n,\ell}^1,h}_{R(n,\ell),\beta,\beta'}[D_{n,\ell}]\geq \tfrac{1}{2}$, so that (by definition) $Z^{\xi^0_{n,\ell},\mathsf{b}_{n,\ell}^0,h}_{R(n,\ell),\beta,\beta'}[D_{n,\ell}]\geq \tfrac{1}{2}Z^{\xi^0_{n,\ell},\mathsf{b}_{n,\ell}^0,h}_{R(n,\ell),\beta,\beta'}$. This last inequality gives \eqref{eq:pf4.5} when combined with \eqref{eq:pf1}.

We turn to the proof of \eqref{eq:pf1}. We use the following mapping. If $\omega$ is such that $(\omega,\mathsf{a}^0)$ is $(\xi^0_{n,\ell},\mathsf{b}^0_{n,\ell})$-compatible, then we can associate at most $2^{\varepsilon n^{d-1}}$ configurations $\omega'$ such that $(\omega',\mathsf{a}^1)$ is $(\xi^1_{n,\ell},\mathsf{b}^1_{n,\ell})$-compatible constructed as follows: they coincide with $\omega$ on $E(R(n,\ell))$, and potentially differ from it only on the set $\{xy\in \partial^+_E R(n,\ell): x\in R(n,\ell), \: y\notin R(n,\ell), \: \mathsf{a}^0_y=0 \textup{ and }\mathsf{a}^1_y>0\}$, which is of cardinality at most $\varepsilon n^{d-1}$ under the occurrence of $\mathcal{A}_0(\varepsilon,n,\ell)$. Note that any $\omega'$ such that $(\omega',\mathsf{a}^1)$ is $(\xi^1_{n,\ell},\mathsf{b}^1_{n,\ell})$-compatible can be obtained according to the above procedure.

Observe that, for $xy\in \overline{E}^+(R(n,\ell))$,
\begin{equation} \label{eq:pf2}
\sqrt{1-p_{xy}(\mathsf a^0)}= e^{\beta(\mathsf{a}^1_x \mathsf{a}^1_y- \mathsf{a}^0_x \mathsf{a}^0_y)} \sqrt{1-p_{xy}(\mathsf{a}^1)},
\end{equation}
and, for any $\omega'$ obtained from $\omega$, 
\begin{equation}\label{eq:pf3}
\left(\frac{p_{xy}(\mathsf{a}^0)}{1-p_{xy}(\mathsf{a}^0)
}\right)^{\omega_{xy}}\leq \left(\frac{p_{xy}(\mathsf{a}^1)}{1-p_{xy}(\mathsf{a}^1)}\right)^{\omega_{xy}'}.
\end{equation}
Similar bounds hold for edges of the form $x\fg$ with $x\in R(n,\ell)$.
Notice also that for every $xy\in \partial^{+}_E R(n,\ell)$ with $x\in R(n,\ell)$ and $y\notin R(n,\ell)$, $\mathsf{a}^1_x \mathsf{a}^1_y- \mathsf{a}^0_x \mathsf{a}^0_y= \mathsf{a}^1_x((\mathsf{b}^1_{n,\ell})_y-(\mathsf{b}^0_{n,\ell})_y)$, since $\mathsf{a}^1=\mathsf{a}^0$ in $R(n,\ell)$. Hence,
\begin{equation}\label{eq:pf4}
\sum_{xy \in \partial^{+}_E R(n,\ell)}(\mathsf{a}^1_x \mathsf{a}^1_y- \mathsf{a}^0_x \mathsf{a}^0_y)\leq \varepsilon n^{d-1},
\end{equation}
by definition of the event $D_{n,\ell}$. Moreover, by Claim \ref{claim: from k bar to k} one has that $|k^{\xi^1_{n,\ell}}(\omega',\mathsf{a}^1)-k^{\xi^0_{n,\ell}}(\omega,\mathsf{a}^0)|\leq 2\varepsilon n^{d-1}$. Finally, the set of $\omega'$ obtained from a given $\omega$ is of cardinality at most $2^{\varepsilon n^{d-1}}$ under the occurrence of $\mathcal{P}(\varepsilon,n,\ell)$.
Hence, combining \eqref{eq:pf2}--\eqref{eq:pf4} with the above observations, we obtain \eqref{eq:pf1}.

Combining \eqref{eq:pf8} and \eqref{eq:pf4.5} gives
\begin{equation}
\mathbf P^{\xi^1_{n,\ell},\mathsf b^1_{n,\ell}, h}_{R(n,\ell), \beta,\beta'}[A_n\cap D_{n,\ell}]\leq 2\cdot  2^{5\varepsilon n^{d-1}}e^{3\beta\varepsilon n^{d-1}}\mathbf P^{\xi^0_{n,\ell},\mathsf b^0_{n,\ell}, h}_{R(n,\ell), \beta,\beta'}[A_n],
\end{equation}
which concludes the proof.
\end{proof}

We now prove Claim \ref{claim:mathcalCD}.
\begin{proof}[Proof of Claim~\textup{\ref{claim:mathcalCD}}] Let $\varepsilon>0$ and $\ell\geq 1$. Let $\delta>0$ to be chosen small enough. Recall the definitions of $A_n,B_n$ from \eqref{eq: def a_n b_n half space}.
Let us first bound $\mathbb P_\beta[\mathcal D(\varepsilon,\delta,n,\ell)]$. Observe that the measure $\nu^{0,h}_{\mathbb H,\beta}$ is ergodic on even functions with respect to translations in $\Gamma$. Indeed, this follows from the exact same argument as \cite[Proposition~2.10]{GunaratnamPanagiotisPanisSeveroPhi42022}. Therefore, Birkhoff's ergodic theorem implies that, almost surely
\begin{equation}
\lim_{n \rightarrow \infty} \frac{1}{|E_n|} \sum_{xy \in E_n} \tau_x \tau_y = \langle \tau_u\tau_v\rangle^{0,h}_{\mathbb H,\beta}= b_{uv}. 
\end{equation}
Hence $\nu^{0,h}_{\mathbb H, \beta}[B_n]=1-o(1)$, where $o(1)$ tends to $0$ as $n$ tends to infinity. Furthermore, observe that by the domain Markov property,
\begin{equation}
\nu^{0,h}_{\mathbb H, \beta}[B_n] = \mathbb E_\beta\Big[\mathbf P^{\xi^0_{n,\ell}, \mathsf b^0_{n,\ell},h}_{R(n,\ell),\beta}[B_n]\Big]\leq \mathbb P_\beta[\mathcal D(\varepsilon,\delta,n,\ell)] + \delta. 
\end{equation}
As a result, if $n$ is large enough, $\mathbb P_\beta[\mathcal{D}(\varepsilon,\delta,n,\ell)]\geq 1-2\delta$.

We now analyse $\mathbb P_\beta[\mathcal C(\varepsilon,\delta,n,\ell)]$. Note that, by applying Birkhoff's ergodic theorem as above, we could get that $\mathbb P_{\beta,\beta'}[\mathcal C(\varepsilon,\delta,n,\ell)]=1-o(1)$, where $\mathbb P_{\beta,\beta'}$ is the monotone coupling between $\Psi^{0,h}_{\mathbb H, \beta, \beta'}$ and $\Psi^{1,h}_{\mathbb H, \beta, \beta'}$. However, it is not clear how to compare this with $\mathbb P_\beta[\mathcal C(\varepsilon,\delta,n,\ell)]$. Instead, we establish a weaker bound, relying on Lemma \ref{lem:random variables}, as we did in the full-space case. Let $\mathbb P = \mathbb P_\beta\big[\mathbf P^{\xi^1_{n,\ell}, \mathsf b^1_{n,\ell}, h}_{R(n,\ell),\beta,\beta'}[\cdot]\big]$. Let the $X_i$ correspond to the random variables $\tau_x\tau_y$ where $xy$ ranges over $E_n$. Note that, since $\mathbf P^{\xi^1_{n,\ell}, \mathsf b^1_{n,\ell}, h}_{R(n,\ell),\beta,\beta'}$ is measurable only with respect to the $(\mathsf a^1,\omega^1)$ marginal, we can apply the monotonicity in parameters of Proposition \ref{prop: wired random cluster properties} to obtain
\begin{equation}
\mathbb E[\tau_x\tau_y]=\mathbb E_\beta\Big[\mathbf E^{\xi^1_{n,\ell}, \mathsf b^1_{n,\ell}, h}_{R(n,\ell),\beta,\beta'}[\tau_x\tau_y]\Big] \geq \mathbb E_{\beta,\beta'}\Big[\mathbf E^{\xi^1_{n,\ell}, \mathsf b^1_{n,\ell}, h}_{R(n,\ell),\beta,\beta'}[\tau_x\tau_y]\Big] = a_{uv},
\end{equation}
where the last equality follows by the domain Markov property.
Furthermore, using the monotonicity in parameters of Proposition \ref{prop: random cluster properties} and the regularity of Proposition \ref{prop:reg_plus half-space}, we obtain that there exists $M>0$ large enough such that
\begin{equation}
\mathbb E[|\tau_x\tau_y|\mathds{1}_{|\tau_x\tau_y|\geq M}]\leq \nu_{\mathbb H,\beta}^{+,h}[|\tau_x\tau_y|\mathds{1}_{|\tau_x\tau_y|\geq M}]\leq \frac{\varepsilon}{2}.
\end{equation}
By Lemma \ref{lem:random variables}, we find that
\begin{equation}
    \mathbb P[A_n]\geq \frac{(\varepsilon/2)}{M-a_{uv}+\varepsilon}=:c_1>0.
\end{equation}
Then,
\begin{equation}
    \mathbb P[A_n]=\mathbb P_\beta\Big[\mathbf P^{\xi^1_{n,\ell}, \mathsf b^1_{n,\ell}, h}_{R(n,\ell),\beta,\beta'}[A_n]\Big]\leq \mathbb P_{\beta}[\mathcal{C}(\varepsilon,\delta,n,\ell)]+\delta.
\end{equation}
Choosing $\delta>0$ small enough so that $\delta\leq \tfrac{c_1}{2}$ and setting $c_0:=\tfrac{c_1}{2}$ concludes the proof.
\end{proof}

\subsection{Bulk absolute values are approximately equal}

We now turn to the proof of Lemma \ref{lem:abs value approx}. 

\begin{proof}[Proof of Lemma \textup{\ref{lem:abs value approx}}] Assume that $\beta>0$ is such that $\beta\notin \Xi$. Fix $\varepsilon,\delta>0$. Let $\ell\geq 1$ to be chosen large enough. We will prove that, for $n$ large enough,
\begin{equation}\label{eq:A_0 and A_1 ineq}
    \mathbb P_\beta[\mathcal{A}_0(\varepsilon,n,\ell)]\geq 1-\delta/2, \qquad \mathbb P_\beta[\mathcal{A}_1(\varepsilon,\delta,n,\ell)]\geq 1-\delta/2,
\end{equation}
starting with the first inequality.

Introduce $X=\sum_{y\in \partial^{\textup{ext},+} R(n,\ell)}(\mathds{1}_{\mathsf{a}^0_y=0}-\mathds{1}_{\mathsf{a}^1_y=0})$ and note that $X\geq 0$ $\mathbb P_\beta$-almost surely. By Markov's inequality,
\begin{equation}\label{eq:pf a0 1}
    \mathbb P_\beta[\mathcal{A}_0(\varepsilon,n,\ell)^c]=\mathbb P_\beta[X>\varepsilon n^{d-1}]\leq \frac{\mathbb E_\beta[X]}{\varepsilon n^{d-1}}.
\end{equation}
Observe that there exists $C_1=C_1(d)>0$ such that,
\begin{equation}\label{eq:pf a0 2}
    \mathbb E_\beta[X]\leq C_1n^{d-1}\Big(\Psi^0_{\mathbb H,\beta}[\mathsf{a}_{\ell' \mathbf{e}_d}=0]-\Psi^1_{\mathbb H,\beta}[\mathsf{a}_{\ell' \mathbf{e}_d}=0]\Big) +C_1\ell n^{d-2},
\end{equation}
where we used translation invariance in the hyperplane $y_d=\ell':=\ell+1$.
Moreover, if $\ell$ is large enough, 
\begin{equation}\label{eq:pf a0 3}
    \Psi^0_{\mathbb H,\beta}[\mathsf{a}_{\ell' \mathbf{e}_d}=0]\leq \Psi^0_{\Lambda_{\ell'/2}(\ell' \mathbf{{e}_d)},\beta}[\mathsf{a}_{\ell' \mathbf{e}_d}=0], \qquad \Psi^1_{\mathbb H,\beta}[\mathsf{a}_{\ell' \mathbf{e}_d}=0]\geq \Psi^1_{\Lambda_{\ell'/2}(\ell' \mathbf{{e}_d)},\beta}[\mathsf{a}_{\ell' \mathbf{e}_d}=0],
\end{equation}
where we used the domain Markov property and monotonicity in boundary conditions in the first inequality, and Proposition \ref{prop: monot plus} in the second one. Since $\beta\notin \Xi$, one has
\begin{equation}\label{eq:pf a0 4}
    \Psi^0_{\mathbb H,\beta}[\mathsf{a}_{\ell' \mathbf{e}_d}=0]-\Psi^1_{\mathbb H,\beta}[\mathsf{a}_{\ell' \mathbf{e}_d}=0]\leq \Psi^0_{\Lambda_{\ell'/2}(\ell' \mathbf{{e}_d)},\beta}[\mathsf{a}_{\ell' \mathbf{e}_d}=0]-\Psi^1_{\Lambda_{\ell'/2}(\ell' \mathbf{{e}_d)},\beta}[\mathsf{a}_{\ell' \mathbf{e}_d}=0]\underset{\ell\rightarrow \infty}{\longrightarrow}0.
\end{equation}
As a result, if $\ell$ is large enough so that $\Psi^0_{\mathbb H,\beta}[\mathsf{a}_{\ell' \mathbf{e}_d}=0]-\Psi^1_{\mathbb H,\beta}[\mathsf{a}_{\ell' \mathbf{e}_d}=0]\leq (\delta/2)(\varepsilon/2)C_1^{-1}$, and $n$ is large enough so that $C_1\ell n^{d-2}\leq (\delta/2)(\varepsilon/2)n^{d-1}$, we find that $\mathbb E_{\beta}[X]\leq (\delta/2)\varepsilon n^{d-1}$. Plugging this inequality in \eqref{eq:pf a0 1} yields that $\mathbb P_\beta[\mathcal{A}_0(\varepsilon,n,\ell)^c]\leq \delta/2$.

We now turn to the second inequality in \eqref{eq:A_0 and A_1 ineq}. Recall that, for a sample of $(\mathsf{b}_{n,\ell}^1,\mathsf{b}_{n,\ell}^0)$, we have defined
\begin{equation}
    D_{n,\ell} =  D_{n,\ell}(\mathsf{b}_{n,\ell}^1,\mathsf{b}_{n,\ell}^0)= \Big\{ \sum_{xy \in \partial^{+}_E R(n,\ell)} \mathsf a_x \Big((\mathsf b^1_{n,\ell})_y - (\mathsf b^0_{n,\ell})_y\Big)\leq \varepsilon n^{d-1}\Big\}.
\end{equation}
Introduce, for some $C>0$ to be chosen large enough below,
\begin{equation}
I_{n,\ell}:=\Big\{\sum_{xy\in \partial^+_E R(n,\ell)} \mathsf{a}_x^2\leq Cn^{d-1}\Big\},
\end{equation}
and
\begin{equation}
\mathcal J_{n,\ell}:=\Big\{\sum_{xy\in \partial^+_E R(n,\ell)}\Big((\mathsf b^1_{n,\ell})_y - (\mathsf b^0_{n,\ell})_y\Big)^2\leq \frac{\varepsilon^2}{C}n^{d-1} \Big\}.
\end{equation}
We claim that it suffices to show that for $C$ large enough, and every $n$ large enough,
\begin{equation}\label{eq: D_n bound}
   \mathbb E_\beta\Bigl[ \mathbbm{1}_{\mathcal{J}_{n,\ell}} \mathbf P^{\xi^1_{n,\ell}, \mathsf b^1_{n,\ell},h}_{R(n,\ell),\beta}[I_{n,\ell}]\Bigr]\geq 1-(\delta/2)^2.
\end{equation}
Indeed, let us assume that \eqref{eq: D_n bound} holds. By the Cauchy--Schwarz inequality, for every $(\mathsf b^1_{n,\ell},\mathsf b^0_{n,\ell}) \in \mathcal J_{n,\ell}$, we have that 
\begin{equation}
    I_{n,\ell} \subset D_{n,\ell},
\end{equation}
so that, letting $\mathbb P = \mathbb P_\beta\big[\mathbf P^{\xi^1_{n,\ell}, \mathsf b^1_{n,\ell}, h}_{R(n,\ell),\beta}[\cdot]\big]$, one has
\begin{equation}
\mathbb P[D_{n,\ell}]\geq \mathbb E_\beta\Bigl[ \mathbbm{1}_{\mathcal{J}_{n,\ell}} \mathbf P^{\xi^1_{n,\ell}, \mathsf b^1_{n,\ell},h}_{R(n,\ell),\beta}[I_{n,\ell}]\Bigr]\geq 1-(\delta/2)^2.
\end{equation}
Thus, applying Markov's inequality to the complement of the event 
\begin{equation}
\mathcal A'_1(\varepsilon,\delta,n,\ell):= \Big\{ (\mathsf a^1, \omega^1,  \mathsf a^0, \omega^0) : \mathbf P^{\xi^1_{n,\ell}, \mathsf b^1_{n,\ell},h}_{R(n,\ell), \beta}[D^c_{n,\ell}]\leq (\delta/2)^{-1} \mathbb{P}[D^c_{n,\ell}] \Big\}, 
\end{equation}
we obtain 
\begin{equation}
\mathbb{P}_\beta [\mathcal A'_1(\varepsilon,\delta,n,\ell)] = 1- \mathbb P_\beta \Big[ \mathbf P^{\xi^1_{n,\ell}, \mathsf b^1_{n,\ell},h}_{R(n,\ell), \beta}[D^c_{n,\ell}] > (\delta/2)^{-1}\mathbb E_\beta\big[\mathbf P^{\xi^1_{n,\ell}, \mathsf b^1_{n,\ell},h}_{R(n,\ell), \beta}[D^c_{n,\ell}]\big] \Big]\geq 
1-\delta/2.
\end{equation}
Since $\mathsf{b}^0_{n,\ell}\leq \mathsf{b}^1_{n,\ell}$, $D_{n,\ell}$ is a decreasing event (for $\mathsf{a}^1$) so that we have $\mathbf P^{\xi^1_{n,\ell}, \mathsf b^1_{n,\ell},h}_{R(n,\ell), \beta,\beta'}[D_{n,\ell}]\geq \mathbf P^{\xi^1_{n,\ell}, \mathsf b^1_{n,\ell},h}_{R(n,\ell), \beta}[D_{n,\ell}]$ by Proposition \ref{prop:stoc_monotonicity}. Combining the above with the fact that $(\delta/2)^{-1} \mathbb{P}[D^c_{n,\ell}]\leq (\delta/2) $ by our assumption \eqref{eq: D_n bound}, we get that
\begin{equation}
\mathbb{P}_{\beta}[\mathcal A_1(\varepsilon,\delta,n,\ell)]\geq \mathbb{P}_{\beta}[\mathcal A'_1(\varepsilon,\delta,n,\ell)]\geq 1- \delta/2.
\end{equation}

We now turn to the proof of \eqref{eq: D_n bound}. By a union bound it suffices to prove that
\begin{equation}
    \mathbb P_\beta[\mathcal{J}_{n,\ell}^c]\leq \frac{\delta^2}{8}, \quad\textup{and}\quad\mathbb P_\beta\Big[\mathbf P^{\xi^1_{n,\ell}, \mathsf b^1_{n,\ell},h}_{R(n,\ell), \beta}[I_{n,\ell}^c]\Big]= \mathbb P_\beta[I_{n,\ell}^c]\leq \frac{\delta^2}{8},
\end{equation}
where for the second inequality we used the domain Markov property, noting that $I_{n,\ell}$ is measurable with respect to the $\mathsf a^1$-marginal. Observe that the bound on $\mathbb P_\beta[I_{n,\ell}^c]$ follows by regularity by choosing $C$ large enough. Below, we potentially take $C$ even larger.

We now turn to the bound on $\mathbb P_\beta [\mathcal J_{n,\ell}^c]$. Consider the events 
\begin{equation}
\mathcal F_{n,\ell}=\Big\{\sum_{xy\in \partial^+_E R(n,\ell): \: x_d=\ell}(\mathsf{a}^1_y)^2 \le \sum_{xy\in \partial^+_E R(n,\ell):\: x_d=\ell} \Big(\Psi^{1,h}_{\mathbb{H},\beta}[(\mathsf{a}_y)^2 ]+\frac{\varepsilon}{C}\Big)\Big\}, 
\end{equation}
\begin{equation}
\mathcal G_{n,\ell}=\Big\{\sum_{xy\in \partial^+_E R(n,\ell):\: x_d=\ell}(\mathsf{a}^0_y)^2 \ge \sum_{xy\in \partial^+_E R(n,\ell):\:x_d=\ell} \Big(\Psi^{0,h}_{\mathbb{H},\beta}[(\mathsf{a}_y)^2]-\frac{\varepsilon}{C}\Big)\Big\},
\end{equation}
\begin{equation}
\mathcal H_{n,\ell}=\Big\{\sum_{xy\in \partial^+_E R(n,\ell):\: x_d<\ell} (\mathsf{a}_y^1)^2\leq C\ell n^{d-2}\Big\}.
\end{equation}
Applying Birkhoff's ergodic theorem as in the proof of Claim~\ref{claim:mathcalCD} and using (again) regularity, we see that (for each fixed $\ell\geq 1$) for every $C$ large enough, if $n$ is large enough
\begin{equation}
\mathbb{P}_{\beta}[\mathcal F_{n,\ell}\cap \mathcal G_{n,\ell}\cap \mathcal H_{n,\ell}]\geq 1-\frac{\delta^2}{8}.
\end{equation}
We now show that, for a good choice of $\ell$, for every $n$ and $C$ large enough, we have
\begin{equation}\label{eq:inclusion}
\mathcal F_{n,\ell}\cap \mathcal G_{n,\ell}\cap \mathcal H_{n,\ell}\subset \mathcal J_{n,\ell},    
\end{equation}
from which the desired result follows.
Consider $k$ large enough so that
\begin{equation}
    \Psi^1_{\Lambda_k,\beta}[(\mathsf{a}_0)^2] \leq \Psi^1_\beta[(\mathsf{a}_0)^2] + \frac{\varepsilon}{C}, \quad \Psi^0_{\Lambda_k,\beta}[(\mathsf a_0)^2]\geq \Psi_\beta^0[(\mathsf{a}_0)^2]-\frac{\varepsilon}{C}.
\end{equation}
For $\ell$ large enough, if $x \in \mathbb Z^{d-1}\times \{\ell\}$, then by monotonicity (see Propositions \ref{prop:monotonicity} and \ref{prop: monot plus}),
\begin{equation}
\Psi^{1,h}_{\mathbb H,\beta}[(\mathsf a_x)^2] \leq \Psi^1_{\Lambda_k(x),\beta}[(\mathsf a_x)^2], \quad \Psi^{0,h}_{\mathbb H,\beta}[(\mathsf a_x)^2] \geq \Psi^0_{\Lambda_k(x),\beta}[(\mathsf a_x)^2],
\end{equation}
where $\Lambda_k(x) = x + \Lambda_k$. As a result, if $x \in \mathbb Z^{d-1}\times \{\ell\}$, and $y\sim x$,
\begin{equation}
   \Psi^{1,h}_{\mathbb H,\beta}[(\mathsf a_y)^2] - \Psi^{0,h}_{\mathbb H, \beta}[(\mathsf a_y)^2] \leq \Psi^{1}_{\beta}[(\mathsf a_y)^2] - \Psi^{0}_{\beta}[(\mathsf a_y)^2]  + \frac{2\varepsilon}{C}.
\end{equation}
Observe that, since $\mathsf{a}^1_y\geq \mathsf{a}^0_y\geq 0$,
\begin{equation}\label{eq: sum a1a0diff}
\sum_{xy\in \partial^+_E R(n,\ell)} (\mathsf a^1_y - \mathsf a^0_y)^2 \leq \sum_{xy\in \partial^+_E R(n,\ell)} (\mathsf a^1_y - \mathsf a^0_y)(\mathsf a^1_y+\mathsf a^0_y)=\sum_{xy\in \partial^+_E R(n,\ell)} (\mathsf a^1_y)^2 - (\mathsf a^0_y)^2. 
\end{equation}
Thus, on the event $\mathcal F_{n,\ell} \cap \mathcal G_{n,\ell}\cap \mathcal H_{n,\ell}$, the left-hand side of \eqref{eq: sum a1a0diff} is bounded from above by
\begin{multline}\label{eq: last equation 3.4}
\sum_{xy\in \partial^+_E R(n,\ell): \: x_d=\ell} \Big(\Psi^{1,h}_{\mathbb{H},\beta}[(\mathsf{a}_y)^2]+\frac{\varepsilon}{C}\Big)-\sum_{xy\in \partial^+_E R(n,\ell): \: x_d=\ell} \Big(\Psi^{0,h}_{\mathbb{H},\beta}[(\mathsf{a}_y)^2]-\frac{\varepsilon}{C}\Big)+C\ell n^{d-2}\\\leq \sum_{xy\in \partial^+_E R(n,\ell):\: x_d=\ell} \Big(\Psi^{1}_{\beta}[(\mathsf{a}_y)^2]-\Psi^{0}_{\beta}[(\mathsf{a}_y)^2] +\frac{4\varepsilon}{C}\Big)+C\ell n^{d-2},
\end{multline}
where the last term on the right-hand side is the contribution coming from $x_d<\ell$. 
Since $\Psi^1_{\beta}=\Psi^0_{\beta}$ by our assumption, and since $\ell$ is fixed, for every $C$ sufficiently large and $n$ sufficiently large in terms of $C$, the right-hand side of \eqref{eq: last equation 3.4} is at most $\varepsilon n^{d-1}$. Thus
we obtain \eqref{eq:inclusion}, which concludes the proof.
\end{proof}

\subsection{Bulk induced partitions are approximately equal}

We conclude this section, and the proof of Theorem \ref{thm:half-space}, by proving Lemma \ref{lem:partitions approx}.

\begin{proof}[Proof of Lemma \textup{\ref{lem:partitions approx}}] Fix $\beta>0$ such that $\Psi^1_\beta=\Psi^0_\beta$. 
We will prove that, for every $\varepsilon,\delta > 0$, there exists $\ell$ large enough such that for every $n$ large enough, for every $\omega\in \{0,1\}^{\overline E^+(R(n,\ell))}$ and every $\mathsf{a}\in (\mathbb{R}^+)^{R(n,\ell)}$, 
$\overline{k}^{(\xi^0_{n,\ell},\mathsf{b}^0_{n,\ell})}(\omega,\mathsf{a}) \leq \overline{k}^{(\xi^1_{n,\ell},\mathsf{b}^1_{n,\ell})}(\omega,\mathsf{a}) + \varepsilon n^{d-1}$ with probability at least $1-\delta$. We begin with two preliminary observations.

\vspace{5pt}

Let $m>0$. For $x\in \mathbb{Z}^d$, write $S_m(x):=x+\{-m,\ldots,m\}^{d-1}\times\{0\}$, $R_m(x):=x+R(m,m)$ and $\partial^{\textup{up}}R_m(x):=\partial R_m(x)\setminus S_{m-1}(x)$. See Figure \ref{fig:BK} for an illustration. Define $\mathcal{A}_m(x)=\{\omega^0_e=\omega^1_e \, \forall \, e\in E(R_m(x)), \mathbbm{1}_{\mathsf{a}^0_y=0}=\mathbbm{1}_{\mathsf{a}^1_y=0} \, \forall \, y\in R_m(x)\}$. Observe that for every $m>0$, 
\begin{equation}\label{eq: A bound}
\lim_{\ell\to\infty}\mathbb{P}_{\beta}[\mathcal{A}_m(\ell \cdot \mathbf{e}_d)]=1,
\end{equation}
where $\mathbf{e}_d=(0,\ldots,0,1)\in \mathbb{Z}^d$.
Indeed, by the defining property of the monotone coupling and a union bound,
\begin{equation}
\mathbb{P}_{\beta}[\exists \, e\in E(R_m(\ell\cdot \mathbf{e}_d)): \omega^0_e\neq \omega^1_e]\le \sum_{e\in E(R_m(\ell\cdot \mathbf{e}_d))}\left(\Psi^{1,h}_{\mathbb{H},\beta}[\omega_e=1]-\Psi^{0,h}_{\mathbb{H},\beta}[\omega_e=1]\right).
\end{equation}
To handle the right-hand side, note that by monotonicity properties of the free and wired measures (see Propositions \ref{prop: random cluster properties} and \ref{prop: wired random cluster properties}), and our assumption that $\Psi^1_{\beta}=\Psi^0_{\beta}$, for every $e=xy\in E(R_m(\ell\cdot \mathbf{e}_d))$
\begin{equation}
\Psi^{1,h}_{\mathbb{H},\beta}[\omega_e=1]-\Psi^{0,h}_{\mathbb{H},\beta}[\omega_e=1]\le \Psi^1_{\Lambda_{\ell/2}(x),\beta}[\omega_e=1]-\Psi^0_{\Lambda_{\ell/2}(x),\beta}[\omega_e=1]\underset{\ell\rightarrow \infty}{\longrightarrow} 0.
\end{equation}
Similarly, the event $\{\mathbbm{1}_{\mathsf{a}^0_y=0}=\mathbbm{1}_{\mathsf{a}^1_y=0} \, \forall \, y\in R_m(\ell \cdot \mathbf{e}_d)\}$ happens with probability tending to $1$ as $\ell$ tends to infinity, which proves \eqref{eq: A bound}.

    \begin{figure}[htb]
        \centering
        \includegraphics[width=0.6\linewidth]{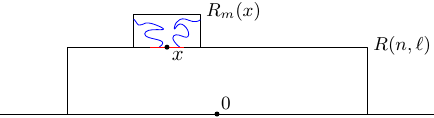}
        \caption{An illustration of the relevant objects appearing in the proof of Lemma \ref{lem:partitions approx}. The bold red line corresponds to $S_k(x)$. The blue lines are paths in disjoint clusters connecting $S_k(x)$ to $\partial^{\textup{up}}R_m(x)$. We prove that with high probability there are at most two such clusters in $R_m(x)$.}
        \label{fig:BK}
    \end{figure}

We now fix $k,m,n \in \mathbb N^*$ such that $k\leq m \leq n$ (think $k \ll m \ll n$). Let us define $\mathcal{N}_{k,m}(x)$ to be the number of clusters of $\omega\cap R_m(x)$ that intersect both $S_k(x)$ and $\partial^{\textup{up}}R_m(x)$, and write $\mathcal{B}_{k,m}(x)=\{\mathcal{N}_{k,m}(x)\le 2\}$. 
\begin{Claim} \label{claim: num-clusters-leq2}
For every $c>0$ and $k \in \mathbb N^*$, there exist $m>k$ and infinitely many $\ell>0$ such that
\begin{equation} \label{eq: num-clusters-prob}
\Psi^{1,h}_{\mathbb{H},\beta}[\mathcal{B}_{k,m}(\ell \cdot \mathbf{e}_d)]\ge 1-c.  
\end{equation}
\end{Claim}

\vspace{5pt}

We assume the claim for now and conclude the proof. Let $c>0$ to be chosen small enough and note that by \eqref{eq: num-clusters-prob} and \eqref{eq: A bound}, for every $k \in \mathbb N^*$, for $m$ as in \eqref{eq: num-clusters-prob}, for infinitely many $\ell$ we have 
\begin{equation}\label{eq:AB bound}
\mathbb P_\beta[\mathcal{A}_{m}((\ell+1)\cdot \mathbf{e}_d)\cap \mathcal{B}_{k,m}((\ell+1)\cdot \mathbf{e}_d)]\ge 1-c,
\end{equation}
where we abuse the notation to write $\mathcal{B}_{k,m}((\ell+1)\cdot\mathbf{e}_d)$ for the corresponding event for $\omega^1$ in the coupling (note, however, that on the event $\mathcal{A}_m((\ell+1)\cdot \mathbf{e}_d)$, $\omega^1$ and $\omega^0$ agree on the relevant domain).
Let $n':=(n-k)/(2k+1)$, where we assume it is an integer without loss of generality, and let us now partition $S_n((\ell+1)\cdot \mathbf{e}_d)$ into boxes of the form $S_k(x)$ for $x\in (2k+1)S_{n'}((\ell+1)\cdot \mathbf{e}_d)$. 

Let $\mathcal{E}_{k,m,n,\ell}$ denote the event that there are at most $\varepsilon n^{d-1}/(4k)^{d-1}$ vertices $x\in (2k+1)S_{n'}((\ell+1)\cdot \mathbf{e}_d)$ such that the event $(\mathcal{A}_{m}(x)\cap \mathcal{B}_{k,m}(x))^c$ happens. Using \eqref{eq:AB bound}, translation invariance in the half-space, and Markov's inequality, we deduce that if $c>0$ is small enough (as a function of $\varepsilon$ and $\delta$ that have been fixed at the beginning of the proof, and also $k$) in \eqref{eq:AB bound}, then for $m$ sufficiently large, and infinitely many $\ell$, we have
\begin{equation}\label{eq:bound E k m n ell}
\mathbb{P}_{\beta}[\mathcal{E}_{k,m,n,\ell}]\geq 1-\delta,
\end{equation}
for $n$ large enough (depending on $k$ and $\ell$).

We now show that for every $\varepsilon>0$ there exists $k$ such that, for every $m>k$ and $\ell$, one has
\begin{equation}
    \mathcal{E}_{k,m,n,\ell}\subset \mathcal P(\varepsilon, n,\ell)
\end{equation}
for $n$ large enough (depending on $m$ and $\ell$). 
Upon choosing $m>k$ sufficiently large and then $\ell$ sufficiently large (in terms of $\delta,k,m$), 
the desired result follows from~\eqref{eq:bound E k m n ell}.

First, observe that any cluster of $\omega \in \{0,1\}^{\overline{E}^+(R(n,\ell))}$ which does not intersect $\partial^{\textup{ext},+} R(n,\ell)$ contributes $0$ to $
\overline{k}^{(\xi^0_{n,\ell},\mathsf{b}^0_{n,\ell})}(\omega,\mathsf{a}) - \overline{k}^{(\xi^1_{n,\ell},\mathsf{b}^1_{n,\ell})}(\omega,\mathsf{a})$. The same holds for any $\omega$ cluster that intersects $\partial^{\textup{ext},+} R(n,\ell)$ \emph{only} at vertices whose $\xi^0_{n,\ell}$ and $\xi^1_{n,\ell}$ partition classes coincide. 
All remaining $\omega$ clusters need to intersect $\partial^{\textup{ext},+} R(n,\ell)$ in some partition class of $\xi^0_{n,\ell}$ that does not coincide with a $\xi^1_{n,\ell}$ partition class. 
Thus, we can upper bound $
\overline{k}^{(\xi^0_{n,\ell},\mathsf{b}^0_{n,\ell})}(\omega,\mathsf{a}) - \overline{k}^{(\xi^1_{n,\ell},\mathsf{b}^1_{n,\ell})}(\omega,\mathsf{a})$ by the number of partition classes of $\xi^0_{n,\ell}$ that do not coincide with a $\xi^1_{n,\ell}$ partition class. 
To estimate the latter, note that for every $x\in S_{n-m}((\ell+1)\cdot \mathbf e_d)$, on the event $\mathcal{A}_m(x)$, the clusters of $\omega^0\cap R_m(x)$ and $\omega^1\cap R_m(x)$ that intersect $S_k(x)$ (in fact even $S_m(x)$) but do not intersect $\partial^{\textup{up}}R_m(x)$ coincide. 
This implies that the corresponding partition classes of $\xi^0_{n,\ell}$ and $\xi^1_{n,\ell}$ also coincide, and there are at most $2$ partition classes intersecting $S_k(x)$ that do not coincide. Thus, we can conclude that, for some constant $C>0$,
\begin{equation}
\overline{k}^{(\xi^0_{n,\ell},\mathsf{b}^0_{n,\ell})}(\omega,\mathsf{a}) - \overline{k}^{(\xi^1_{n,\ell},\mathsf{b}^1_{n,\ell})}(\omega,\mathsf{a})\le (2k+1)^{d-1}\cdot\varepsilon  \frac{n^{d-1}}{(4k)^{d-1}}+2\cdot\frac{(2n+1)^{d-1}}{(2k+1)^{d-1}}+C\ell n^{d-2}
\end{equation}
where the last term corresponds to contributions to $
\overline{k}^{(\xi^0_{n,\ell},\mathsf{b}^0_{n,\ell})}(\omega,\mathsf{a}) - \overline{k}^{(\xi^1_{n,\ell},\mathsf{b}^1_{n,\ell})}(\omega,\mathsf{a})$ coming from $\partial^{\textup{ext},+}R(n,\ell)\setminus S_{n-m}((\ell+1)\mathbf e_d)$. For every $\varepsilon>0$ there exists $k$ such that the term on the right-hand side above is at most $\varepsilon n^{d-1}$ for $n$ large enough (depending on $k$ and $\ell$).
\end{proof}

\begin{proof}[Proof of Claim \textup{\ref{claim: num-clusters-leq2}}]
We apply an argument similar to that of Burton and Keane \cite{BurtonKeane1989density} for the uniqueness of the infinite cluster. 

Let $m\geq 1$. Let us write $\mathcal{C}_m(x)$ to denote the cluster of $x$ in $\omega\cap R_m(x)$, where $\omega\in \{0,1\}^{E(\mathbb H)}$. We say that $x$ is an $R_m(x)$-trifurcation if $\mathcal{C}_m(x)$ intersects $\partial^{\textup{up}} R_m(x)$ and $\mathcal{C}_m(x)\setminus \{x\}$ contains at least 3 connected components that intersect $\partial^{\textup{up}} R_m(x)$. We denote the event that $x$ is an $R_m(x)$-trifurcation by $\mathcal{T}_m(x)$. Let us write $\mathcal{K}_{m,\ell}$ for the (random) number of vertices $x\in R_m(\ell\cdot \mathbf{e}_d)$ which are $R_{2m}(x)$-trifurcations. Note that by definition, each of these $R_{2m}(x)$-trifurcations is connected in $R_m(\ell\cdot \mathbf{e}_d)$ to $\partial^{\textup{up}} R_m(\ell\cdot \mathbf{e}_d)$. See Figure \ref{fig:forest}.

    \begin{figure}[htb] 
        \centering
        \includegraphics[width=0.88\linewidth]{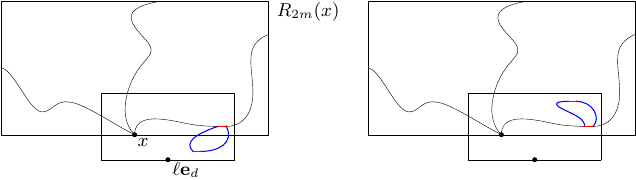}
        \caption{An illustration of the procedure involved in the construction of the forest $\mathcal{F}$. The cycle $\mathsf{C}$ is represented in blue. On the left, if we decide to erase the top red edge at step $j$, we destroy the fact that $x$ is an $R_{2m}(x)$-trifurcation. However, if we erase the one on the bottom, the property is preserved. On the right, since the cycle is entirely included in $R_{m}(\ell\cdot\mathbf{e}_d)\cap R_{2m}(x)$, we can erase either of the two red edges without changing the fact that $x$ is an $R_{2m}(x)$-trifurcation.}
        \label{fig:forest}
    \end{figure}

By translation invariance in the half-space,
\begin{equation}\label{eq: K_m}
\Psi^{1,h}_{\mathbb{H},\beta}[\mathcal{K}_{m,\ell}]=(2m+1)^{d-1}\sum_{i=0}^m \Psi^{1,h}_{\mathbb{H},\beta}[\mathcal{T}_{2m}((\ell+i)\cdot\mathbf{e}_d)].
\end{equation}
We claim that $\mathcal{K}_{m,\ell}\le |\partial^{\textup{up}} R_m(\ell\cdot \mathbf{e}_d)|$. Indeed, let us construct a spanning forest of $\omega\cap R_m(\ell\cdot \mathbf{e}_d)$ as follows. We pick a cycle $\mathsf C$ of $\omega\cap R_m(\ell\cdot \mathbf{e}_d)$ in an arbitrary way and among all edges $xy$ of $\mathsf C$ we remove one such that both endpoints $x$ and $y$ have the same $d$-th coordinate and this coordinate is the smallest possible among vertices in $\mathsf C$ (such an edge always exists). We repeat this procedure until no cycles are left in $\omega\cap R_m(\ell\cdot \mathbf{e}_d)$, at which step we obtain a spanning forest $\mathcal{F}$. We claim that all of the $\mathcal{K}_{m,\ell}$ vertices $x\in R_m(\ell\cdot \mathbf{e}_d)$ which are $R_{2m}(x)$-trifurcations in $\omega$ are still $R_{2m}(x)$-trifurcations in $\mathcal{F}$. Indeed, let us argue by induction. Let us consider an edge $e$ we removed from a cycle $\mathsf C$ at some step $j$, and consider a vertex $x\in R_m(\ell\cdot \mathbf{e}_d)$ which is an $R_{2m}(x)$-trifurcation at step $j-1$. If $x$ has strictly larger $d$-th coordinate than $e$, then removing $e$ leaves the connected component of $x$ in $R_{2m}(x)$ at step $j$ intact, so that $x$ is an $R_{2m}(x)$-trifurcation at step $j$. If not, then $x$ is again an $R_{2m}(x)$-trifurcation at step $j$, since the endpoints of $e$ are still connected in $\mathsf C\setminus\{e\}$, which lies in $R_{2m}(x)\cap R_m(\ell\cdot \mathbf{e}_d)$ by the way $e$ was chosen. The claim follows from induction on $j$. See Figure \ref{fig:forest} for an illustration.

Now we successively remove all leaves of $\mathcal{F}$ which do not lie in $\partial^{\textup{up}} R_m(\ell \cdot \mathbf{e}_d)$ until we are left with a forest $\mathcal{F}'$. It is not hard to see that all of the vertices $x\in R_m(\ell \cdot \mathbf{e}_d)$ which are $R_{2m}(x)$-trifurcations in $\omega$ are vertices of degree at least $3$ in $\mathcal{F}'$, and all leaves of $\mathcal{F}'$ lie in $\partial^{\textup{up}} R_m(\ell\cdot \mathbf{e}_d)$. In a forest, there are fewer vertices of degree at least $3$ than leaves, hence it follows that
\begin{equation} \label{eq: bound on K_m}
\mathcal{K}_{m,\ell}\le |\partial^{\textup{up}} R_m(\ell \cdot \mathbf{e}_d)|, 
\end{equation}
as desired.

Combining \eqref{eq: bound on K_m} with \eqref{eq: K_m}, we deduce that for every $\ell\ge 0$,
\begin{equation}\sum_{i=0}^m \Psi^{1,h}_{\mathbb{H},\beta}[\mathcal{T}_{2m}((\ell+i)\cdot\mathbf{e}_d)]\le C_1,
\end{equation}
for some constant $C_1>0$ depending only on $d$. It follows that there exists $i\in \{0,1,\ldots,m\}$ such that 
\begin{equation}
\Psi^{1,h}_{\mathbb{H},\beta}[\mathcal{T}_{2m}((\ell+i)\cdot \mathbf{e}_d)]\le \frac{C_1}{m}.
\end{equation}

Fix $c>0$ and $k\in \mathbb N^*$, and let $m\geq k$ to be chosen large enough. We now relate $\mathcal{T}_{2m}((\ell+i)\cdot \mathbf{e}_d)$ to the event $\mathcal{B}_{k,2m}((\ell+i)\cdot \mathbf{e}_d)$ that we are interested in. Note that on the event $\mathcal{B}_{k,2m}((\ell+i)\cdot \mathbf{e}_d)^c$ we can do a standard surgery (on a slight thickening of $S_k((\ell+i)\cdot\mathbf{e}_d)$) to create an $R_{2m}((\ell+i)\cdot \mathbf{e}_d)$-trifurcation. By the finite energy property of Proposition~\ref{prop:finite-energy}, there exists $C_2=C_2(k,\beta,C_1,d)>0$ such that $\Psi^{1,h}_{\mathbb{H},\beta}[\mathcal{B}_{k,2m}((\ell+i)\cdot \mathbf{e}_d)]\ge 1-C_2/m$. Choosing $m\geq C_2/c$ we get that $\Psi^{1,h}_{\mathbb{H},\beta}[\mathcal{B}_{k,2m}((\ell+i)\cdot \mathbf{e}_d)]\ge 1-c$, as desired.
\end{proof}

\begin{Rem}\label{rem:FK_adaptation}
The proofs of Theorems~\ref{thm:half-space} and \ref{thm:full-space uniqueness} rely crucially on certain monotonicity properties that apply more generally to the standard random-cluster model with cluster weight $q\ge 1$. In that setting, the proofs simplify considerably, since there is no need to deal with the absolute-value field: one can avoid using the connection with a spin model in the Radon--Nikodym derivative calculation and work directly with the random cluster measure itself as in the proof of \cite[Theorem 1.12]{DuminilLecturesOnIsingandPottsModels2019}. 
\end{Rem}

\section{Strict positivity of the surface tension}\label{sec:surface_tension}

The goal of this section is to prove that the \emph{surface tension} of the spin
model is strictly positive when $\beta > \beta_c$ and use this result to prove that a certain connection event happens with very high probability under the random cluster measure with wired boundary conditions. We begin by defining the 
surface tension and relating it to a relevant percolation event. Strict 
positivity is then established via a differentiation argument similar to 
that of \cite{lebowitz1981surface,gunaratnam2025supercritical}, 
which requires constructing suitable \emph{strip measures} in a manner 
slightly different from \cite{gunaratnam2025supercritical}. A key ingredient 
is Lebowitz's inequality, originally derived for Ising-type systems in 
\cite{Lebowitz1977CoexistencePhasesIsing,lebowitz1978number}, which we adapt to our setting at the end of this section.

Let us now state the main result of this section. To this end, we first introduce some notation. Let $L,M\geq 1$. Define the vertical strip $\mathcal R(L,M):= \{-L,\ldots,L\}^{d-1}\times \{-M,\dots,M\}$ and recall the definition of $\partial^{\textup{ext}} \mathcal{R}(L,M)$. Introduce the 
upper and lower boundaries
\begin{equation}
\partial^+ \mathcal R(L,M) := \{x\in \partial^{\textup{ext}} \mathcal{R}(L,M): x\in\{-L-1,\ldots,L+1\}^{d-1} \times \{ 0, \ldots ,M+1\}\},
\end{equation}
\begin{equation}
\partial^- \mathcal R(L,M) := \{x\in \partial^{\textup{ext}} \mathcal{R}(L,M): x\in\{-L-1,\ldots,L+1\}^{d-1} \times \{ -M-1, \ldots ,-1 \}\}.
\end{equation}
See Figure \ref{fig:strip} for an illustration. 

\begin{figure}[htb]
 \centering
    \includegraphics[width=0.35\linewidth]{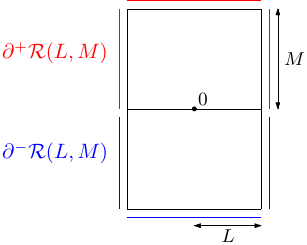}
    \caption{An illustration of $\mathcal{R}(L,M)$ and its boundaries $\partial^{+}\mathcal{R}(L,M)$ (in red) and $\partial^{-}\mathcal{R}(L,M)$ (in blue).}
    \label{fig:strip}
\end{figure}

We define the infinite vertical strip $\mathcal R(L):= \{-L,\dots,L\}^{d-1} \times \mathbb Z$, with upper and lower boundaries 
\begin{equation}
\partial^+ \mathcal R(L) := \{x\in \partial^{\textup{ext}} \mathcal{R}(L): x\in\{-L-1,\ldots,L+1\}^{d-1} \times \{ 0, 1, \ldots \}\},
\end{equation}
\begin{equation}
\partial^- \mathcal R(L) := \{x\in \partial^{\textup{ext}} \mathcal{R}(L): x\in\{-L-1,\ldots,L+1\}^{d-1} \times \{\ldots ,-1 \}\}.
\end{equation}

Below we write $\Psi^{1}_{\mathcal R(L),\beta}$ for the wired measure on $\mathcal R(L)$ that arises as the weak limit of $\Psi^{1}_{\mathcal R(L,M),\beta}$ as $M\to\infty$. This limit is well-defined as justified in Lemma \ref{lem:++} below.

\begin{Thm}\label{thm: wired disconnection}
Let $\beta>\beta_c$. There exists $c_1>0$ such that for every $L \geq 1$,
\begin{equation}\label{eq: proba strip}
\Psi^{1}_{\mathcal R(L),\beta}[ \partial^+\mathcal{R}(L)\longleftrightarrow\partial^-\mathcal{R}(L)] \geq 1 - e^{-c_1L^{d-1}}.
\end{equation}
\end{Thm}

\subsection{Definition of the surface tension}

Our definition of the surface tension is analogous to the one in \cite{gunaratnam2025supercritical} except that we need to modify it to accommodate the new definition of the plus measure.

Denote by $\nu^{+,-}_{\mathcal{R}(L,M),\beta}$ the probability measure on $\mathbb R^{\overline{\mathcal{R}(L,M)}}$ with expectation $\langle \cdot \rangle^{+,-}_{\mathcal{R}(L,M),\beta}$ and defined by the Radon--Nikodym density
\begin{multline}
\mathrm{d}\nu^{+,-}_{\mathcal{R}(L,M),\beta}(\tau) \\= \frac{1}{\mathbf{Z}^{+,-}_{\mathcal{R}(L,M),\beta}} \exp\Big(\beta \sum_{xy\in \overline E(\mathcal{R}(L,M))} \tau_x \tau_y
\Big)\prod_{x\in \mathcal{R}(L,M)}\mathrm{d}\rho(\tau_x) \prod_{x \in \partial^+ \mathcal{R}(L,M)} \mathrm{d}\zeta (\tau_x) \prod_{x \in \partial^- \mathcal{R}(L,M)} \mathrm{d}\zeta (-\tau_x),
\end{multline}
where $\mathbf{Z}^{+,-}_{\mathcal{R}(L,M),\beta}$ is a normalisation constant and $\zeta$ is the same measure as in Proposition \ref{prop:reg_plus}. We also write $\mathbf{Z}^{+,+}_{\mathcal{R}(L,M),\beta}=\mathbf{Z}^{+}_{\mathcal{R}(L,M),\beta}$ and $\nu^{+,+}_{\mathcal{R}(L,M),\beta}=\nu^{+}_{\mathcal{R}(L,M),\beta}$.

\begin{Def}[Surface tension] Let $L,M\geq 1$ and $\beta>0$. The surface tension of the model on $\mathcal{R}(L,M)$ at inverse temperature $\beta$ is defined as
\begin{equation}
    \boldsymbol{\tau}^{L,M}(\beta):=-\frac{1}{L^{d-1}}\log  \frac{\mathbf{Z}^{+,-}_{\mathcal{R}(L,M),\beta}}{\mathbf{Z}^{+,+}_{\mathcal{R}(L,M),\beta}}.
\end{equation}
Moreover, we also define
\begin{equation}
    \boldsymbol{\tau}^L(\beta):=\liminf_{M\rightarrow \infty}\boldsymbol{\tau}^{L,M}(\beta), \qquad \boldsymbol{\tau}(\beta):=\liminf_{L\rightarrow \infty}\boldsymbol{\tau}^{L}(\beta).
\end{equation}
\end{Def}

The surface tension can be re-expressed in terms of a disconnection event under the measure $\Psi^1_{\mathcal{R}(L,M),\beta}$. We emphasise that below the connection takes place before identifying the vertices in $\partial^{\rm ext} \mathcal{R}(L,M)$ to a single vertex.

\begin{Lem}\label{lem: surface tension expression} 
Let $L,M\geq 1$ and $\beta>0$. One has
\begin{equation}
    \frac{\mathbf{Z}^{+,-}_{\mathcal{R}(L,M),\beta}}{\mathbf{Z}^{+,+}_{\mathcal{R}(L,M),\beta}}=\Psi^1_{\mathcal{R}(L,M),\beta}[\partial^+\mathcal{R}(L,M)\centernot\longleftrightarrow \partial^-\mathcal{R}(L,M)].
\end{equation}
In particular, one has $\boldsymbol{\tau}^{L,M}(\beta)\geq 0$. 
\end{Lem}

\begin{proof} 
We follow the same strategy as in the proof of \cite[Lemma~5.3]{gunaratnam2025supercritical}. Recall the definition of $Z^1_{\mathcal{R}(L,M),\beta}$ from \eqref{eq: random cluster explicit density wired}. We claim that  
\begin{equation}\label{eq: pp partition equality}
\mathbf{Z}^{+,+}_{\mathcal{R}(L,M),\beta}=\frac{1}{2}Z^1_{\mathcal R(L,M),\beta},   
\end{equation}
and
\begin{equation}\label{eq: pm parition equality}
\mathbf{Z}^{+,-}_{\mathcal{R}(L,M),\beta}=\frac{1}{2}Z^1_{\mathcal R(L,M),\beta}[\partial^+\mathcal{R}(L,M)\centernot\longleftrightarrow \partial^-\mathcal{R}(L,M)],   
\end{equation}
from which the desired result follows. Note that here, for convenience, the partition function with an event in the square brackets denotes the unnormalized expectation (i.e.\ partition function times the expectation).

Let us prove the claim. We will only prove \eqref{eq: pm parition equality}, as \eqref{eq: pp partition equality} follows similarly and is simpler to deduce. First, observe that 
\begin{equation}
\mathbf{Z}^{+,-}_{\mathcal{R}(L,M),\beta}=\int Z^{\rm Ising,+,-}_{\mathcal{R}(L,M),\beta,\mathsf{a}}
\prod_{x\in \mathcal{R}(L,M)}\mathrm{d} \rho(\mathsf{a}_x) \prod_{x\in \partial^{\textup{ext}}\mathcal{R}(L,M)}\mathrm{d}\zeta(\mathsf{a}_x),  
\end{equation}
where
\begin{equation}\label{eq:Ising_def}
Z^{\rm Ising,+,-}_{\mathcal{R}(L,M),\beta,\mathsf{a}}=\sum_{\sigma\in \{-1,0,1\}^{\overline{\mathcal{R}(L,M)}}} \mathbbm{1}_{\sigma\sim (+,-,\mathsf{a})}
 \exp\Bigg( \beta\sum_{xy\in \overline{E}(\mathcal{R}(L,M))}  \mathsf a_x \mathsf{a}_y \sigma_x \sigma_y\Bigg),
\end{equation}
and $\sigma\sim (+,-,\mathsf{a})$ means that $\sigma_x=1$ for every $x\in \partial^+\mathcal{R}(L,M)$, $\sigma_x=-1$ for every $x\in \partial^-\mathcal{R}(L,M)$, and $\sigma_x=0$ if and only if $\mathsf{a}_x=0$. The equality \eqref{eq: pm parition equality} then follows if we can establish
\begin{equation}\label{eq: partition function equality}
Z^{\textup{Ising},+,-}_{\mathcal{R}(L,M),\beta, \mathsf{a}}=\frac{1}{2}{Z}^w_{\mathcal{R}(L,M),\beta, \mathsf{a}}[\partial^+\mathcal{R}(L,M)\centernot\longleftrightarrow\partial^-\mathcal{R}(L,M)]\prod_{xy \in \overline{E}(\mathcal{R}(L,M))} \sqrt{1-p(\beta,\mathsf a)_{xy}},
\end{equation}
as seen from \eqref{eq: random cluster explicit density wired}, where we recall that $Z^{w}_{\mathcal{R}(L,M),\beta,\mathsf{a}}$ is the partition function of $\phi^w_{\mathcal{R}(L,M),\beta,\mathsf{a}}$; see \eqref{def: random weight rc}.

We say that a spin configuration $\sigma\in \{-1,0,1\}^{\overline{\mathcal{R}(L,M)}}$ is compatible with $\omega\in \{0,1\}^{\overline{E}(\mathcal{R}(L,M))}$ and $\mathsf{a}$ if $\sigma$ is constant in the connected components of the graph with vertex set $\{x\in \overline{\mathcal{R}(L,M)}: \mathsf{a}_x\neq 0\}$ and edge set $\{xy: \omega_{xy}=1\}$, and equal to $0$ on $\{x\in \overline{\mathcal{R}(L,M)}:\: \mathsf{a}_x=0\}$.  We write $(\omega,\mathsf{a})\sim_{\fg}\sigma$ if, in addition, $\sigma_x=1$ for every $x\in \partial^+\mathcal{R}(L,M)$ and $\sigma_x=-1$ for every $x\in \partial^-\mathcal{R}(L,M)$. Note that, given $\mathsf{a}$ and a configuration $\omega$ realising the event $\{\partial^+\mathcal{R}(L,M)\centernot\longleftrightarrow \partial^-\mathcal{R}(L,M)\}$, the number of configurations $\sigma$ such that $(\omega,\mathsf{a})\sim_{\fg}\sigma$ is $2^{ k^w(\omega,\mathsf{a})-1}$, where we recall that $k^w(\omega,\mathsf{a})$ is the number of connected components after identifying all vertices in $\partial^{\textup{ext}}\mathcal{R}(L,M)$. Conversely, if $(\omega,\mathsf{a})\sim_{\mathfrak g} \sigma$, then $\omega \in \{\partial^+\mathcal{R}(L,M)\centernot\longleftrightarrow \partial^-\mathcal{R}(L,M)\}$. Therefore, we can write
\begin{align}
\begin{aligned}
Z^w_{\mathcal{R}(L,M),\beta,\mathsf{a}}[\partial^+\mathcal{R}&(L,M)\centernot\longleftrightarrow \partial^-\mathcal{R}(L,M)]
\\&=2 \sum_{\substack{\omega\in \{0,1\}^{\overline{E}(\mathcal{R}(L,M))}\\ \partial^+\mathcal{R}(L,M)\centernot\longleftrightarrow \partial^-\mathcal{R}(L,M)}}2^{k^w(\omega,\mathsf{a})-1}\prod_{xy \in \overline{E}(\mathcal{R}(L,M))} \Big(\frac{p(\beta,\mathsf{a})_{xy}}{1-p(\beta,\mathsf{a})_{xy}}\Big)^{\omega_{xy}} 
\\&= 2 \sum_{\substack{\sigma\in \{-1,0,1\}^{\overline{\mathcal{R}(L,M)}}\\\sigma\sim (+,-,\mathsf{a})}} \sum_{(\omega,\mathsf{a})\sim_\fg \sigma}\prod_{xy \in \overline{E}(\mathcal{R}(L,M))} \Big(\frac{p(\beta,\mathsf{a})_{xy}}{1-p(\beta,\mathsf{a})_{xy}}\Big)^{\omega_{xy}}. 
\end{aligned}
\end{align}
Furthermore, if $(\omega,\mathsf{a})\sim_\fg \sigma$, then $\omega_{xy}$ needs to be $0$ whenever $\sigma_x\neq \sigma_y$, and both values $0$ and $1$ are allowed whenever $\sigma_x=\sigma_y\neq 0$. This gives that $Z^w_{\mathcal{R}(L,M),\beta,\mathsf{a}}[\partial^+\mathcal{R}(L,M)\centernot\longleftrightarrow \partial^-\mathcal{R}(L,M)]$ is equal to
\begin{align}
\begin{aligned}
2\sum_{\sigma\in \{-1,0,1\}^{\overline{\mathcal R(L,M)}}}\mathbbm{1}_{\sigma\sim (+,-,\mathsf{a})}&\prod_{\substack{xy \in \overline{E}(\mathcal{R}(L,M))\\ \sigma_x=\sigma_y}}\sum_{\omega_{xy}\in \{0,1\}} \Big(\frac{p(\beta,\mathsf{a})_{xy}}{1-p(\beta,\mathsf{a})_{xy}}\Big)^{\omega_{xy}}\\
&=2\sum_{\sigma\in \{-1,0,1\}^{\overline{\mathcal{R}}(L,M)}}\mathbbm{1}_{\sigma\sim (+,-,\mathsf{a})}\prod_{xy \in \overline{E}(\mathcal{R}(L,M))}e^{2\beta \mathsf{a}_x \mathsf{a}_y \mathbbm{1}_{\sigma_x=\sigma_y}}
\\&=2 Z^{\textup{Ising},+,-}_{\mathcal{R}(L,M),\beta,\mathsf{a}}\prod_{xy\in \overline{E}(\mathcal{R}(L,M))}e^{\beta\mathsf{a}_x\mathsf{a}_y},
\end{aligned}
\end{align}
where we used that $\mathbbm{1}_{\sigma_x=\sigma_y}=(\sigma_x\sigma_y+1)/2$ whenever $\sigma_x\sigma_y\in \{\pm 1\}$ in the last equality. It follows that
\eqref{eq: partition function equality} holds. This completes the proof. 
\end{proof}

\subsection{Infinite volume strip measures}

We now construct infinite volume strip measures.

\begin{Lem}\label{lem:++}
Let $\beta> 0$ and $L\geq 1$. There exist measures $\langle \cdot \rangle^{+,+}_{\mathcal{R}(L),\beta}$ and $\langle \cdot \rangle^{+,-}_{\mathcal{R}(L),\beta}$ such that 
\begin{equation}
\lim_{M\to\infty}\langle \cdot \rangle^{+,+}_{\mathcal{R}(L,M),\beta}=\langle \cdot \rangle^{+,+}_{\mathcal{R}(L),\beta}  \quad \text{and} \quad \lim_{M\to\infty}\langle \cdot \rangle^{+,-}_{\mathcal{R}(L,M),\beta}=\langle \cdot \rangle^{+,-}_{\mathcal{R}(L),\beta},
\end{equation}
where the limits hold in the weak sense.
Furthermore, for every $x\in \{-L,\ldots,L\}^{d-1}$ we have 
\begin{equation}\label{eq:magnetisation convergence}
\lim_{N\to\infty} \langle \tau_{(x,N)} \rangle^{+,-}_{\mathcal{R}(L),\beta}=\langle \tau_{(x,0)} \rangle^{+,+}_{\mathcal{R}(L),\beta} \quad \text{and} \quad \lim_{N\to\infty} \langle \tau_{(x,-N)} \rangle^{+,-}_{\mathcal{R}(L),\beta}=-\langle \tau_{(x,0)} \rangle^{+,+}_{\mathcal{R}(L),\beta}.
\end{equation}
\end{Lem}

\begin{proof} 
We first show the convergence of $\langle \cdot \rangle^{+,+}_{\mathcal{R}(L,M),\beta}$. To this end, let us write $\langle \cdot \rangle^{+,0}_{\mathcal{R}(L,M),\beta}$ for the (free) spin measure in $\mathbb{R}^{\mathcal{R}(L+1,M)}$ with single-site measure $\rho$ in $\mathcal{R}(L,M)$ and single-site measure $\zeta$ in $\mathcal{R}(L+1,M)\setminus \mathcal{R}(L,M)$. In words, it is similar to $\langle \cdot \rangle^{+,+}_{\mathcal{R}(L,M),\beta}$ except that we have ``turned off'' the $\zeta$ measure on the part of $\partial^{\textup{ext}}R(L,M)$ of $d$-th coordinate equal to $M+1$ or $-(M+1)$.
By monotonicity\footnote{We used a monotonicity statement for free measures with inhomogeneous single-site measure. This is reminiscent of what was discussed in Remark \ref{rem:plus}, except that now the measures $\langle \cdot \rangle^{+,0}_{\mathcal{R}(L,M),\beta}$ are stochastically increasing in $M$ since we do not ``push away'' the $\zeta$ single-site measures.} in the coupling constants and regularity, $\langle \cdot \rangle^{+,0}_{\mathcal{R}(L,M),\beta}$ converges weakly as $M\to\infty$ to a measure $\langle \cdot \rangle^{+,0}_{\mathcal{R}(L),\beta}$. Our aim is to show that $\langle \cdot \rangle^{+,+}_{\mathcal{R}(L,M),\beta}$ converges weakly as $M\to\infty$ to $\langle \cdot \rangle^{+,0}_{\mathcal{R}(L),\beta}$, so then we may define $\langle \cdot \rangle^{+,+}_{\mathcal R(L),\beta}:=\langle \cdot \rangle^{+,0}_{\mathcal R(L),\beta}$.

To this end, note that $\langle \cdot \rangle^{+,+}_{\mathcal{R}(L,M),\beta} \succcurlyeq \langle \cdot \rangle^{+,0}_{\mathcal{R}(L,M),\beta}$ by the domain Markov property (in which we condition on the values of the spins of $d$-th coordinate equal to $M+1$ or $-(M+1)$), monotonicity in boundary conditions, and the fact that for $\tau \sim \langle \cdot \rangle^{+,+}_{\mathcal{R}(L,M),\beta}$ we almost surely have $\tau_x>0$ for every $x\in \mathcal{R}(L+1,M+1)\setminus \mathcal{R}(L+1,M)$. Hence, for any increasing local event $\mathcal{A}$ we have $\liminf_{M\to\infty}\nu^{+,+}_{\mathcal{R}(L,M),\beta}[\mathcal{A}]\geq \nu^{+,0}_{\mathcal{R}(L),\beta}[\mathcal{A}]$. 

We now show that the converse inequality holds with the liminf replaced by limsup. Let $N$ be such that $\mathcal{A}$ is measurable with respect to $\mathcal{R}(L+1,N)$. Define 
\begin{equation}
\iota := \sup \{i>0:~ \tau_{(x,2i)}\leq 0, \tau_{(x,-2i)}\leq 0, \; \forall x \in \{-L,\ldots, L \}^{d-1}\},
\end{equation}
with the convention that $\iota=-\infty$ if no such $i$ exists. We first prove that $\iota>0$ with high probability under $\nu^{+,+}_{\mathcal R(L,M),\beta}$ provided $M$ is large enough.
Let $B_K$ be the event that for at least half of the $i\in \{1,2,\ldots,M/3\}$ the event  
\begin{equation}
A_{K,i}:=\{|\tau_{(x,\ell)}|\leq K, |\tau_{(x,-\ell)}|\leq K ,\: \forall x\in \{-L-1,\ldots,L+1\}^{d-1}, \forall \ell \in \{2i-1,2i+1\} \}
\end{equation}
occurs. Then, for every $\varepsilon > 0$, by regularity and  Markov's inequality applied to the random variable $\mathcal{N}:=\sum_{i=1}^{M/3}\mathds{1}_{A_{K,i}^c}$, by choosing $K$ sufficiently large we obtain $\nu^{+,+}_{\mathcal{R}(L,M),\beta}[ B_K]\geq 1-\varepsilon/2$.
Let $\mathcal{R}^+_{\textup{odd}}(L):=\{(x,2j+1): x\in \{-L-1,\ldots,L+1\}^{d-1}, \: j\geq -M\}$. Using the domain Markov property, we obtain that the probability
\begin{equation}
\nu^{+,+}_{\mathcal{R}(L,M),\beta}\left[\tau_{(x,2i)}\leq 0, \tau_{(x,-2i)}\leq 0\, \forall x\in \{-L,\ldots,L\}^{d-1} \Big |\: \tau_{|\mathcal{R}^+_{\textup{odd}}(L)}=\eta_{|\mathcal{R}^+_{\textup{odd}}(L)} \right]
\end{equation}
remains bounded away from $0$, uniformly over $i$ and $\eta_{|\mathcal{R}^+_{\textup{odd}}(L)}\in A_{K,i}$. Furthermore, these events are independent of each other conditionally on $\tau_{|\mathcal{R}^+_{\textup{odd}}(L)}$. Therefore, for every $M$ large enough we have that $\nu^{+,+}_{\mathcal{R}(L,M),\beta}[\iota> N]\geq 1-\varepsilon$. Let us now condition on the event $\{\iota = i\}$ for some $i>N$, which is measurable with respect to $\tau_{(x,j)}$ for $|j|\geq 2i$ and $x \in \{-L,\ldots,L\}^{d-1}$. By the domain Markov property and monotonicity in the boundary conditions, $\nu^{+,+}_{\mathcal{R}(L,M),\beta}[\mathcal A \mid \iota=i]\leq \nu^{+,0}_{\mathcal{R}(L,2i-1),\beta}[\mathcal A]$. By monotonicity in the volume of the latter measure, $\nu^{+,+}_{\mathcal{R}(L,M),\beta}[\mathcal A \mid \iota=i]\leq \nu^{+,0}_{\mathcal{R}(L,M),\beta}[\mathcal A]$. Hence,
\begin{multline}
\nu^{+,+}_{\mathcal{R}(L,M),\beta}[\mathcal A]\leq \sum_{i=N+1}^M \nu^{+,+}_{\mathcal{R}(L,M),\beta}[\mathcal A \mid \iota=i]\nu^{+,+}_{\mathcal{R}(L,M),\beta}[\iota=i]+\nu^{+,+}_{\mathcal{R}(L,M),\beta}[\iota\leq N]\\\leq \nu^{+,0}_{\mathcal{R}(L,M),\beta}[\mathcal A]+\varepsilon. 
\end{multline}
Since $\varepsilon>0$ is arbitrary, this implies that $\limsup_{M\to\infty} \nu^{+,+}_{\mathcal{R}(L,M),\beta}[\mathcal A]\leq \nu^{+,0}_{\mathcal{R}(L),\beta}[\mathcal A]$. It follows that $\nu^{+,+}_{\mathcal{R}(L,M),\beta}$ converges to $\nu^{+,0}_{\mathcal{R}(L),\beta}$, as desired.

For $\langle \cdot \rangle^{+,-}_{\mathcal{R}(L,M),\beta}$, note that by flipping the sign of all spins $\tau_x$ for $x\in \{-L-1,\ldots,L+1\}^{d-1}\times \{-M,\ldots,-1\}$, and denoting the resulting configuration $\tilde \tau$,  we obtain
\begin{equation}
\langle F(\tau) \rangle^{+,-}_{\mathcal{R}(L,M),\beta}=  \frac{\langle F(\tilde \tau) \prod_{x\in \{-L,\ldots,L\}^{d-1}}\exp(-2\beta \tau_{(x,-1)} \tau_{(x,0)} )\rangle^{+,+}_{\mathcal{R}(L,M),\beta}}{\langle \prod_{x\in \{-L,\ldots,L\}^{d-1}}\exp(-2\beta \tau_{(x,-1)} \tau_{(x,0)} )\rangle^{+,+}_{\mathcal{R}(L,M),\beta}}  
\end{equation}
for every bounded (local) measurable function $F$. The convergence of $\langle \cdot \rangle^{+,-}_{\mathcal{R}(L,M),\beta}$ then follows from the convergence and the regularity of $\langle \cdot \rangle^{+,+}_{\mathcal{R}(L,M),\beta}$.

\vspace{5pt}

We turn to the proof of \eqref{eq:magnetisation convergence}. The argument is an adaptation of \cite[Lemma~5.6]{gunaratnam2025supercritical}. We only handle $\langle \tau_{(x,N)} \rangle^{+,-}_{\mathcal{R}(L),\beta}$, as the case of $\langle \tau_{(x,-N)} \rangle^{+,-}_{\mathcal{R}(L),\beta}$ follows by global flip. Let us define 
\begin{equation}
\iota' := \min \{i>0:~ \tau_{(x,2i)} >0, \, \forall x \in \{-L,\ldots, L \}^{d-1}\},
\end{equation}
with the convention that the minimum of the empty set is infinity.
Arguing as for $\iota$, we can deduce that $\iota'$ is almost surely finite. Let us now condition on the event $\{\iota' = i\}$, which is measurable with respect to $\tau_{(x,j)}$ for $j\leq 2i$ and $x \in \{-L,\ldots,L\}^{d-1}$. By the Markov property and monotonicity in boundary condition, we have that for every $N\geq 2i$
\begin{equation}
\langle \tau_{(x,N)} 
 \mid \iota' = i \rangle^{+,-}_{\mathcal{R}(L),\beta}\geq  \langle \tau_{(x,0)} \rangle^{+,0}_{\mathcal{R}(L,N-2i),\beta}.
\end{equation}
As a consequence,
\begin{equation}
    \langle \tau_{(x,N)}  \rangle^{+,-}_{\mathcal{R}(L),\beta}\geq \langle \tau_{(x,0)} \rangle^{+,0}_{\mathcal{R}(L,N-2i),\beta}\nu_{\mathcal{R}(L),\beta}^{+,-}[\iota'\leq i]+\langle \tau_{(x,N)}\mathds{1}_{\iota'>i}  \rangle^{+,-}_{\mathcal{R}(L),\beta}.
\end{equation}
By the Cauchy--Schwarz inequality and regularity, there exists a constant $C>0$ such that for every $i\geq 1$ we have
\begin{equation}
|\langle \tau_{(x,N)} 
\mathbbm{1}_{\iota'>i} \rangle^{+,-}_{\mathcal{R}(L),\beta}|\leq \sqrt{\langle \tau^2_{(x,N)}  \rangle^{+,-}_{\mathcal{R}(L),\beta}\nu_{\mathcal{R}(L),\beta}^{+,-}[\iota'>i]}
\leq C\sqrt{\nu_{\mathcal{R}(L),\beta}^{+,-}[\iota'>i]}=: \varepsilon_i.
\end{equation}
Thus, since (by the first part of the proof) $\lim_{N\to\infty}\langle \tau_{(x,0)} \rangle^{+,0}_{\mathcal{R}(L,N),\beta}=\langle \tau_{(x,0)} \rangle^{+,+}_{\mathcal{R}(L),\beta}$, for every $i>0$,
\begin{equation}
\liminf_{N\to\infty} \langle \tau_{(x,N)} 
 \rangle^{+,-}_{\mathcal{R}(L),\beta}\geq \langle \tau_{(x,0)} \rangle^{+,+}_{\mathcal{R}(L),\beta}\nu_{\mathcal{R}(L),\beta}^{+,-}[\iota'\leq i]-\varepsilon_i.
 \end{equation}
 Since $\iota'$ is finite almost surely, $\nu_{\mathcal{R}(L),\beta}^{+,-}[\iota'\leq i]\rightarrow 1$ and $\varepsilon_i\rightarrow 0$ as $i\rightarrow \infty$. We can deduce that 
\begin{equation}
\liminf_{N\to\infty} \langle \tau_{(x,N)} 
 \rangle^{+,-}_{\mathcal{R}(L),\beta}\geq \langle \tau_{(x,0)} \rangle^{+,+}_{\mathcal{R}(L),\beta}.
 \end{equation}
Moreover, by conditioning on the value of $\tau_z$ for $z\in \partial^- \mathcal{R}(L)$ and using the domain Markov property and monotonicity in boundary conditions, for every $N>0$,
\begin{equation}
    \langle \tau_{(x,N)}\rangle^{+,-}_{\mathcal{R}(L),\beta}\leq \langle \tau_{(x,N)} \rangle^{+,+}_{\mathcal{R}(L),\beta}= \langle \tau_{(x,0)} \rangle^{+,+}_{\mathcal{R}(L),\beta}.
\end{equation}
The desired result follows readily from the last two displayed equations.
\end{proof}

\subsection{Strict positivity via Lebowitz's inequality}

We prove that, for every $\beta>\beta_c$, one has $\boldsymbol{\tau}(\beta) > 0$. We first state Lebowitz's inequality, which we prove in the next subsection. Given a moment function $A:V\to\mathbb{N}$, we recall the shorthand notation 
\begin{equation}
\tau_A=\prod_{x\in V} \tau^{A_x}_x.
\end{equation}
In the following statement, we consider a small generalisation of the spin measures considered above: we allow for inhomogeneous single-site measures that satisfy $(ii)$ in Definition \ref{def:single site}, and that are either even or supported on $\mathbb R^+$.

\begin{Prop} \label{prop: lebowitz}
Fix a finite graph $G = (V,E)$. Consider two sets of coupling constants $\{J_{xy}\}_{xy\in E}$ and $\{J'_{xy}\}_{xy\in E}$ satisfying $|J'_{xy}| \le J_{xy}$ for every $xy\in E$, and two external magnetic fields $\mathsf{h}$ and $\mathsf{h}'$ satisfying $|\mathsf{h}'_x|\le \mathsf{h}_x$ for every $x\in V$. Consider a family of single-site measures $(\rho_x)_{x\in V}$ such that for every $x\in V$, $\rho_x$ is either an even measure over $\mathbb{R}$ or $\rho_x$ is supported on $\mathbb{R}^+$. For every pair of moment functions $A, B:V\to\mathbb{N}$, one has
\begin{equation}
\langle \tau_A \tau_B \rangle - \langle \tau_A \tau_B \rangle' \ge |\langle \tau_A \rangle \langle\tau_B \rangle' - \langle \tau_A \rangle' \langle \tau_B \rangle|,
\end{equation}
where $\langle \cdot \rangle$ and $\langle \cdot \rangle'$ stand for the expectations under the measures $\nu_{G,\rho,J, \mathsf{h}}$ and $\nu_{G,\rho,J',\mathsf{h}'}$ respectively\footnote{Note that $\beta$ has been chosen equal to $1$ here which explains why it disappeared from the notation.}.
\end{Prop}

We are now equipped to prove the main result of this section.
\begin{Prop}\label{prop: positivity of tau}
For every $\beta \geq \beta_c$,
\begin{equation}
\boldsymbol{\tau}(\beta) \geq 2^d\int_{\beta_c}^\beta \big(\langle \tau_0 \rangle_u^+\big)^2 \mathrm{d}u.  
\end{equation}
In particular, for $\beta>\beta_c$, one has $\boldsymbol{\tau}(\beta)>0$.
\end{Prop}

\begin{proof}
Let $L, N\geq 1$ and $M> N$. We first estimate the derivative in $\beta$ of  $\boldsymbol{\tau}^{L,M}(\beta)$. By definition of $\boldsymbol{\tau}^{L,M}(\beta)$, one has
\begin{equation}\label{eq:pos surf ten 0}
\begin{aligned}
\frac{\mathrm{d}}{\mathrm{d}\beta}\boldsymbol{\tau}^{L,M} (\beta)&=\frac{1}{L^{d-1}} \sum_{xy\in \overline{E}(\mathcal R(L,M))} \left(\langle \tau_x\tau_y \rangle^{+,+}_{\mathcal R(L,M),\beta} - \langle \tau_x\tau_y \rangle^{+,-}_{\mathcal R(L,M),\beta}\right)
\\&\geq \frac{1}{L^{d-1}}\sum_{\substack{x\in \{-L,\ldots,L\}^{d-1}\\ -N\leq j\leq N}} \left( \langle \tau_{(x,j)} \tau_{(x,j+1)} \rangle_{\mathcal R(L,M), \beta}^{+,+} -  \langle \tau_{(x,j)} \tau_{(x,j+1)} \rangle_{\mathcal R(L,M), \beta}^{+,-}  \right).
\end{aligned}
\end{equation}
where the inequality follows from the fact that $\langle \tau_u\tau_v\rangle^{+,+}_{\mathcal{R}(L,M),\beta}\geq \langle \tau_u\tau_v\rangle^{+,-}_{\mathcal{R}(L,M),\beta}$ for every $u,v\in \overline{\mathcal{R}(L,M)}$.
We now send $M\to\infty$.
By Lemma~\ref{lem:++}, we have that $\langle \tau_{(x,j)} \tau_{(x,j+1)} \rangle_{\mathcal R(L,M), \beta}^{+,+} -  \langle \tau_{(x,j)} \tau_{(x,j+1)} \rangle_{\mathcal R(L,M), \beta}^{+,-}$ converges to 
\begin{equation}\label{eq:pos surf ten 1}
\langle \tau_{(x,j)} \tau_{(x,j+1)} \rangle_{\mathcal R(L), \beta}^{+,+} -  \langle \tau_{(x,j)} \tau_{(x,j+1)} \rangle_{\mathcal R(L), \beta}^{+,-}
\end{equation}
as $M$ tends to infinity.
Observe that the Lebowitz inequality applies to the measures $\langle \cdot \rangle^{+,+}_{\mathcal{R}(L),\beta}$ and $\langle \cdot\rangle^{+,-}_{\mathcal{R}(L),\beta}$. To see this note that $\rho$ and $\zeta$ satisfy the assumptions of Proposition~\ref{prop: lebowitz}, and we can express $\langle \cdot\rangle^{+,-}_{\mathcal{R}(L),\beta}$ as a spin measure on $\overline{\mathcal{R}(L)}$ with single-site measure $\rho$ on $\mathcal{R}(L)$ and $\zeta$ on $\partial^{\textup{ext}}\mathcal{R}(L)$, and with coupling constants being equal to $-\beta$ for every edge $xy$ with $x\in \partial^-\mathcal{R}(L)$ and $y\in \overline{\mathcal{R}(L)}\setminus \partial^-\mathcal{R}(L)$, and otherwise being equal to $\beta$.  By Lebowitz's inequality and vertical translation invariance of $\langle \cdot \rangle^{+,+}_{\mathcal{R}(L),\beta}$, we obtain that the quantity in \eqref{eq:pos surf ten 1} is at least
\begin{multline}
\langle \tau_{(x,j)} \rangle^{+,+}_{\mathcal R(L),\beta}\langle \tau_{(x,j+1)} \rangle^{+,-}_{\mathcal R(L),\beta} - \langle \tau_{(x,j+1)} \rangle^{+,+}_{\mathcal R(L),\beta}\langle \tau_{(x,j)} \rangle^{+,-}_{\mathcal R(L),\beta} 
\\
= \langle \tau_{(x,0)} \rangle^{+,+}_{\mathcal R(L),\beta} \left( \langle \tau_{(x,j+1)} \rangle^{+,-}_{\mathcal R(L),\beta} - \langle \tau_{(x,j)} \rangle^{+,-}_{\mathcal R(L),\beta} \right).
\end{multline}
Taking the liminf as $M$ tends to infinity in \eqref{eq:pos surf ten 0} and telescoping gives
\begin{equation}
\liminf_{M\rightarrow \infty}\Big(\frac{\mathrm{d}}{\mathrm{d}\beta}\boldsymbol{\tau}^{L,M}(\beta) \Big)\geq\frac{1}{L^{d-1}}\sum_{x\in \{-L,\ldots,L\}^{d-1}} \langle \tau_{(x,0)} \rangle^{+,+}_{\mathcal R(L),\beta}\left(\langle \tau_{(x,N+1)} \rangle^{+,-}_{\mathcal R(L),\beta} - \langle \tau_{(x,-N)} \rangle^{+,-}_{\mathcal R(L),\beta} \right).
\end{equation}
We now send $N$ to infinity and use the second part of Lemma~\ref{lem:++} to obtain that 
\begin{equation}
\liminf_{M\rightarrow \infty}\Big(\frac{\mathrm{d}}{\mathrm{d}\beta}\boldsymbol{\tau}^{L,M}(\beta) \Big)\geq  \frac 2{L^{d-1}} \sum_{x\in \{-L,\ldots,L\}^{d-1}} \left(\langle \tau_{(x,0)} \rangle^{+,+}_{\mathcal R(L),\beta}\right)^2.
\end{equation}
Note that $\langle \tau_{(x,0)} \rangle^{+,+}_{\mathcal R(L),\beta}\geq \langle \tau_0 \rangle^+_{\beta}$ by \eqref{eq:plus monotonicity}, Lemma~\ref{lem:++}, and translation invariance of $\langle \cdot \rangle^+_{\beta}$. Hence, integrating between $\beta_c$ and $\beta$, and using Fatou's lemma and the fact that $|\{-L,\ldots, L\}^{d-1}|=(2L+1)^{d-1}$,
\begin{equation}
\boldsymbol{\tau}^L(\beta)=\liminf_{M\rightarrow\infty}\boldsymbol{\tau}^{L,M}(\beta)\geq 
\liminf_{M\rightarrow\infty}\Big(\boldsymbol{\tau}^{L,M}(\beta)-\boldsymbol{\tau}^{L,M}(\beta_c)\Big)\geq 2^d\int_{\beta_c}^{\beta} \left( \langle \tau_0 \rangle^{+}_{u}\right)^2 \mathrm{d}u,
\end{equation}
where we used Lemma \ref{lem: surface tension expression} in the first inequality to argue that $\boldsymbol{\tau}^{L,M}(\beta_c)\geq 0$.
The desired result follows from taking the liminf as $L$ tends to infinity.
\end{proof}

\subsection{Proof of disconnection probability decay for wired measure}

We are now equipped to prove Theorem~\ref{thm: wired disconnection}.

\begin{proof}[Proof of Theorem~\textup{\ref{thm: wired disconnection}}]
By Lemma~\ref{lem:++} and the Edwards--Sokal coupling, $\Psi^1_{\mathcal{R}(L,M),\beta}$ converges weakly as $M\to\infty$ to a measure denoted by $\Psi^1_{\mathcal{R}(L),\beta}$. 
We claim that
\begin{equation}\label{eq:pf cv non local event}
\lim_{M\rightarrow \infty} \Psi^1_{\mathcal{R}(L,M),\beta}[\partial^+\mathcal{R}(L,M)\centernot\longleftrightarrow \partial^-\mathcal{R}(L,M)]=\Psi^1_{\mathcal{R}(L),\beta}[\partial^+\mathcal{R}(L)\centernot\longleftrightarrow \partial^-\mathcal{R}(L)],
\end{equation}
which implies the desired result by Lemma~\ref{lem: surface tension expression} and Proposition~\ref{prop: positivity of tau}.

Let us prove \eqref{eq:pf cv non local event}. On the one hand, for every $L\leq N<M$ we have
\begin{multline}
\Psi^1_{\mathcal{R}(L,M),\beta}[\partial^+\mathcal{R}(L,M)\centernot\longleftrightarrow \partial^-\mathcal{R}(L,M)]\\ \leq \Psi^1_{\mathcal{R}(L,M),\beta}[\partial^+\mathcal{R}(L)\centernot\longleftrightarrow \partial^-\mathcal{R}(L) \text{ in } \overline{\mathcal{R}(L,N)}],     
\end{multline}
and by first sending $M$ to infinity and then $N$ to infinity, the right-hand side converges to $\Psi^1_{\mathcal{R}(L),\beta}[\partial^+\mathcal{R}(L)\centernot\longleftrightarrow \partial^-\mathcal{R}(L)]$. On the other hand, one has
\begin{equation} \label{eq: R+LMevent}
\lim_{N\to\infty}\lim_{M\to \infty}\Psi^1_{\mathcal{R}(L,M),\beta}[\partial^+\mathcal{R}(L,M)\longleftrightarrow \partial^-\mathcal{R}(L,M), \partial^+\mathcal{R}(L)\centernot\longleftrightarrow \partial^-\mathcal{R}(L) \text{ in } \overline{\mathcal{R}(L,N)}]=0,
\end{equation} 
To see why, first note---by a finite energy argument similar to the one used in Lemma~\ref{lem:++}---that the probability that there exists at least one $K \in \{1,\ldots,N\}$ such that in $\{-L-1,\dots,L+1\}^{d-1} \times \{-K,K\}$ all of the induced edges are open, can be made arbitrarily close to $1$ by choosing $N$ large enough, uniformly in $M>N$. Such an event is incompatible with the intersection of events in \eqref{eq: R+LMevent}. This proves the claim and concludes the proof.
\end{proof}

\subsection{Proof of the Lebowitz inequality}

We now proceed with the proof of Proposition~\ref{prop: lebowitz}. The first step in the proof is to rewrite $\langle \tau_A \tau_B \rangle - \langle \tau_A \tau_B \rangle'$ in terms of a duplicated system consisting of two independent spin configurations. After rotating the duplicated system and conditioning on its absolute-value field, the desired inequality reduces to Griffiths' inequality for a ferromagnetic spin system on an enlarged graph. Due to our assumptions on the single-site measures, the latter is an Ising system: spins take values in $\{\pm 1\}$ on vertices for which $\rho_x$ is an even measure, while they are fixed to the value $+1$ on vertices where the single-site measure $\rho_x$ is supported on $\mathbb{R}^+$.

\begin{proof}[Proof of Proposition~\textup{\ref{prop: lebowitz}}]
Let $\tau, \tau'$ be configurations sampled independently according to the measures $\mu:=\nu_{G,\rho,J,\mathsf{h}}$ and $\mu':=\nu_{G,\rho,J',\mathsf{h}'}$.
It suffices to prove that
\begin{equation}\label{eq:pfleb1}
\mathbb{E}_{\mu\otimes\mu'} [\tau_A\tau_B-\tau_A'\tau_B'] \ge \mathbb{E}_{\mu\otimes\mu'}[\tau_A\tau_B'-\tau_A'\tau_B],
\end{equation}
because the right-hand side equals $\langle \tau_A \rangle \langle\tau_B \rangle' - \langle \tau_A \rangle' \langle \tau_B \rangle$, and the second required inequality follows from exchanging the roles of $A$ and $B$.
Observe that \eqref{eq:pfleb1} can be rewritten as
\begin{equation}\label{eq:pfleb2}
\mathbb{E}_{\mu\otimes\mu'} [(\tau_A + \tau'_A)(\tau_B - \tau'_B)] \ge 0.
\end{equation}
To show \eqref{eq:pfleb2}, we introduce new spin variables labelled by the set of vertices $V$:
\begin{equation}
u_x := \tau_x + \tau'_x, \qquad v_x := \tau_x - \tau'_x.
\end{equation}
In terms of these new variables, we can write 
\begin{equation}
\tau_A+\tau_A' = \prod_{x\in V} \left(\frac{u_x+v_x}{2}\right)^{A_x} + \prod_{x\in V} \left(\frac{u_x-v_x}{2}\right)^{A_x}=2^{-|A|}\sum_{S\le A}\binom{A}{S}(1+(-1)^{|A|-|S|})u_S v_{A\setminus S},
\end{equation}
and
\begin{equation}
\tau_B-\tau_B' = \prod_{y\in V} \left(\frac{u_y+v_y}{2}\right)^{B_y} - \prod_{y\in V} \left(\frac{u_y-v_y}{2}\right)^{B_y}=
2^{-|B|}\sum_{S\le B}\binom{B}{S}(1-(-1)^{|B|-|S|})u_S v_{B\setminus S},
\end{equation}
where $|A|:=\sum_{x\in V} A_x$, the notation $S\le A$ means that $S:V \to\mathbb{N}$ is a moment function such that $S_x\le A_x$ for every $x\in V$, $A\setminus S$ denotes the moment function which is equal to $A_x- S_x$ for every $x\in V$, and $\binom{A}{S}:=\prod_{x\in V}\binom{A_x}{S_x}$.
Hence, it suffices to show that for all moment functions $A,B$,  
\begin{equation}\label{eq:what-we-need-to-prove}
\mathbb{E}_{\mu\otimes\mu'}[u_Av_B]\geq 0.
\end{equation}
We now rewrite the measure $\mu\otimes \mu'$ in terms of the new spin variables $u,v$, the new coupling constants 
\begin{equation}
K_{xy} := \frac{1}{2}(J_{xy} + J'_{xy}), \qquad L_{xy} := \frac{1}{2}(J_{xy} - J'_{xy}),
\end{equation}
(both of which are non-negative by the assumption on $J_{xy}, J'_{xy}$), as well as the new external magnetic fields
\begin{equation}
K_x:=\frac{1}{2}(\mathsf{h}_x+\mathsf{h}'_x), \qquad L_x:= \frac{1}{2}(\mathsf{h}_x-\mathsf{h}'_x)
\end{equation}
(both of which are non-negative by the assumption on $\mathsf{h}_x,\mathsf{h}'_x$).
We have 
\begin{align*}
   \mathrm{d}\mu\otimes \mu'[(\tau, \tau')] \propto \exp\left\{ \sum_{xy\in E} (J_{xy} \tau_x \tau_y +J'_{xy}\tau_x'\tau_y') +\sum_{x\in V} (\mathsf{h}_x\tau_x+\mathsf{h}_x'\tau_x') \right\}\prod_{x\in V}\mathrm{d}\rho_x(\tau_x)\mathrm{d}\rho_x(\tau'_x),
\end{align*}
and using the new variables we can write 
\begin{align*}
J_{xy} \tau_x \tau_y +J'_{xy}\tau_x'\tau_y'
&=
(K_{xy}+L_{xy})\cdot \frac{u_x+v_x}{2}\cdot\frac{u_y+v_y}{2}+(K_{xy}-L_{xy})\cdot \frac{u_x-v_x}{2}\cdot\frac{u_y-v_y}{2}
\\&=
K_{xy}\cdot \frac{u_xu_y+v_xv_y}{2}+L_{xy}\cdot \frac{u_xv_y+v_xu_y}{2}
\end{align*}
and
\begin{align*}
\mathsf{h}_x \tau_x +\mathsf{h}'_x\tau_x'
&=
(K_x+L_x)\cdot \frac{u_x+v_x}{2}+(K_x-L_x)\cdot \frac{u_x-v_x}{2}=
K_x\cdot u_x+L_x\cdot v_x.
\end{align*}
Also, let $\nu_x$ denote the pushforward of $\rho_x\times \rho_x$ by the map $F(a,b)=\left(a+b,a-b\right)$.
Since $F$ is a continuous bijection, the whole measure can now be written as 
\begin{align}
\frac{\mathrm{d}\mu\otimes \mu'[(u, v)]}{\prod_{x\in V}\mathrm{d}\nu_x(u_x,v_x)}\propto \exp\Biggl\{\sum_{xy\in E} \left(K_{xy} \frac{u_xu_y+v_xv_y}{2}+L_{xy} \frac{u_xv_y+v_xu_y}{2}\right) +\sum_{x\in V}(K_xu_x+L_xv_x)\Biggr\}.
\end{align}

We say that a measure $\nu$ on $\mathbb{R}^2$ is even if for every Borel measurable set $A$, we have $\nu(f_1(A))=\nu(f_2(A))=\nu(A)$, where $f_1(x,y)=(-x,y)$ and $f_2(x,y)=(x,-y)$. We claim that for every $x\in V$, the above measure $\nu_x$ is either an even measure on $\mathbb{R}^2$ or it is supported on $\mathbb{R}^+\times \mathbb{R}$ and for every Borel measurable set $A$, $\nu_x(f_2(A))=\nu_x(A)$. Indeed, it suffices to consider sets of the form $A_1\times A_2$, since these sets generate the $\sigma$-algebra. If $\rho_x$ is an even measure, then $\rho_x\otimes \rho_x$ is an even measure on $\mathbb{R}^2$ and invariant under $f_3(x,y)=(-y,-x)$. Hence
\begin{equation}\nu_x(A_1,A_2)=\rho_x\otimes \rho_x(F^{-1}(A_1,A_2))=\rho_x\otimes \rho_x(f_3(F^{-1}(A_1,A_2)))=\nu_x(-A_1,A_2),
\end{equation}
where we used that $F^{-1}(a,b)=(\tfrac{a+b}{2},\tfrac{a-b}{2})$ and $f_3(F^{-1}(a,b))=(\tfrac{-a+b}{2},\tfrac{-a-b}{2})$.
If $\rho_x$ is supported on $\mathbb{R}^+$, then $\nu_x$ is supported on $\mathbb{R}^+\times \mathbb{R}$ since $F^{-1}(A_1,A_2)\subset (\mathbb{R}^+)^2$ implies that $A_1\subset \mathbb{R}^+$.
On the other hand, in any case, by the invariance of $\rho_x\otimes\rho_x$ under $f_4(x,y)=(y,x)$,
\begin{equation}
\nu_x(A_1,A_2)=\rho_x\otimes\rho_x(F^{-1}(A_1,A_2))=\rho_x\otimes\rho_x(f_4(F^{-1}(A_1,A_2)))=\nu_x(A_1,-A_2).
\end{equation}
This proves the claim.

We can view the latter expression as a spin measure on the enhanced graph $G\times K_2$ with vertex set $V^2:=\{(x,i): x\in V, i\in \{1,2\}\}$ and edge set $E^2=\{\{(x,i),(y,j)\}: xy\in E, i,j\in \{1,2\}\}$. Let us write for $x\in V$, $\tau_{(x,1)}=u_x$ and $\tau_{(x,2)}=v_x$. Then, by the above property of each $\nu_x$, conditioning on the absolute value field $(|\tau_z|)_{z\in V^2}$, the sign field $(\textup{sgn}(\tau_z))_{z\in V^2}$ is distributed according to a (ferromagnetic) Ising measure on $G\times K_2$ with coupling constants $\tfrac{1}{2} K_{xy}\cdot |\tau_{(x,i)}|\cdot |\tau_{(y,j)}|$ for $i=j$ and $\tfrac{1}{2}L_{xy}\cdot |\tau_{(x,i)}|\cdot |\tau_{(y,j)}|$ for $i\neq j$, external magnetic field $K_x|\tau_{(x,1)}|$ and $L_x|\tau_{(x,2)}|$, and non-negative boundary conditions. Observe that when $x\in V$ is such that $\rho_x$ is supported on $\mathbb R^+$, one has $\textup{sgn}(\tau_{(x,1)})=1$.

Hence, the inequality \eqref{eq:what-we-need-to-prove} follows from Griffiths' first inequality for this Ising measure. This completes the proof.
\end{proof}

\begin{Rem}
Let us describe an alternative way of handling the general case in which $\rho_x$ is allowed to be supported on $\mathbb{R}^+$. The inequality can be reduced to the case of even single-site measures as follows. Let $S$ denote the set of vertices whose single-site measure is supported on $\mathbb{R}^+$. For $T>0$, consider two spin systems whose external magnetic fields are given by $h_x+T$ and $h'_x+T$, respectively, for every $x\in S$, and by $h_x$ and $h'_x$ otherwise. In the limit $T\to\infty$, the spins on $S$ are forced to be positive. To counterbalance the change in the measure caused by increasing the magnetic field, symmetrise each $\rho_x$ for $x\in S$ in a tilted way, specifically, define $\tilde{\rho}_x$ by $\mathrm{d}\tilde{\rho_x}(t)= e^{-T\cdot |t|}\cdot (\mathrm{d}\rho_x(t)+\mathrm{d}\rho_x(-t))$. For $x\notin S$, both the single-site measure and the external magnetic field remain unchanged. Applying the Lebowitz inequality to the system with single-site measures $\widetilde{\rho}_x$ and then letting $T\to\infty$ yields the desired result.
\end{Rem}

\section{Comparison between boundary conditions}\label{sec:comparison_bc}

The aim of this section is to use Theorem~\ref{thm: wired disconnection} to prove Theorem~\ref{thm:sharpness}. We first prove that crossings happen with probability exponentially close to 1 under a free measure on an appropriate strip $\mathcal{R}(L,M)$ with a well-chosen (large) magnetic field on the boundary. Then, we use a differentiation argument and the almost everywhere uniqueness of half-space measures to deduce a similar crossing estimate for the free measure (i.e.\ without magnetic field on the boundary), thereby proving Theorem~\ref{thm:sharpness}. 

\subsection{From wired to free with a large boundary magnetic field}
For $h>0$, let $\Psi^{0,h}_{\mathcal{R}(L,M),\beta}$ denote the random cluster measure $\Psi^0_{\mathcal{R}(L,M),\beta,\mathsf{h}}$ with external magnetic field $\mathsf{h}$ equal to $h$ on $\partial^{\rm bot} \mathcal{R}(L,M)\cup \partial^{\rm top} \mathcal{R}(L,M)$ and $0$ otherwise, where
\begin{align}
    \partial^{\rm bot} \mathcal{R}(L,M):=\{-L,\ldots,L\}^{d-1}\times \{-M\},  \\ \partial^{\rm top} \mathcal{R}(L,M):= \{-L,\ldots,L\}^{d-1}\times \{M\}.
\end{align}

\begin{Lem}\label{lem: strip to rectangle}
Let $\beta>\beta_c$. For every $h>0$ sufficiently large, there exist $\delta,c_2>0$ such that, for every $L$ large enough,
\begin{equation}
\Psi^{0,h}_{\mathcal R(L,\delta L),\beta}[\partial^{{\rm bot}} \mathcal R(L, \delta L) \longleftrightarrow\partial^{{\rm top}} \mathcal R(L, \delta L)] \geq 1 - e^{-c_2 L^{d-1}}. 
\end{equation}
\end{Lem}
\begin{proof}
Let $\beta>\beta_c$ and consider $\delta\in (0,1)$ a small enough constant to be chosen later. Recall Theorem \ref{thm: wired disconnection}. Our first goal is to compare $\Psi^1_{\mathcal R(L),\beta}$ with $\Psi^{(w,h)}_{\mathcal{R}(L,\delta L),\beta}$, where we recall that the latter is the wired random cluster measure on $\mathcal{R}(L,\delta L)$ with absolute value field equal to $h$ in $\partial^{\textup{ext}}\mathcal{R}(L,\delta L)$, up to a cost that can be made smaller than $e^{(c_1/2)L^{d-1}}$ by choosing $h$ to be large enough. We then force certain edges on the lateral sides of $\mathcal{R}(L,\delta L)$ to be closed, thereby ensuring that any path from $\partial^{+}\mathcal R(L,\delta L)$ to $\partial^{-}\mathcal R(L,\delta L)$ must connect $\partial^{\rm bot}\mathcal R(L,\delta L)$ to $\partial^{\rm top}\mathcal R(L,\delta L)$. To make the cost of this surgery sufficiently small, we will need to choose $\delta$ small enough. The proof is similar to that of \cite[Lemma 5.8]{gunaratnam2025supercritical} and is split into two steps. 

\paragraph{Step 1.}
Introduce
$A := \mathcal{R}(L+1,\delta L+1)\setminus \mathcal{R}(L,\delta L)$ and $B:=\mathcal{R}(L,\delta L)\setminus \mathcal{R}(L-1,\delta L)$ and write $V:=A\cup B$.
By the FKG inequality and regularity, we can choose $H>0$ sufficiently large so that
\begin{equation}
\Psi^1_{\mathcal{R}(L),\beta}\big[\mathsf{a}|_V\leq H \big]\geq e^{-(c_1/2) \, L^{d-1}},
\end{equation}
where $c_1$ is the constant from Theorem~\ref{thm: wired disconnection}. Hence,
\begin{equation}\label{eq: proof lemma T1}
\Psi^1_{\mathcal{R}(L),\beta} \big[\partial^+\mathcal{R}(L)\centernot\longleftrightarrow\partial^-\mathcal{R}(L)]
\geq 
\Psi^1_{\mathcal{R}(L),\beta} \big[\partial^+\mathcal{R}(L)\centernot\longleftrightarrow\partial^-\mathcal{R}(L) \mid \mathsf{a}|_V\leq H \big]\,
e^{-(c_1/2) L^{d-1}}.
\end{equation}
By monotonicity in the absolute value and the FKG inequality from Proposition~\ref{prop: wired random cluster properties} (applied to the conditional measure $\Psi^1_{\mathcal R(L), \beta}[\ \cdot \mid \mathsf a|_V = H]$), we obtain
\begin{equation}
\begin{aligned}
\Psi^1_{\mathcal{R}(L),\beta} \big[\partial^+\mathcal{R}(L)\centernot\longleftrightarrow\partial^-\mathcal{R}(L) \mid \mathsf{a}|_V\leq H \big]
&\geq 
\Psi^1_{\mathcal{R}(L),\beta} \big[\partial^+\mathcal{R}(L)\centernot\longleftrightarrow\partial^-\mathcal{R}(L) \mid \mathsf{a}|_V=H \big] \\
&\geq 
\Psi^1_{\mathcal{R}(L),\beta} \big[\partial^+\mathcal{R}(L)\centernot\longleftrightarrow\partial^-\mathcal{R}(L) \mid \mathsf a|_V=H, \mathcal{A} \big],
\end{aligned}
\end{equation}
where
\begin{equation}
\mathcal A := \{\omega_e = 1,\ \forall \, e \in E(A)\}.
\end{equation}
Let $\mathsf{H}$ denote the absolute value field equal to $H$ on
$A$ and $0$ elsewhere. 
On the event $\mathcal A$, all vertices in $A$ are in the same cluster, so the domain Markov property from Proposition~\ref{prop: wired random cluster properties} yields
\begin{equation}\label{eq: proof lemma T2}
\Psi^1_{\mathcal R(L),\beta}\big[\, \cdot \mid \mathsf a|_V=H, \mathcal A\big]
=
\Psi^{(w,\mathsf{H})}_{\mathcal R(L,\delta L),\beta}\big[\, \cdot \mid \mathsf a|_B=H \big],
\end{equation}
where $\Psi^{(w,\mathsf{H})}_{\mathcal R(L,\delta L),\beta}$ denotes the random cluster measure with $(\{\partial^{\textup{ext}} \mathcal{R}(L,\delta L)\},\mathsf{H})$ boundary conditions.
Combining \eqref{eq: proof lemma T1} and \eqref{eq: proof lemma T2}, and applying Theorem~\ref{thm: wired disconnection}, we obtain
\begin{equation}\label{eq: comparison 1}
\Psi^{(w,\mathsf{H})}_{\mathcal R(L,\delta L),\beta}
\big[\partial^+\mathcal{R}(L,\delta L)\centernot\longleftrightarrow\partial^-\mathcal{R}(L,\delta L) \mid \mathsf a|_B=H\big]
\leq e^{-(c_1/2) L^{d-1}}.
\end{equation}

\paragraph{Step 2.} We now move to the second part of the proof. Let us define the event
\begin{equation}
\mathcal{E}:= \{\partial^{\rm bot} \mathcal R(L, \delta L)\centernot\longleftrightarrow \partial^{\rm top} \mathcal R(L, \delta L)\}.
\end{equation}
Define the set of edges $I:=\{xy: x\sim y, x\in A, y\in B\}$. Observe that
\begin{equation}\label{eq:Iproofinclusion}
\mathcal{E}\cap \{\omega|_{I}=0\} \subset \{\partial^+ \mathcal R(L, \delta L)\centernot\longleftrightarrow \partial^- \mathcal R(L, \delta L)\}\cap \{\omega|_{I}=0\}.
\end{equation}
See Figure \ref{fig:I graph} for an illustration. 

\begin{figure}[htb]
    \centering
    \includegraphics[width=0.6\linewidth]{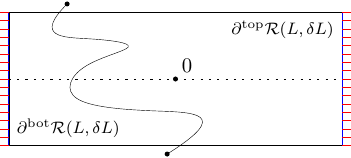}
    \caption{An illustration of \eqref{eq:Iproofinclusion}. The set $B$ is represented in blue. The edges of $I$ are represented in red. When $\omega_{|I}=0$ and $\partial^+\mathcal{R}(L,\delta)\connect{}\partial^-\mathcal{R}(L,\delta L)$, one also has that $\partial^\textup{bot}\mathcal{R}(L,\delta L)\connect{}\partial^\textup{top}\mathcal{R}(L,\delta L)$.}
    \label{fig:I graph}
\end{figure}

Note that since $\mathsf{a}_x=H$ for every $x\in V$ under $\Psi^{(w,\mathsf{H})}_{\mathcal R(L,\delta L),\beta}\big[\, \cdot \mid \mathsf a|_B=H \big]$, by finite-energy there exists $t>0$ such that for every $e\in I$ and every $\theta\in \{0,1\}^{\overline{E}(\mathcal{R}(L,\delta L))\setminus \{e\}}$,
\begin{equation}
\Psi^{(w,\mathsf{H})}_{\mathcal R(L,\delta L),\beta}\big[\omega_e=0 \mid \mathsf a|_B=H,  \omega_{xy}=\theta_{xy} \, \forall \, xy \in \overline{E}(\mathcal{R}(L,\delta L))\setminus \{e\}\big]\geq e^{-t}.    
\end{equation}
This implies that 
\begin{equation}
\Psi^{(w,\mathsf{H})}_{\mathcal R(L,\delta L),\beta}\big[\omega|_{I}=0 \mid \mathsf a|_B=H\big]\geq e^{-t |I|},   
\end{equation}
so that by \eqref{eq:Iproofinclusion}
\begin{equation}
\Psi^{(w,\mathsf{H})}_{\mathcal R(L,\delta L),\beta}\big[\mathcal{E} \mid \omega|_{I}=0 ,\mathsf a|_B=H\big]\leq e^{t |I|} \Psi^{(w,\mathsf{H})}_{\mathcal R(L,\delta L),\beta}\big[\partial^+ \mathcal R(L, \delta L)\centernot\longleftrightarrow \partial^- \mathcal R(L, \delta L) \mid \mathsf a|_B=H\big].
\end{equation}
Since $|I|=O(\delta L^{d-1})$, we can choose $\delta$ to be small enough so that $t|I|\leq (c_1/4) L^{d-1}$. Combining the above with \eqref{eq: comparison 1} and using the domain Markov property we obtain that 
\begin{equation}\label{eq: comparison 2}
\Psi^{0,H}_{\mathcal{R}(L,\delta L),\beta}[\mathcal{E} \mid \mathsf{a}|_B=H]\leq e^{-(c_1/4) L^{d-1}},   
\end{equation}
where the measure $\Psi^{0,H}_{\mathcal{R}(L,\delta L),\beta}$ was defined above.
To remove the conditioning, we will use the fact that ``closing'' a vertex is less costly than closing an edge. More precisely, let $I':=\{xy: x\in B, y\in\mathcal{R}(L,\delta L)\cup \{\mathfrak{g}\}, x\sim y\}$. Observe that
\begin{equation}
\Psi^{0,H}_{\mathcal{R}(L,\delta L),\beta}[\mathcal{E} \mid \mathsf{a}|_B=H, \omega|_{I'}=0]=\frac{\Psi^{0,H}_{\mathcal{R}(L,\delta L),\beta}\Big[\mathbbm{1}_\mathcal{E} \prod_{xy\in I'}\sqrt{1-p_{xy}(\mathsf{a}^H)}\:\Big|\:\mathsf{a}|_B=0\Big]}{\Psi^{0,H}_{\mathcal{R}(L,\delta L),\beta}\Big[\prod_{xy\in I'}\sqrt{1-p_{xy}(\mathsf{a}^H)}\:\Big|\: \mathsf{a}|_B=0\Big]},   
\end{equation}
where $\mathsf{a}^H$ denotes the configuration which is equal to $H$ for $x\in B$ and is equal to $\mathsf{a}$ otherwise. Since $\sqrt{1-p_{xy}(\mathsf{a}^H)}$ is a decreasing function of $\mathsf{a}$, we can apply the FKG inequality and monotonicity in the absolute value field to deduce that 
\begin{equation}\label{eq: comparison 3}
\Psi^{0,H}_{\mathcal{R}(L,\delta L),\beta}[\mathcal{E} \mid \mathsf{a}|_B=H, \omega|_{I'}=0]\geq \Psi^{0,H}_{\mathcal{R}(L,\delta L),\beta}[\mathcal{E}\mid \mathsf{a}|_B=0]\geq \Psi^{0,H}_{\mathcal{R}(L-1,\delta L),\beta}[\mathcal{E}]\geq  \Psi^{0,H}_{\mathcal{R}(L,\delta L),\beta}[\mathcal{E}],
\end{equation}
where we used the fact that $\mathcal{E}$ is decreasing and Proposition \ref{prop: random cluster properties} in the last inequality.
Finally, by bounding $\mathsf{a}_x$ by $H$ for every $x\in \overline{B}$ and using the finite energy property as above, we obtain that there exists $t'>0$ such that
\begin{equation}
\Psi^{0,H}_{\mathcal{R}(L,\delta L),\beta}[\omega|_{I'}=0 \mid \mathsf{a}|_B=H]\geq e^{-t'|I'|}\geq e^{-(c_1/8) L^{d-1}},
\end{equation}
provided that $\delta$ is small enough. This implies that 
\begin{equation}
\Psi^{0,H}_{\mathcal{R}(L,\delta L),\beta}[\mathcal{E} \mid \mathsf{a}|_B=H, \omega|_{I'}=0]\leq e^{(c_1/8) L^{d-1}} \Psi^{0,H}_{\mathcal{R}(L,\delta L),\beta}[\mathcal{E} \mid \mathsf{a}|_B=H],
\end{equation}
which when combined with \eqref{eq: comparison 2} and \eqref{eq: comparison 3} implies the desired result.
\end{proof}

\subsection{Removing the boundary magnetic field}

\begin{Lem}\label{lem: large to free bc} Let $\beta>\beta_c$. There exist $\delta,c_3>0$ such that for every $L\geq 1$,
\begin{equation}\label{eq: large to free bc}
\Psi^0_{\mathcal R(L,\delta L),\beta}[ \partial^{\rm bot} \mathcal R(L, \delta L/2) \longleftrightarrow \partial^{\rm top} \mathcal R(L, \delta L/2)] \geq 1 - e^{-c_3 L^{d-1}}. 
\end{equation}
\end{Lem}
\begin{proof}
Let $\beta>\beta_c$. By Theorem~\ref{thm:half-space}, there exists $\beta'\in (\beta_c,\beta)$ such that for almost every $h>0$ we have that 
$\Psi^{1,h}_{\mathbb{H},\beta'}= \Psi^{0,h}_{\mathbb{H},\beta'}$. By Lemma~\ref{lem: strip to rectangle}, there exist $h,\delta,c_2>0$ such that, for every $L$ large enough 
\begin{equation}
\Psi^{0,h}_{\mathcal R(L,\delta L),\beta'}[\partial^{{\rm bot}} \mathcal R(L, \delta L) \longleftrightarrow\partial^{{\rm top}} \mathcal R(L, \delta L)] \geq 1 - e^{-c_2 L^{d-1}}
\end{equation} 
for every $L\geq 1$.

Write $\mathcal{D}=\{\partial^{\rm bot} \mathcal R(L, \delta L/2) \centernot\longleftrightarrow \partial^{\rm top} \mathcal R(L, \delta L/2) \text{ in } \overline{\mathcal R(L, \delta L/2)}\}$. By monotonicity in the coupling constants (see Proposition \ref{prop: random cluster properties}), we have 
\begin{equation}
\Psi^0_{\mathcal R(L,\delta L),\beta}[ \mathcal{D}]\leq \Psi^0_{\mathcal R(L,\delta L),\beta'}[ \mathcal{D}].
\end{equation}
Hence it suffices to show that there exists $c_3>0$ such that for every $L$ large enough,
\begin{equation}
\Psi^0_{\mathcal R(L,\delta L),\beta'}[ \mathcal{D}] \leq e^{-c_3 L^{d-1}}. 
\end{equation}
Without loss of generality, we may assume that for almost every $h>0$ we have that 
$\Psi^{1,h}_{\mathbb{H},\beta}= \Psi^{0,h}_{\mathbb{H},\beta}$, and prove the above inequality at parameter $\beta$. 

We will prove this by gradually decreasing the value of the external magnetic field from $s=h$ to $s=0$ in order to interpolate between $\Psi^{0,h}_{\mathcal R(L,\delta L),\beta}[ \mathcal{D} ]$ and $\Psi^{0}_{\mathcal R(L,\delta L),\beta}[ \mathcal{D} ]$. To this end, consider the logarithmic derivative with respect to $s$ and note that by Russo's formula (cf.\ the calculation of \cite[Theorem~(2.43)]{Grimmett2006RCM} in the case of the usual random cluster measure):
\begin{equation} \label{eq: log derivative}
\frac{\partial \log \Psi^{0,s}_{\mathcal R(L,\delta L),\beta}[ \mathcal{D} ]}{\partial s}= \sum_{x\in \partial^{\pm}\mathcal{R}(L,\delta L)} \Big(\Psi^{0,s}_{\mathcal R(L,\delta L),\beta}\left[F_x(\mathsf{a},\omega) \mid  \mathcal{D}\right] -\Psi^{0,s}_{\mathcal R(L,\delta L),\beta}\left[F_x(\mathsf{a},\omega)\right]\Big),
\end{equation}
where $\partial^{\pm}\mathcal{R}(L,\delta L)=\partial^{\rm bot}\mathcal{R}(L,\delta L)\cup \partial^{\rm top}\mathcal{R}(L,\delta L)$, and
\begin{equation}
F_x(\mathsf{a},\omega)=-\beta \mathsf{a}_x+\frac{\omega_{x\fg}}{p(\beta,s,\mathsf{a})_{x\fg}(1-p(\beta,s,\mathsf{a})_{x\fg})} \frac{\partial p(\beta,s,\mathsf{a})_{x\fg}}{\partial s}=-\beta\mathsf{a}_x+2\beta \mathsf{a}_x\frac{\omega_{x\fg}}{p(\beta,s,\mathsf{a})_{x\fg}}.
\end{equation}
Here the term $-\beta \mathsf{a}_x$ arises from differentiating $e^{-\beta \mathsf{a}_x s}=\sqrt{1-p(\beta,s,\mathsf{a})_{x\fg}}$ in \eqref{eq: random cluster explicit density}, while the other term arises from differentiating $\left(\frac{p(\beta,s,\mathsf{a})}{1-p(\beta,s,\mathsf{a})}\right)^{\omega_{x\fg}}$. Write $F_x=F^1_x-F^2_x$, where 
\begin{equation}
F^1_x(\mathsf{a},\omega)=2\beta \mathsf{a}_x\frac{\omega_{x\fg}}{p(\beta,s,\mathsf{a})_{x\fg}}, \quad
F^2_x(\mathsf{a},\omega)=\beta \mathsf{a}_x.
\end{equation}
Observe that $\mathsf{a}_x/p(\beta,s,\mathsf{a})_{x\fg}$
is an increasing function of $\mathsf{a}_x$, hence both $F^1_x$ and $F^2_x$ are increasing.
We will split the sum in \eqref{eq: log derivative} according to boundary and bulk contributions. Let $\epsilon \in (0,\delta/2)$ to be fixed below. Define
\begin{equation}
{\rm Bulk}(\varepsilon, L):= \{ x \in \partial^{\pm}\mathcal{R}(L,\delta L): {\rm d}( x, \partial \mathcal R(L,\delta L) \setminus \partial^\pm \mathcal R(L,\delta L)) \geq \varepsilon L \},
\end{equation}
and write
\begin{equation}
\eqref{eq: log derivative} =: I({\rm bulk}) + I({\rm boundary})
\end{equation}
with 
\begin{equation}\label{eq: derivative formula bulk}
I({\rm bulk}) = \sum_{x \in {\rm Bulk}(\varepsilon,L)}
\Big(\Psi^{0,s}_{\mathcal R(L,\delta L),\beta}\left[F_x(\mathsf{a},\omega) \mid  \mathcal{D}\right] -\Psi^{0,s}_{\mathcal R(L,\delta L),\beta}\left[F_x(\mathsf{a},\omega)\right]\Big),
\end{equation}
\begin{equation}\label{eq: derivative formula boundary}
I({\rm boundary}) = \sum_{x \in \partial^{\pm}\mathcal{R}(L,\delta L)\setminus{\rm Bulk}(\varepsilon,L)}
\Big(\Psi^{0,s}_{\mathcal R(L,\delta L),\beta}\left[F_x(\mathsf{a},\omega) \mid  \mathcal{D}\right] -\Psi^{0,s}_{\mathcal R(L,\delta L),\beta}\left[F_x(\mathsf{a},\omega)\right]\Big).
\end{equation}

Let us treat the boundary contribution first. 
Notice that by regularity, there exists $C>0$ such that, for every $x\in \partial^{\pm}\mathcal{R}(L,\delta L)$ and every $s\in [0,h]$,
\begin{equation}
    \max \left\{ \left|\Psi^{0,s}_{\mathcal R(L,\delta L),\beta}\left[F_x(\mathsf{a},\omega)\right]\right|,  
    \left|\Psi^{0,s}_{\mathcal R(L,\delta L),\beta}\left[F_x(\mathsf{a},\omega) \mid  \mathcal{D}\right]\right| \right\} \leq C,\label{eq: general control summand}
\end{equation}
where for the second term we used that
\begin{align}
\left|\Psi^{0,s}_{\mathcal R(L,\delta L),\beta}\left[F_x(\mathsf{a},\omega) \mid  \mathcal{D}\right]\right|&\leq \Psi^{0,s}_{\mathcal R(L,\delta L),\beta}[F^1_x(\mathsf{a},\omega) \mid  \mathcal{D}] + \Psi^{0,s}_{\mathcal R(L,\delta L),\beta}[F^2_x(\mathsf{a},\omega) \mid  \mathcal{D}]
\\& \leq
\Psi^{0,s}_{\mathcal R(L,\delta L),\beta}[F^1_x(\mathsf{a},\omega)] + \Psi^{0,s}_{\mathcal R(L,\delta L),\beta}[F^2_x(\mathsf{a},\omega)],
\end{align}
with the second inequality following from the FKG inequality (recall that $\mathcal{D}$ is decreasing and $F^1_x(\mathsf{a},\omega)$ is increasing). Hence, there exists $C_1 > 0$ such that, for every $s \in (0,h]$,
\begin{equation}
|I({\rm boundary})| \leq C_1 \varepsilon L^{d-1}.
\end{equation}

We now analyse the bulk contribution, which requires some more care. Without loss of generality, we only treat the case $x\in \partial^{\rm bot}\mathcal{R}(L,\delta L)$. 
Define $\mathcal{B}(L,\varepsilon,x):=(\Lambda_{\varepsilon L}\cap \mathbb H) +x$, and for $\mathsf{b}\in (\mathbb{R}^+)^{\partial \mathcal{B}(L,\varepsilon,x)}$,
\begin{equation}
\mathsf{h}(x;s)_y:=
\begin{cases}
s & \text{for } y\in \partial \mathcal{B}(L,\varepsilon,x)\cap \{y\in \mathbb Z^d: y_d=-\delta L\},\\
0 & \text{otherwise.}
\end{cases}    
\end{equation} 
See Figure \ref{fig:removinglargemagfield} for an illustration. 

\begin{figure}[htb]
    \centering
    \includegraphics[width=0.5\linewidth]{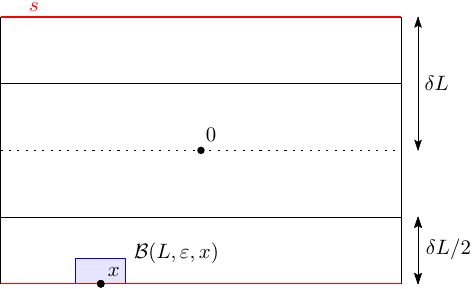}
    \caption{An illustration of the proof of Lemma \ref{lem: large to free bc}. The blue region is the half box $\mathcal{B}(L,\varepsilon,x)$. The top and bottom boundaries of $\mathcal{R}(L,\delta L)$ (in red) carry a magnetic field $s$. The law of $G_x(\mathsf{a},\omega)$ within $\mathcal{B}(L,\varepsilon,x)$ approximates that of $G_x(\mathsf{a},\omega)$ under the half-space measures $\Psi^{0,s}_{\mathbb H,\beta}$ or $\Psi^{1,s}_{\mathbb H,\beta}$. Thanks to Theorem \ref{thm:half-space}, these measures coincide for almost every $s\in (0,h]$.}
    \label{fig:removinglargemagfield}
\end{figure}

Note that the event $\mathcal{D}$ is supported on $\mathcal{R}(L,\delta L/2)$.
Hence, if $x\in \partial^{\rm bot}\mathcal{R}(L,\delta L)$ with $|x|\leq L-\varepsilon L$, by partitioning according to the boundary conditions in $\mathcal{B}(L,\varepsilon,x)$ and using the monotonicity in boundary conditions (see Proposition \ref{prop: random cluster properties}) we get
\begin{align}\label{eq: F 1 lower bound}
\Psi^{0,s}_{\mathcal R(L,\delta L),\beta}[F^1_x(\mathsf{a},\omega) \mid  \mathcal{D}]&\geq \Psi^{0}_{\mathcal{B}(L,\varepsilon,x),\beta,\mathsf{h}(x;s)}[F^1_x(\mathsf{a},\omega)],
\\ \label{eq: F 2 lower bound}
\Psi^{0,s}_{\mathcal R(L,\delta L),\beta}[F^2_x(\mathsf{a},\omega)\mid  \mathcal{D}]
&\geq  \Psi^{0}_{\mathcal{B}(L,\varepsilon,x),\beta,\mathsf{h}(x;s)}[F^2_x(\mathsf{a},\omega)].
\end{align}
Now, let us derive upper bounds on $\Psi^{0,s}_{\mathcal R(L,\delta L),\beta}[F^1_x(\mathsf{a},\omega) ]$ and $\Psi^{0,s}_{\mathcal R(L,\delta L),\beta}[F^2_x(\mathsf{a},\omega)\mid  \mathcal{D}]$. By Remark \ref{rem:plus}, we may choose $\zeta$ in Proposition \ref{prop:reg_plus} in such a way that it is supported on $[h,\infty)$.
By monotonicity in the boundary conditions and monotonicity of the wired measure in the volume (see Proposition \ref{prop: wired random cluster properties}),
\begin{equation} \label{eq: F 1 upper bound}
\Psi^{0,s}_{\mathcal R(L,\delta L),\beta}[F^1_x(\mathsf{a},\omega)]
\leq \Psi^{1,s}_{\mathcal R(L,\delta L),\beta}[F^1_x(\mathsf{a},\omega)]\leq \Psi^{1,s}_{\mathcal{B}(L,\varepsilon,x),\beta}[F^1_x(\mathsf{a},\omega)],
\end{equation}  
where we abuse notation and let $\Psi^{1,s}_{\mathcal R(L,\delta L),\beta}$ and $\Psi^{1,s}_{\mathcal{B}(L,\varepsilon,x),\beta}$ denote half-space wired measures in the (translated) half-space 
$\mathbb H+x$. 
A similar argument gives that
\begin{equation} \label{eq: F 2 upper bound}
\Psi^{0,s}_{\mathcal R(L,\delta L),\beta}[F^2_x(\mathsf{a},\omega) ]\leq \Psi^{1}_{\mathcal{B}(L,\varepsilon,x),\beta,\mathsf{h}(x;s)}[F^2_x(\mathsf{a},\omega)].
\end{equation}
Thus, by the triangle inequality and the fact that $\Psi^{0,s}_{\mathcal R(L,\delta L),\beta}[F^1_x(\mathsf{a},\omega) \mid  \mathcal{D}]\leq \Psi^{0,s}_{\mathcal R(L,\delta L),\beta}[F^1_x(\mathsf{a},\omega)]$, $\Psi^{0,s}_{\mathcal R(L,\delta L),\beta}[F^2_x(\mathsf{a},\omega) \mid  \mathcal{D}]\leq \Psi^{0,s}_{\mathcal R(L,\delta L),\beta}[F^2_x(\mathsf{a},\omega)]$, thanks to the FKG inequality,
\begin{multline}
\Big|\Psi^{0,s}_{\mathcal R(L,\delta L),\beta}\left[F_x(\mathsf{a},\omega)\mid  \mathcal{D}\right]-\Psi^{0,s}_{\mathcal R(L,\delta L),\beta}\left[F_x(\mathsf{a},\omega) \right]\Big| \leq \\ \Psi^{1}_{\mathcal{B}(L,\varepsilon,x),\beta,\mathsf{h}(x;s)}\left[G_x(\mathsf{a},\omega)\right] -\Psi^{0}_{\mathcal{B}(L,\varepsilon,x),\beta,\mathsf{h}(x;s)}\left[G_x(\mathsf{a},\omega)\right],  
\end{multline}
where $G_x=F^1_x+F^2_x$.
By our assumption, one has that $\Psi^{0,s}_{\mathbb H,\beta}=\Psi^{1,s}_{\mathbb H,\beta}$ for almost every $s\in (0,h]$. As a result, for almost every $s \in (0,h]$,
\begin{equation}\label{eq: summand goes to 0 far away from the boundary}
   \Psi^{1}_{\mathcal{B}(L,\varepsilon,x),\beta,\mathsf{h}(x;s)}\left[G_x(\mathsf{a},\omega)\right]-\Psi^{0}_{\mathcal{B}(L,\varepsilon,x),\beta,\mathsf{h}(x;s)}\left[G_x(\mathsf{a},\omega)\right]=o_s(1),
\end{equation}
where the term $o_s(1)$ tends to $0$ as $L$ tends to infinity uniformly over $x\in \partial^{\rm bot}\mathcal{R}(L,\delta L)$ such that $|x|\leq L-\varepsilon L$, and by regularity we have that $|o_s(1)|\leq C$ uniformly over $s\in (0,h]$.

Combining the above, we obtain that for almost every $s\in (0,h]$
\begin{equation}
\Bigg|\frac{\partial \log \Psi^{0,s}_{\mathcal R(L,\delta L),\beta}[ \mathcal{D} ]}{\partial s}\Bigg|\leq C_1 \varepsilon L^{d-1}+ \sum_{{\rm Bulk}(\varepsilon,L)} o_s(1).
\end{equation}
Integrating in $s$ and using the dominated convergence theorem, we obtain that there exists $C_2>0$ such that for $L$ large enough,
\begin{equation}
\Big|\log \Psi^{0,h}_{\mathcal R(L,\delta L),\beta}[ \mathcal{D} ]- \log \Psi^{0}_{\mathcal R(L,\delta L),\beta}[ \mathcal{D} ]\Big|\leq C_2\varepsilon L^{d-1}.  
\end{equation}
The proof follows from choosing $\varepsilon$ small enough in such a way that $C_2\varepsilon \leq c_2/2$, and from fixing $c_3$ small enough to accommodate the small values of $L$.
\end{proof}

\subsection{Proof of supercritical sharpness}
We now use Lemma \ref{lem: large to free bc} to prove Theorem~\ref{thm:sharpness}. The proof is the same as that of \cite[Theorem 5.1]{gunaratnam2025supercritical}. We include it for the reader's convenience.

\begin{proof}[Proof of Theorem~\textup{\ref{thm:sharpness}}]
By Lemma \ref{lem: large to free bc} there exist $\delta,c_3>0$ such that for every $L\geq 1$,
\begin{equation}\label{eq:surface_free}
\Psi^0_{\mathcal R(L,\delta L),\beta}[ \partial^{\mathrm{bot}} \mathcal R(L, \delta L/2) \centernot\longleftrightarrow \partial^{\mathrm {top}} \mathcal R(L, \delta L/2)] \leq e^{-c_3 L^{d-1}}.
\end{equation} Let $\ell\leq \delta L/2$. Note that we can cover the hyperplane $\{-L,\ldots,L\}^{d-1}\times \{0\}$ by $m\leq C({L}/{\ell})^{d-1}$ boxes $\Lambda_\ell(x_1),\ldots,\Lambda_\ell(x_m)$ (all included in $\mathcal{R}(L,\delta L)$) with $x_1,\ldots,x_m\in\{-L,\ldots,L\}^{d-1}\times \{0\}$. Fix such a choice and note that
\begin{equation}
\bigcap_{i=1}^{m} \{\Lambda_\ell(x_i)\centernot\longleftrightarrow\partial \Lambda_{\delta L/2}(x_i)\} \subset \{\partial^{\rm bot} \mathcal{R}(L, \delta L/2)\centernot\longleftrightarrow\partial^{\rm top} \mathcal{R}(L, \delta L/2)\}.
\end{equation}
Therefore, by the FKG inequality together with \eqref{eq:surface_free}, there exists $i_0\in\{1,\ldots,m\}$ such that 
\begin{equation} \label{eq: diff g}
\Psi^0_{\mathcal{R}(L,\delta L),\beta}[\Lambda_\ell(x_{i_0})\centernot\longleftrightarrow\partial \Lambda_{\delta L/2}(x_{i_0})]\leq e^{-c_3 \ell^{d-1}/C}.
\end{equation}
By the monotonicity of the free measure in the volume we obtain 
\begin{equation}
\Psi^0_{\Lambda_{2L}(x_{i_0}),\beta}[\Lambda_\ell(x_{i_0})\centernot\longleftrightarrow\partial \Lambda_{\delta L/2}(x_{i_0})]\leq \Psi^0_{\mathcal{R}(L,\delta L),\beta}[\Lambda_\ell(x_{i_0})\centernot\longleftrightarrow\partial \Lambda_{\delta L/2}(x_{i_0})]\leq e^{-c_3 \ell^{d-1}/C}.
\end{equation}
This implies the desired result with $c= \min\{\delta/4, c_3/C\}$.
\end{proof}

\bibliographystyle{plain}
\bibliography{biblio.bib}
\end{document}